\documentclass[a4paper, 10pt, english, reqno]{amsart}

\usepackage[utf8]{inputenc}
\usepackage{fullpage}
\usepackage{graphicx}

\usepackage{amsmath, amssymb, amsfonts, amsthm, mathtools}
\usepackage{mathrsfs}

\usepackage{xcolor}
\usepackage{algorithm}
\usepackage{algpseudocode}
\usepackage{subcaption}
\usepackage{float}
\usepackage{thmtools}
\usepackage{empheq}

\usepackage{hyperref}
\usepackage{cleveref}

\DeclareMathOperator*{\Law}{\mathrm{Law}}
\DeclareMathOperator*{\Tr}{\mathrm{Tr}}

\def\R{\mathbb R}
\def\N{\mathbb N}

\def\d{\mathrm d}
\def\tcr{\textcolor{red}}

\newtheorem{theorem}{Theorem}[section]
\newtheorem{lemma}[theorem]{Lemma}

\newtheorem{definition}[theorem]{Definition}
\newtheorem{proposition}[theorem]{Proposition}
\newtheorem{remark}[theorem]{Remark}

\newcounter{Acond}
\renewcommand{\theAcond}{A\arabic{Acond}}
\crefname{Acond}{Assumption}{Assumptions}

\title{\textbf{Policy Gradient for Continuous-Time Mean-Field Control}}
\date{}

\author[E. Bayraktar]{Erhan Bayraktar\textsuperscript{1}}
\thanks{\textsuperscript{1} Department of Mathematics, University of Michigan, Ann Arbor, MI, USA}

\author[M. Hernandez]{Martin Hernandez\textsuperscript{2,*}}
\thanks{\textsuperscript{2} Department of Statistics and Data Science, University of California, Los Angeles, CA, USA}
\author[Q. Yan]{Qinxin Yan\textsuperscript{3}}
\thanks{\textsuperscript{3} Program in Applied and Computational Mathematics, Princeton University, Princeton, NJ, USA}

\author[Y. Zhu]{Yuhua Zhu\textsuperscript{2}}

\thanks{\textsuperscript{*} Corresponding author (\texttt{martinh@ucla.edu}).}

\date{}

\begin{document}
\maketitle

\begin{abstract}
This paper develops a policy gradient method for entropy-regularized mean-field control in the discounted infinite-horizon setting. We consider randomized feedback policies and a coupled representative-particle/population system, in which the representative state evolves jointly with a population law governed by a McKean--Vlasov equation. The resulting value function is therefore defined on the product space $\mathbb R^d \times \mathcal P_2(\mathbb R^d)$.

A key distinction from existing policy gradient methods for mean-field control is that, after computing the value function under a fixed policy, our approach does not require solving an additional equation to obtain the policy gradient. Instead, we derive an explicit policy gradient formula directly in terms of the value function. The formulation is based on an instantaneous advantage function, which quantifies the gain of taking a given action relative to the current randomized policy.  We establish a G\^ateaux policy-gradient formula, which gives the first-order variation of the objective along arbitrary policy perturbations, and then derive the corresponding ascent direction under finite-dimensional policy parametrization.

The resulting formula leads to a model-based actor--critic scheme. The critic is obtained by solving the associated linear stationary Hamilton--Jacobi--Bellman equation for the value function, using cylindrical functions to represent dependence on the population law. The actor is then updated according to the derived policy-gradient formula. We further analyze the well-posedness of the PDE in a polynomial-growth function class. Finally, we illustrate the proposed method through numerical experiments on an LQR model and a crowd-motion problem.
\end{abstract}

\tableofcontents

\section{Introduction}

Many control problems involve a large population of identical agents whose dynamics interact weakly, and only through the empirical distribution of the population. Such models arise naturally in applications including opinion formation \cite{T, hekr02}, biological swarming and flocking \cite{CS, cafotove10}, crowd dynamics \cite{CCH13}, and financial systemic risk \cite{CFS}. See also \cite{carmona2018probabilistic1, carmona2018probabilistic2, benbook} for  general background. When the number of agents $N$ is large, directly optimizing the full interacting $N$-particle system becomes prohibitively difficult, both computationally and analytically. The mean-field viewpoint then replaces the full microscopic system by a limiting problem in which a representative agent interacts with the distribution of the population. This limit is closely related to the theory of interacting particle systems, propagation of chaos, and McKean--Vlasov (MKV) stochastic differential equations, where under suitable assumptions, the empirical measure of the $N$-particle system converges to a measure-valued limit, and the state of a representative particle is then described by a stochastic differential equation whose coefficients depend on its own law; see, for instance, \cite{carmona2018probabilistic1,  Lacker}.

From this large-population perspective, two different limiting formulations emerge. If one starts from a \emph{central} control problem for the full population, or equivalently takes the mean-field limit of a socially optimal large-agent control problem, one is led to \emph{mean-field control} (MFC), which is also known as  McKean--Vlasov optimal control. If instead each agent optimizes an individual objective against the population and one
then passes to the limit, one obtains a \emph{mean-field game} (MFG), whose solution concept is Nash equilibrium. In this sense, MFC and MFG should be viewed as two distinct asymptotic regimes of large-population systems of exchangeable interacting agents. The key difference lies in whether the
 optimization is collective or strategic, and in the order in which optimization and passage to the limit are performed; see \cite{CarmDelLach, Lacker, lali07, carbook, benbook}. The present paper is concerned exclusively with the control formulation, namely MFC.

\smallbreak

A distinctive feature of the present framework is the formulation of 
the controlled dynamics through a \emph{decoupled system}: a 
MKV equation governing the evolution of the population law, 
coupled with an independent stochastic differential equation for a 
representative agent \cite{BuckdahnLiPengRainer2017, 
BayraktarCossoPham2018}. In this formulation, the population flow 
$(\mu_t)_{t \ge 0}$ is generated by the MKV equation and is 
deterministic, while the representative agent evolves according to its 
own SDE driven by an independent Brownian motion. A key advantage of 
this decoupled structure is that it allows the value function to be 
defined not only at the level of the population distribution, as is 
standard in the Wasserstein-space formulation \cite{PhamWei}, but also 
for any given initial position of the representative agent. The value 
function therefore takes as arguments both an initial state 
$s \in \mathbb{R}^d$ for the representative agent and an initial 
population law $\mu \in \mathcal{P}_2(\mathbb{R}^d)$, a joint 
dependence that plays a central role in the derivation of the policy 
gradient and in the design of the actor-critic algorithm.

\smallbreak

We work in the infinite-horizon discounted setting with entropy regularization and randomized feedback policies. This choice is natural for several reasons. First, entropy regularization promotes exploration and smooths the optimization over policies, a feature that is central in continuous-time reinforcement learning and remains analytically useful even
when the controlled dynamics are known \cite{wang2020, jia2022, jia2023}.
Second, the temperature parameter $\lambda>0$ connects the regularized problem to the classical deterministic formulation and yields a family of problems with better structural properties. Third, in the linear-quadratic setting, entropy regularization leads to explicit Gaussian optimal policies, which provide a tractable benchmark for the general theory. 

\smallbreak

The goal of this paper is to develop a model-based actor--critic method for mean-field control problems. Actor--critic methods are a standard class of algorithms in reinforcement learning, combining a critic, which evaluates the current policy, with an actor, which updates the policy \cite[Chapter 13]{SuttonBarto2018}. In our setting, for a fixed randomized feedback policy, the critic evaluates its performance through a HJB equation. The actor then improves the policy within a parametrized class, using a policy gradient formula for the performance functional, which is derived from a policy gradient theorem. We work in the model-based setting, that is, the dynamics are known. This allows us to focus
on the analytic structure of the problem and provides a foundation for
future data-driven extensions in \cite{secondpaper}.

\smallbreak

A key difficulty is the structure of the policy gradient. In classical single-agent control, the gradient of the performance functional with respect to the policy parameters depends only on the state process and the corresponding advantage function. In the mean-field setting, by contrast, the performance functional depends on the policy through both the representative dynamics and  the population flow induced by the same policy. As a result, the policy gradient theorem provided in this work decomposes into two structurally distinct contributions: a \emph{representative term}, capturing the direct effect of the policy through the state dynamics, and a \emph{population term}, capturing the indirect effect through the dependence of the value function on the population law. This decomposition has no counterpart in the standard single-agent theory and is one of the main contributions of the present paper.

\smallbreak

\paragraph{Contributions.}

\begin{enumerate}

  \item \textbf{Policy gradient theorem.}
  We derive a rigorous policy gradient theorem for the coupled
  representative agent-population system in the infinite-horizon setting. The
  resulting formula separates the representative and population
  contributions to the gradient of the performance functional,
  expressed in terms of two advantage functions: the representative
  advantage function, which extends the classical single-agent
  structure, and the population advantage function, which is specific
  to the mean-field framework and involves the derivative of the
  value function on probability measure space. We also establish a parametric version of this
  formula for families of randomized policies, providing a rigorous
  basis for gradient-based policy improvement.

 \item \textbf{Policy evaluation and infinite-horizon mean-field LQR.}
  We analyze the linear stationary Hamilton--Jacobi--Bellman equation on
  $\mathbb R^d \times \mathcal P_2(\mathbb R^d)$ arising in the policy
  evaluation step and establish classical well-posedness results. We also study
  the infinite-horizon entropy-regularized mean-field LQR problem,
  deriving an explicit characterization of the optimal value function
  and the associated optimal Gaussian policy. This LQR analysis provides
  both an explicit benchmark for the general theory and a natural
  reference setting for the numerical experiments.

  \item \textbf{Actor-critic algorithm and numerical experiments.}
  Based on the theoretical analysis, we introduce a finite-dimensional
  actor-critic algorithm combining a Galerkin approximation of the
  stationary HJB equation for policy evaluation with the policy
  gradient formula for policy improvement. Basis functions are
  constructed using cylindrical functions of the population law, for
  which the Lions derivatives entering the HJB operator can be
  evaluated in closed form. Numerical experiments on the mean-field
  systemic risk model and on a nonlinear two-dimensional crowd-aversion
  problem corroborate the theoretical results and demonstrate the
  practical effectiveness of the proposed method.

\end{enumerate}

\smallbreak
\noindent
\textbf{Related work.}

\smallbreak
\noindent

\textbf{Mean-field limits, propagation of chaos, and McKean--Vlasov
control.}
The mean-field description of large populations of exchangeable
interacting particles is classically tied to propagation of chaos and
McKean--Vlasov (MKV) dynamics; see, for instance,
\cite{SznitmanChaos,BayraktarWuZhangFBSDE,ChaintronDiez2022,
carmona2018probabilistic1}. In the uncontrolled setting, the connection
between MKV SDEs and associated PDEs was developed in
\cite{BuckdahnLiPengRainer2017}. In the controlled setting, rigorous
limit theory linking finite-agent control problems with
MKV control was established in \cite{Lacker} and further
developed in \cite{DjetePossamaiTan2022}. The distinction between the
cooperative mean-field control limit and the strategic mean-field game
limit was clarified in \cite{CarmDelLach}; see also
\cite{lali07,carbook,benbook,carmona2018probabilistic2} for general
background.

The theoretical foundations of MKV optimal control are well
developed, particularly in finite time horizon. The stochastic
Pontryagin maximum principle was derived in \cite{CarmDel15}; see also
\cite{Bonnet2017ThePM,anddje10,bucetal11} for dynamics depending on
moments of the law. The dynamic programming approach, leading to HJB
equations on the Wasserstein space, was initiated in
\cite{PhamWei,LaurPironn14}. Feynman--Kac representations for HJB
equations on Wasserstein spaces were obtained in
\cite{BayraktarCossoPham2018,chassagneux2019}. Viscosity-solution
results for MKV control were first studied in
\cite{GangboHamilton2008} for a problem from classical mechanics, and
were later developed through the Lions lifting in
\cite{PhamWei,PhamWei2018}. Related viscosity theories for equations on
the Wasserstein space include the Eikonal equation
\cite{SonerYanTorus,SonerYanEikonal}, comparison results for
second-order PDEs \cite{BayraktarEkrenZhangCPDE2025,
BayraktarEkrenHeZhangJDE2026}, and a fully second-order HJB theory
\cite{BayraktarEkrenCheungQiuTaiZhangSICON}. We also refer to
\cite{BurzoniIgnazioReppenSoner2020} for viscosity solutions in the
setting of controlled MKV jump-diffusions. Infinite-horizon
problems on Wasserstein spaces have been studied, for instance, in
\cite{HyndKil2015}, although most of the above theory is formulated in
finite time horizon.

Our analysis is complementary to this literature. Motivated by policy iteration (see
\cite{Howard1960,SuttonBarto2018}), we
focus on the linear stationary HJB equation associated with a fixed
randomized policy. While most of the above works study
HJB equations on $\mathcal P_2(\mathbb R^d)$, our coupled system leads to a linear stationary HJB on $\mathbb R^d\times\mathcal P_2(\mathbb R^d)$, for which we establish existence and uniqueness of
classical solutions.

\smallbreak

\noindent
\textbf{Numerical methods for mean-field control and policy gradient.} A broad range of numerical methods has been developed for mean-field
control problems with known dynamics. PDE-based and augmented-Lagrangian
methods have been proposed for mean-field type control problems with
congestion effects \cite{AchdouLauriereCongestionI,AchdouLauriereCongestionII},
while MKV control problems have also been approached through
numerical schemes for the associated forward-backward stochastic
differential equations \cite{chassagneux2019}. We refer
to \cite{lauriere2021} for a comprehensive survey on FBSDE for MKV, and dynamic programming based methods, including particle
methods, and PDE-based discretizations. Related particle
approximations for second-order PDEs on the Wasserstein space have been
studied in \cite{BayraktarEkrenZhangParticleSICON2025}. In high-dimensional
regimes, these methods are complemented by a growing literature based on
deep learning and neural-network parametrizations; see, for instance,
\cite{ruthotto2020,Carmona2022,soner2025learning}.

A related approximation paradigm, in the context of actor-critic, is linear value-function approximation,
which is classical in reinforcement learning and is widely used to make
policy evaluation and policy iteration computationally feasible in large
or continuous state spaces; see, for instance,
\cite{Szepesvari2010Algorithms,JinYangWangJordan2020,
AyoubJiaSzepesvariWangYang2020,YinHaoAbbasiYadkoriLazicSzepesvari2022}.
In the mean-field-control setting, \cite{BayraktarKara2025LFA} develops
learning algorithms with linear function approximation for finite-state,
finite-action MFC problems and proves quantitative performance bounds.  In the linear-quadratic
setting, convergence and error analyses of policy-gradient methods have
been studied in both discrete and continuous time
\cite{carmona2019linear, wang2021global, frikha2024full}. 
The present paper is complementary to these works. Methods developed for
discrete-time MDPs or mean-field MDPs do not transfer directly to the
continuous-time diffusion setting. Indeed, the conventional
state-action value function collapses to the value function as the
sampling interval tends to zero, so that its action dependence survives
only at the infinitesimal level through an instantaneous advantage
rate; see \cite{tallec2019, jia_qlearning2023}. This motivates formulations of actor-critic methods that are
intrinsic to the continuous-time dynamics. 

We would like to mention two lines of work that is most related to this paper (also see Remarks \ref{rmk:adv}, \ref{rmk:key diff} and \ref{rmk:para-gd} for more detailed discussion). \cite{frikhaActorCritic} derived a policy gradient formula for MFC in
finite horizon by differentiating the HJB characterization of the
value function with respect to the policy parameter, and
more recently, they generalized this algorithm to moment neural network approximations \cite{pham2025actor}.  
Our work is closest in spirit to the above two works. Differentiating the HJB provides a
natural route for parametric actor--critic methods, but it also requires
solving an additional sensitivity equation, or equivalently an additional
critic-type problem, to compute the ascent direction.
In contrast, our approach derives an explicit policy-gradient formula
directly in terms of the value function. This avoids the need for a separate
sensitivity equation and leads to a simpler actor update. In addition, our
formulation is first developed at the level of randomized feedback policies
themselves. We characterize the G\^ateaux policy gradient, namely the
first-order variation of the objective along arbitrary admissible policy
perturbations, before specializing it to finite-dimensional policy
parametrizations. This gives a parameterization-free description of policy
improvement and then yields an explicit ascent direction for parametrized
policies. Finally, the policy gradient formulation is expressed through two instantaneous advantage functions, one for the representative-agent dynamics and one for the population-law dynamics. These advantages are defined from the value function and the infinitesimal generators, and do not require derivatives of the drift or diffusion with respect to the policy parameters. This form is closer to the classical policy-gradient structure in reinforcement learning, where the policy gradient is expressed in terms of an advantage function.

On the other hand, \cite{Wei2023ContinuousTQ} also derives a related instantaneous advantage q-function for mean-field control. Their formulation is population-based,
whereas ours is built on the coupled representative-agent/population system.
More importantly, the advantage function plays a different role. \cite{Wei2023ContinuousTQ} uses the instantaneous
advantage to construct the integrated $q$-functions for all test policies in a neighbourhood of the target policy. Then the martingale property of the integrated q-functions is used to solve the MFC. {This integrated $q$-function approach has been further developed in \cite{WeiYuYuanUnified} for MFG and MFC problems, including jump-diffusion dynamics. More recent works extend this continuous-time $q$-learning framework to MFC with controlled common noise \cite{RenWeiYuZhouCommonNoiseI,RenWeiYuZhouCommonNoiseII}. In that setting, the conditional population law is stochastic and the policy-improvement step becomes more implicit, leading to fixed-point characterizations of the optimal policy. In contrast with these papers,} we use the pointwise-in-action $q$-function directly
to derive the policy-gradient formula, yielding an explicit actor update once
$V^\pi$ and its derivatives are available. Thus, our actor--critic pipeline
does not require an auxiliary integrated $q$-function or testing over a
neighborhood of policies.


\smallbreak

\noindent
\textbf{Mean-field linear-quadratic control.}
The linear-quadratic structure has played a central role in mean-field
control theory. In the finite-horizon setting, Riccati-based
characterizations for mean-field stochastic linear-quadratic control
were developed in \cite{Yong2013MFSLQ}, while the infinite-horizon
problem with deterministic coefficients was studied in
\cite{HuangLQRmeanfield2015}. Subsequent works investigated the
structural solvability of mean-field linear-quadratic problems,
including closed-loop solvability \cite{LiSunYong2016MFCL}, open-loop
solvability \cite{Sun2017MFOpenLoop}, and weak closed-loop solvability
\cite{SunWang2021MFWeakCL}. Extensions to conditional
McKean--Vlasov dynamics with random coefficients and common noise were
obtained in \cite{Pham2016LQCMKV,BaseiPham2019WeakMartingale}.  \cite{frikhaActorCritic}  analyze the entropy regularized finite-time mean-field LQR. More
recently, long-time asymptotic properties, including turnpike behavior,
have been studied in \cite{BayraktarJianTurnpike,
BayraktarJianCommonNoiseTurnpike}. We complement this literature by
analyzing an entropy-regularized infinite-horizon mean-field LQR problem
with stationary randomized policies, for which the optimal Gaussian
policy and the optimal value function are characterized explicitly
through a coupled algebraic Riccati system.

\smallbreak

\section{Problem Formulation}\label{sec:formulation}

Let $d,n,m\in\mathbb N$. The state space is $\mathbb R^d$, the action
space is a Borel set $\mathcal A\subseteq \mathbb R^m$, and the
population state is described by a measure in
$\mathcal P_2(\mathbb R^d)$, the set of Borel probability measures on
$\mathbb R^d$ with finite second moment. We equip
$\mathcal P_2(\mathbb R^d)$ with the $2$-Wasserstein distance
\[
  \mathcal W_2(\mu,\varrho)
  :=
  \inf_{\gamma\in\Gamma(\mu,\varrho)}
  \left(
    \int_{\mathbb R^d\times\mathbb R^d}
    |x-y|^2\,\gamma(dx,dy)
  \right)^{1/2}.
\]

We consider an infinite-horizon discounted mean-field optimal control
problem. Starting from an initial condition
$(s,\mu)\in\mathbb R^d\times\mathcal P_2(\mathbb R^d)$, where $s$
denotes the state of a representative agent and $\mu$ the distribution
of the population, a social planner seeks to optimize the collective
performance of the system over time. The objective is to maximize the
expected discounted reward generated by the controlled mean-field
dynamics. This leads to the optimal value function over an admissible set $\Pi_{\mathrm{add}}$
\begin{align}\label{eq:optimal_CP}
  V^{*}(s,\mu)
  \,:=\,
  \sup_{\pi\in \Pi_{\mathrm{add}}} V^\pi(s,\mu),
  \qquad
  (s,\mu)\in\mathbb R^d\times\mathcal P_2(\mathbb R^d).
\end{align}
Here $V^\pi(s,\mu)$ denotes the value associated with a policy
$\pi\in\Pi_{\mathrm{add}}$ when the system starts from $(s,\mu)$.
Whenever there exists a policy $\pi^*\in \Pi_{\mathrm{add}}$ such that
$V^*(s,\mu)=V^{\pi^*}(s,\mu)$ for all
$(s,\mu)\in\mathbb R^d\times\mathcal P_2(\mathbb R^d)$, we call
$\pi^*$ an optimal policy.

More precisely, for a given policy $\pi\in \Pi_{\mathrm{add}}$, the
corresponding value is defined by
\begin{align}\label{eq:value_fn}
  V^{\pi}(s,\mu)
  \,:=\,
  \mathbb E\!\left[
    \int_0^\infty e^{-\beta t}\,
    r_\lambda^\pi(s_t,\mu_t)\,\d t
    \;\bigg|\;
    s_0=s,\ \mu_0=\mu
  \right],
\end{align}
where $\beta>0$ is the discount factor, $r_\lambda^\pi$ is the running
reward, and $(s_t,\mu_t)_{t\ge0}$ denotes the controlled
representative-population state process. In classical mean-field
control, one often works with deterministic feedback controls. In the
present work, by contrast, we allow for randomized feedback policies.
This choice is natural in reinforcement learning, where exploration
plays a central role and is naturally encoded through Markov kernels.
Accordingly, we take $\Pi_{\mathrm{add}}$ to be the class of Markov
kernels
\[
  \pi:\mathbb R^d\times\mathcal P_2(\mathbb R^d)\longrightarrow
  \mathcal P(\mathcal A),
  \qquad
  (x,\mu)\longmapsto \pi(\cdot|x,\mu),
\]
such that $\pi(\cdot|x,\mu)\ll \nu$ for every $(x,\mu)$, with jointly
measurable density $p_\pi(\cdot|x,\mu)$, where $\nu$ denotes a
$\sigma$-finite reference measure on
$(\mathcal A,\mathcal B(\mathcal A))$. The running reward in
\eqref{eq:value_fn} is the entropy-regularized averaged reward
\begin{equation}\label{eq:reg_reward_avg}
  \begin{aligned}
    r_\lambda^\pi(x,\mu)
    &:= r^\pi(x,\mu)+\lambda\,\mathcal H\bigl(\pi(\cdot|x,\mu)\bigr)\\
    &:= \int_{\mathcal A} r(x,\mu,a)\,\pi(\d a|x,\mu)
    -\lambda\int_{\mathcal A}
    \log p_\pi(a|x,\mu)\,\pi(\d a|x,\mu),
  \end{aligned}
\end{equation}
where $\lambda>0$ is the temperature parameter. Entropy regularization
is standard in continuous-time reinforcement learning and promotes
exploration; see, for instance,
\cite{wang2020, jia2022, jia2023, jia_qlearning2023,
Wei2023ContinuousTQ, frikhaActorCritic}.

We now introduce the controlled dynamics associated with a policy
$\pi\in \Pi_{\mathrm{add}}$. Let
$b:\mathbb R^d\times\mathcal P_2(\mathbb R^d)\times\mathcal A
\to\mathbb R^d$
and
$\sigma:\mathbb R^d\times\mathcal P_2(\mathbb R^d)\times\mathcal A
\to\mathbb R^{d\times n}$
be measurable coefficients. The averaged drift and diffusion matrix
induced by $\pi$ are defined by
\begin{align*}
  b^\pi(x,\mu)
  := \int_{\mathcal A} b(x,\mu,a)\,\pi(\d a|x,\mu),
  \qquad
  \Sigma^\pi(x,\mu)
  := \int_{\mathcal A}
  \sigma(x,\mu,a)\sigma(x,\mu,a)^\top\,\pi(\d a|x,\mu),
\end{align*}
and we write
$\sigma^\pi(x,\mu):=\bigl(\Sigma^\pi(x,\mu)\bigr)^{1/2}
\in\mathbb R^{d\times d}$ for the symmetric positive semidefinite
square root of $\Sigma^\pi$. The representative state process then
solves
\begin{align}\label{eq:rep_SDE}
  \d s_t
  = b^\pi(s_t,\mu_t)\,\d t
  + \sigma^\pi(s_t,\mu_t)\,\d B_t,
  \qquad
  s_0=s,
\end{align}
while the population distribution $\mu_t$ is generated by the McKean-Vlasov
equation
\begin{align}\label{eq:MKV_SDE}
  \d\tilde s_t
  = b^\pi(\tilde s_t,\mu_t)\,\d t
  + \sigma^\pi(\tilde s_t,\mu_t)\,\d\tilde B_t,
  \qquad
  \tilde s_0\sim\mu,
  \qquad
  \mu_t=\Law(\tilde s_t),
\end{align}
where $B$ and $\tilde B$ are independent $d$-dimensional Brownian
motions. The pair $(s_t,\mu_t)$ appearing in \eqref{eq:value_fn} is
therefore obtained by coupling a representative agent with the
population law generated by \eqref{eq:MKV_SDE}. 
\smallbreak

Formally, for each initial condition $(s,\mu)\in\mathbb R^d\times\mathcal P_2(\mathbb R^d)$ and each horizon $T>0$, we work on a probability space $(\Omega,\mathcal F,\mathbb P)$ endowed with a filtration $\mathbb F=(\mathcal F_t)_{0\le t\le T}$ such that $\mathcal F_0$ contains all $\mathbb P$-null sets of $\mathcal F$, the filtration is right-continuous, and $(\Omega,\mathcal F,\mathbb F,\mathbb P)$ supports two independent $d$-dimensional $\mathbb F$-Brownian motions $B$ and $\tilde B$, together with an $\mathcal F_0$-measurable random variable $\tilde s_0$ satisfying $\Law(\tilde s_0)=\mu$ and independent of $(B,\tilde B)$. The representative component starts from the deterministic value $s$. Under Assumption \ref{A1} (below), the population equation admits a unique strong $\mathbb F$-adapted solution $\tilde s$, and we then define $\mu_t:=\Law(\tilde s_t)$; in particular, $(\mu_t)_{0\le t\le T}$ is a deterministic flow in $\mathcal P_2(\mathbb R^d)$; see
\cite[Chapter 4]{carbook}. Once this flow is fixed, due to Assumption \ref{A1}, the representative dynamics is an $\mathbb F$-adapted stochastic differential equation, and therefore it also admits a unique strong solution.

\begin{remark}[On randomized and aggregated formulations]
\label{rem:equiv}
The state--law formulation adopted here, namely a representative state
coupled with a population law flow on
$\mathbb R^d\times\mathcal P_2(\mathbb R^d)$, is consistent with the
general McKean--Vlasov framework of \cite{BuckdahnLiPengRainer2017} in
the uncontrolled case and \cite{BayraktarCossoPham2018} in the control
setting.

Given a stochastic policy $ \pi:\mathbb R^d\times\mathcal P_2(\mathbb R^d)\to \mathcal P(\mathcal A),$ the most direct non-averaged formulation is to sample actions continuously in time according to the current state and population law, namely
\begin{align}
  &\d s_t
  = b(s_t,\mu_t,a_t^s)\,\d t
     + \sigma(s_t,\mu_t,a_t^s)\,\d B_t, \quad  a_t^s \sim \pi(\cdot|s_t,\mu_t),\label{eq:SDE_MKV_action_a}\\
  &\d\tilde s_t
  = b(\tilde s_t,\mu_t,a_t^\mu)\,\d t
     + \sigma(\tilde s_t,\mu_t,a_t^\mu)\,\d\tilde B_t,\quad  a_t^\mu \sim \pi(\cdot|\tilde s_t,\mu_t),\quad \mu_t=\Law(\tilde s_t)\label{eq:SDE_MKV_action_a2}
\end{align}
Formally, one may realize the sampling through a measurable map
$\Phi_\pi(x,\mu,u)$ such that $\Phi_\pi(x,\mu,U)\sim\pi(\cdot|x,\mu)$
for $U\sim{\rm Unif}[0,1]$, and then set
$a_t^s=\Phi_\pi(s_t,\mu_t,U_t^s)$,
$a_t^\mu=\Phi_\pi(\tilde s_t,\mu_t,U_t^\mu)$. The difficulty is that, on
the usual product space, a continuum of independent samples in time need
not be jointly measurable in $(t,\omega)$, so the coefficients may fail
to be progressively measurable; see \cite[Remark 2.1]{SzpruchTZ2024}. A rigorous pathwise construction
is nevertheless possible on a suitable enriched product space, for
instance a rich Fubini extension in the sense of \cite{Sun2006}; this is
precisely the route taken in \cite{frikhaActorCritic} for \eqref{eq:SDE_MKV_action_a}--\eqref{eq:SDE_MKV_action_a2}.

A different non-averaged interpretation is to sample the action only on a
time grid and then keep it constant between intervention times. This
produces an SDE with piecewise constant random controls on the usual
probability space. Related piecewise constant randomizations also appear
in \cite{BayraktarCossoPham2018} within the randomized McKean--Vlasov
control problem, while in the continuous-time reinforcement learning
literature such sampled dynamics are studied in
\cite{SzpruchTZ2024,JiaOuyangZhang2025}. In general, discretely sampled dynamics do not coincide with the
aggregated dynamics; rather, they converge to them as the mesh size
tends to zero; see \cite[Theorems 4.1--4.2]{JiaOuyangZhang2025}.

By contrast, the aggregated formulation
\eqref{eq:rep_SDE}--\eqref{eq:MKV_SDE} is the appropriate law-level
representation of the continuously randomized model
\eqref{eq:SDE_MKV_action_a}-\eqref{eq:SDE_MKV_action_a2}. More precisely, whenever
\eqref{eq:SDE_MKV_action_a}-\eqref{eq:SDE_MKV_action_a2} is constructed on an enriched space and is
well posed, the law of its state process solves the same
McKean--Vlasov martingale problem as the aggregated system \eqref{eq:rep_SDE}--\eqref{eq:MKV_SDE}. Hence,
under weak uniqueness, both formulations induce the same law and
therefore the same value functional; see
\cite[Lemma A.1]{Wei2023ContinuousTQ},
\cite{KarouiMeleard1990,Lacker} and \cite[Proof of Prop. 2.1]{frikhaActorCritic}.
\end{remark}

\paragraph{Organization of the paper.}  The rest of the paper is devoted to approximating the optimal policy of
\eqref{eq:optimal_CP} through an actor--critic iteration. Given a
parametrized randomized policy $\pi_\omega$, each iteration consists of
two steps:
\begin{enumerate}
    \item \emph{Policy evaluation.} For the current policy
    $\pi_\omega$, the critic computes or approximates the value function
    $V^{\pi_\omega}(s,\mu)$ by solving the associated linear stationary
    HJB equation.

    \item  \emph{Policy improvement.} The actor updates the policy parameters
using a policy gradient formula for the performance functional
$J(\omega)=\mathbb E[V^{\pi_\omega}(s,\mu)]$, for instance through the
gradient ascent step
$\omega\leftarrow \omega+\alpha\nabla_\omega J(\omega)$.
\end{enumerate}
In Section \ref{sec:HJB}, we study the policy-evaluation step and show
that, for a fixed admissible policy, the value function is the unique
classical solution of the associated stationary Hamilton--Jacobi--Bellman
equation. In Section \ref{sec:policy-iteration}, we derive the policy
gradient theorem and its parametric version for the coupled
representative agent--population system. In Section \ref{sec:policy-iteration-algorithm}, we describe how the
actor--critic method can be implemented in practice in the model-based
setting. Section \ref{sec:LQR} analyzes the entropy-regularized
infinite-horizon mean-field linear-quadratic regulator, which provides
an explicit benchmark for the general theory. Finally, Section \ref{sec:numerics} presents
numerical examples illustrating our method.


\section{The Stationary HJB Equation}
\label{sec:HJB}

The central object of study in this section is the value function
$V^\pi$ defined in \eqref{eq:value_fn}. Our goal is to characterize
it as the unique classical solution of a stationary Hamilton-Jacobi-Bellman
(HJB) equation.  

\subsection{The infinitesimal generator and stationary HJB equation}
\label{subsec:Lions}

The dynamics of the pair $(s_t, \mu_t)$ under
policy $\pi$ are governed by \eqref{eq:rep_SDE}-\eqref{eq:MKV_SDE}. To differentiate a function $F: \mathbb{R}^d \times
\mathcal{P}_2(\mathbb{R}^d) \to \mathbb{R}$ with respect to the
measure argument $\mu$, we rely on the notion of \emph{Lions
derivative}, introduced by Lions in his lectures at
the Collège de France \cite{lio12} and further developed
in \cite{carmona2018probabilistic1, carbook}.
Applying the Itô-Lions formula to a smooth function $F:
\mathbb{R}^d \times \mathcal{P}_2(\mathbb{R}^d) \to \mathbb{R}$ and
taking expectations, one identifies the infinitesimal generator of
$(s_t, \mu_t)$ as the operator
\begin{equation}
  \label{eq:L-pi}
  \begin{aligned}
    (\mathcal{L}_{b,\Sigma}^{\pi}F)(s,\mu)
    &:=
    b^{\pi}(s,\mu)\cdot\nabla_s F(s,\mu)
    + \tfrac{1}{2}\,\Sigma^{\pi}(s,\mu) : \nabla^2_{ss} F(s,\mu)\\
    &\quad + \int_{\mathbb{R}^d}
      \Bigl[
        b^{\pi}(\xi,\mu)\cdot \partial_{\mu}F(s,\mu)(\xi)
        + \tfrac{1}{2}\,\Sigma^{\pi}(\xi,\mu)
          : D_\xi\partial_{\mu}F(s,\mu)(\xi)
      \Bigr]\,\mu(\d\xi),
  \end{aligned}
\end{equation}
where $\xi \mapsto \partial_\mu F(s,\mu)(\xi)$ denotes the Lions
derivative of $F$ with respect to $\mu$, evaluated at
$(s,\mu,\xi)\in \mathbb{R}^d\times\mathcal{P}_2(\mathbb{R}^d)\times\mathbb{R}^d$.
We also denote by $D_\xi \partial_\mu F(s,\mu)(\xi)$ the Jacobian
matrix of $\xi \mapsto \partial_\mu F(s,\mu)(\xi)$ with respect to
$\xi \in \mathbb{R}^d$. For a
comprehensive treatment we refer to \cite[Chapter 5]{carmona2018probabilistic1}
and \cite[Chapter 1]{carbook}. Also $A:B := \mathrm{Tr}(A^\top B)$ is the Frobenius inner product,
$\nabla_s F$ and $\nabla^2_{ss} F$ are the gradient and Hessian of $F$
with respect to $s \in \mathbb{R}^d$.
\smallbreak

The generator \eqref{eq:L-pi} has two parts: the first line is the
classical Itô generator for $s_t$ driven by $b^\pi$ and $\Sigma^\pi$, that is \eqref{eq:rep_SDE}, 
while the integral term captures the variation of $F$ as the measure
$\mu_t = \mathcal{L}(\tilde{s}_t)$ evolves under the McKean-Vlasov
equation \eqref{eq:MKV_SDE}.

\smallbreak

In order to simplify the exposition we introduce the following class of functions. 
\begin{definition}[$\mathcal C^{2,2}$ functions with polynomial growth]\label{def:C22poly}
  Let $k \in \mathbb{N}$. A map
  $F:\mathbb{R}^d\times\mathcal{P}_2(\mathbb{R}^d)\to\mathbb{R}^k$
  belongs to
  $\mathcal{C}^{2,2}(\mathbb{R}^d\times\mathcal{P}_2(\mathbb{R}^d))$
  if, for every $\mu\in\mathcal{P}_2(\mathbb{R}^d)$, the map
  $s\mapsto F(s,\mu)$ is twice continuously differentiable on
  $\mathbb{R}^d$, and, for every $s\in\mathbb{R}^d$, the map
  $\mu\mapsto F(s,\mu)$ is Lions differentiable on
  $\mathcal{P}_2(\mathbb{R}^d)$. Moreover, the maps
  $(s,\mu,\xi)\mapsto \partial_\mu F(s,\mu)(\xi)$ and
  $(s,\mu,\xi)\mapsto D_\xi\partial_\mu F(s,\mu)(\xi)$ are jointly
  continuous.

  We say that $F\in\mathcal{C}^{2,2}_{\mathrm{poly}}(\mathbb{R}^d\times\mathcal{P}_2(\mathbb{R}^d))$
  if
  $F\in\mathcal{C}^{2,2}(\mathbb{R}^d\times\mathcal{P}_2(\mathbb{R}^d))$
  and there exists a constant $C_F>0$ such that, for all $(s,\mu,\xi)\in\mathbb{R}^d\times\mathcal{P}_2(\mathbb{R}^d)\times\mathbb{R}^d$,
  \begin{align*}
    |F(s,\mu)|
    &\leq C_F\bigl(1+|s|^2+m_2(\mu)\bigr),\\
    \|\nabla_s F(s,\mu)\| + |\partial_\mu F(s,\mu)(\xi)|
    &\leq C_F\bigl(1+|s|+|\xi|+m_2(\mu)\bigr),\\
    \|D^2_{ss}F(s,\mu)\| + \|D_\xi\partial_\mu F(s,\mu)(\xi)\|
    &\leq C_F\bigl(1+m_2(\mu)\bigr),
  \end{align*}
where $m_p(\mu)$ denotes the $p-$th moment of $\mu$.
\end{definition}

To establish $\mathcal{C}^{2,2}_{\mathrm{poly}}$
regularity of $V^\pi$ and the existence of a solution of the HJB equation in the
classical sense, we additionally require the following smoothness on
the policy-aggregated coefficients.

\medbreak

\refstepcounter{Acond}\label{A1}%
\noindent\textbf{(\theAcond) $\mathcal{C}^{2,2}$-regularity of averaged coefficients.}
Let $\pi \in \Pi_{\mathrm{add}}$ be fixed. We assume that, for every component
$h\in\{b_i^\pi,\sigma_{i,j}^\pi:i,j=1,\dots,d\}$, the derivatives
$\nabla_s h(s,\mu)$, $D^2_{ss}h(s,\mu)$, $\partial_\mu h(s,\mu)(\xi)$,
and $D_\xi\partial_\mu h(s,\mu)(\xi)$ exist for all
$(s,\mu,\xi)\in\mathbb{R}^d\times\mathcal{P}_2(\mathbb{R}^d)\times\mathbb{R}^d$,
are bounded by a constant $K_\pi$, and are locally Lipschitz continuous in $(s,\mu,\xi)$.
Also assume that $r_\lambda^\pi \in
  \mathcal{C}^{2,2}_{\mathrm{poly}}
  (\mathbb{R}^d\times\mathcal{P}_2(\mathbb{R}^d)).$

\begin{remark}\label{rem:linear_growth}
  Using \Cref{A1} we can assume for simplicity that $K_\pi > 0$ is big enough such that for all
  $s \in \mathbb{R}^d$ and $\mu \in \mathcal{P}_2(\mathbb{R}^d)$,
  \begin{align}\label{eq:linear_growth_bpi}
    |b^\pi(s,\mu)| + \|\sigma^\pi(s,\mu)\|
    \leq K_\pi\bigl(1 + |s| + \sqrt{m_2(\mu)}\bigr).
  \end{align}
  Also, since the bounded derivatives in \Cref{A1} ensure the existence of $L_\pi > 0$ such that for all $s, s'\in\mathbb{R}^d$ and
$\mu, \mu'\in\mathcal{P}_2(\mathbb{R}^d)$,
\begin{align*}
  |b^\pi(s,\mu)-b^\pi(s',\mu')|
  +
  \|\sigma^\pi(s,\mu)-\sigma^\pi(s',\mu')\|
  \le
  L_\pi\bigl(|s-s'|+\mathcal{W}_2(\mu,\mu')\bigr).
\end{align*}
  Moreover, since $\Sigma^\pi=\sigma^\pi(\sigma^\pi)^\top$, it follows from
  \Cref{A1} and the product rule that
  $\Sigma^\pi \in \mathcal{C}^{2,2}(\mathbb{R}^d\times\mathcal{P}_2(\mathbb{R}^d))$,
  and there exists $C_\pi>0$ such that, for all $(s,\mu,\xi)\in\mathbb{R}^d\times\mathcal{P}_2(\mathbb{R}^d)\times\mathbb{R}^d$,
\begin{equation}\label{eq:bounds_Simga_pi}
    \begin{aligned}
         \|\Sigma^\pi(s,\mu)\|
    \leq C_\pi\bigl(1+|s|^2+m_2(\mu)\bigr),\\
    \|\nabla_s\Sigma^\pi(s,\mu)\|
    + |\partial_\mu\Sigma^\pi(s,\mu)(\xi)|
    + \|D^2_{ss}\Sigma^\pi(s,\mu)\|
    + \|D_\xi\partial_\mu\Sigma^\pi(s,\mu)(\xi)\|
    &\leq C_\pi\bigl(1+|s|+\sqrt{m_2(\mu)}\bigr).
    \end{aligned}
\end{equation}
\end{remark}

We can now state the main result of this section, which characterizes the infinite-horizon value function as the unique classical solution of the stationary Hamilton-Jacobi-Bellman equation.

\begin{theorem}[Hamilton-Jacobi-Bellman equation]\label{thm:MF_eval_infinite} 
  Let $\pi \in \Pi_{\mathrm{add}}$ be a fixed admissible policy and suppose
  Assumption \ref{A1} holds. Then there exists a constant
  $\bar{\beta}_\pi>0$, depending only on the structural constants in
  Assumption \ref{A1}, such that the following holds. If $\beta>\bar{\beta}_\pi$, then the value function $V^\pi$ defined in \eqref{eq:value_fn} is well defined and belongs to
  $\mathcal{C}^{2,2}_{\mathrm{poly}}
  (\mathbb{R}^d\times\mathcal{P}_2(\mathbb{R}^d))$.
  Moreover, $V^\pi$ is a classical solution of
  \begin{align}\label{eq:HJB-eval}
    (\mathcal{L}_{b,\Sigma}^{\pi}-\beta)V^\pi(s,\mu)
    +r_\lambda^\pi(s,\mu)=0,
    \qquad
    (s,\mu)\in\mathbb{R}^d\times\mathcal{P}_2(\mathbb{R}^d),
  \end{align}
  where $\mathcal{L}_{b,\Sigma}^{\pi}$ is given by \eqref{eq:L-pi}.
  Finally, $V^\pi$ is the unique classical solution of
  \eqref{eq:HJB-eval} in the class
  $\mathcal{C}^{2,2}_{\mathrm{poly}}
  (\mathbb{R}^d\times\mathcal{P}_2(\mathbb{R}^d))$. Moreover, assume in addition  for all $(s,\mu)\in\mathbb R^d\times\mathcal P_2(\mathbb R^d)$
\begin{align}\label{eq:dissipativity_remark}
2\,s\cdot b^\pi(s,\mu)+\|\sigma^\pi(s,\mu)\|_F^2
\le C_0-\kappa |s|^2 + C_{\mathrm{diss}}\,m_2(\mu)
\end{align}
holds. Then the value function
$V^\pi$ is well defined for every $\beta>0$ and there exists $C_r>0$ such that $|r_\lambda^\pi(s,\mu)|\le C_r(1+|s|^2+m_2(\mu))$. Furthermore, if $C_0=0$ 
then $V^\pi$ is also well defined for $\beta=0$.
\end{theorem}

The proof of \Cref{thm:MF_eval_infinite} is provided in \Cref{APP:PROOFS_SEC_HJB}. Several remarks on the above theorem follows. 

First, the constant $\bar\beta_\pi$ in \Cref{thm:MF_eval_infinite} is defined by $\bar\beta_\pi=\max\{\beta_0(2),\beta_{\mathrm{var}}(2)\}$, where $\beta_0(2)$ and $\beta_{\mathrm{var}}(2)$ are given explicitly in \Cref{remark_explicit_beta_02,remark_explicit_beta_var02}, and depend only on $d$ and $K_\pi$.

Second, once the policy $\pi$ is fixed, the stationary HJB equation becomes linear. The proof of \Cref{thm:LQ_infinite} then relies on two preliminary lemmas. The first one provides explicit exponential-in-time $L^p$ bounds for the state processes, while the second one gives exponential-in-time bounds for the derivatives with respect to the state and measure variables. These estimates are the key input in the infinite-horizon argument, since they quantify precisely how the growth in time interacts with the discount factor. Starting from these identities, the proof combines the finite-horizon result of \cite{frikha2024full} with a passage to the limit as $T\to\infty$. The previous lemmas ensure that this limit can be performed both at the level of the value function and in the equation itself, which yields the classical solution of the stationary HJB equation together with uniqueness in the polynomial-growth class. The argument is also closely related in spirit to the McKean-Vlasov PDE framework of \cite{BuckdahnLiPengRainer2017}, although the latter deals with finite-horizon problems, whereas here we consider an infinite-horizon discounted setting.


\section{Policy Gradient}
\label{sec:policy-iteration}

This section develops the policy improvement step for the mean-field
control problem introduced in Section \ref{sec:formulation}.  The
policy evaluation step has already been addressed in
Section \ref{sec:HJB}: given a feedback randomized policy $\pi$, the
value function $V^\pi$ is well defined and can be characterized as the
unique classical solution of a stationary HJB equation.  The goal of
the present section is to derive an explicit formula for the gradient
of the performance functional with respect to a parametric family of
policies.

A distinctive feature of the mean-field setting is that the
performance functional depends on two initial data: the initial law
$\mu \in \mathcal{P}_2(\mathbb{R}^d)$ generating the population flow,
and the initial distribution $\mu_0 \in \mathcal{P}_2(\mathbb{R}^d)$
of the representative agent.  Throughout this section, $\mu$ is fixed
and the performance of a policy $\pi$ is measured by
\begin{equation*}
    J_{\mu,\mu_0}(\pi)
    \;:=\;
    \int_{\mathbb{R}^d} V^\pi(s,\mu)\,\mu_0(\d s).
\end{equation*}
For a parametric family $\{\pi_\omega\}_{\omega\in\mathbb{R}^p}$ we write
\begin{equation}\label{eq:parametric-objective-policy-iteration}
    J(\omega)
    \;:=\;
    J_{\mu,\mu_0}(\pi_\omega)
    \;=\;
    \int_{\mathbb{R}^d} V^{\pi_\omega}(s,\mu)\,\mu_0(\d s).
\end{equation}

The main difficulty is that $J(\omega)$ depends on $\omega$ both
through the representative dynamics and through the induced mean-field
flow. In what follows we will provide a policy gradient theorem and a result for the parametrized case.

\begin{remark}
For any $s_0\in\R^d$, the Dirac measure $\delta_{s_0}$ belongs to $\mathcal{P}_2(\R^d)$. Hence, by choosing $\mu_0=\delta_{s_0}$, the results of this section apply to the pointwise performance functional $J_{\mu,s_0}(\pi):=V^\pi(s_0,\mu)$ as well.
\end{remark}

\subsection{Policy gradient theorem}
\label{subsec:policy-gradient}

Let $\pi\in \Pi_{\mathrm{add}}$ and let $\varphi$ be a measurable
signed kernel $\mathbb{R}^d \times \mathcal{P}_2(\mathbb{R}^d) \to
\mathcal{M}(\mathcal{A})$ satisfying
\begin{equation}\label{eq:signed-kernel-zero-mass}
    \varphi(\cdot \mid s,\mu) \ll \pi(\cdot \mid s,\mu),
    \qquad
    \int_{\mathcal{A}} \varphi(\d a \mid s,\mu) = 0,
    \qquad
    (s,\mu) \in \mathbb{R}^d \times \mathcal{P}_2(\mathbb{R}^d),
\end{equation}
with bounded Radon-Nikodym derivative $\psi := \d\varphi/\d\pi \in
L^\infty$.  The zero-mass condition in \eqref{eq:signed-kernel-zero-mass}
ensures that $\pi^\varepsilon$ remains a probability kernel for small
$\varepsilon$.  For $|\varepsilon|$ small, define the perturbed kernel
\begin{equation}\label{eq:epsilon-perturbed-policy}
    \pi^\varepsilon(\d a \mid s,\mu)
    := \pi(\d a \mid s,\mu) + \varepsilon\,\varphi(\d a \mid s,\mu)
    = \bigl(1 + \varepsilon\psi(s,\mu,a)\bigr)\,\pi(\d a \mid s,\mu),
\end{equation}
and assume $\pi^\varepsilon$ is an admissible feedback randomized
policy for all sufficiently small $|\varepsilon|$.
\smallbreak

Before introducing the policy gradient theorem, we introduce the following advantage functions. 

\begin{definition}[Advantage functions]
\label{def:essential-pointwise-kernels}
Let $\pi$ be a feedback randomized policy with value function
$V^\pi$.  For $(s,\mu,a) \in \mathbb{R}^d \times
\mathcal{P}_2(\mathbb{R}^d) \times \mathcal{A}$, define the
\emph{representative advantage function}
\begin{equation*}
    q^{\pi}_{\mathrm{rep}}(s,\mu,a)
    \;:=\;
    r(s,\mu,a)
    - \lambda\log p_\pi(a \mid s,\mu)
    + b(s,\mu,a) \cdot \nabla_s V^\pi(s,\mu)
    + \tfrac{1}{2}\,\Sigma(s,\mu,a) : D^2_{ss}V^\pi(s,\mu),
\end{equation*}
and for $(\xi,a) \in \mathbb{R}^d \times \mathcal{A}$, define the
\emph{population advantage function}
\begin{equation*}
    q^{\pi}_{\mathrm{pop}}(s,\mu,\xi,a)
    \;:=\;
    b(\xi,\mu,a) \cdot \partial_\mu V^\pi(s,\mu)(\xi)
    + \tfrac{1}{2}\,\Sigma(\xi,\mu,a) :
      D_\xi\partial_\mu V^\pi(s,\mu)(\xi).
\end{equation*}
\end{definition}

\begin{remark}[Interpretation of the advantage functions]\label{rmk:adv}
In standard reinforcement learning, the advantage function can be
written through the Bellman evaluation identity as
\[
A^\pi(s,a)
=
r(s,a)+\gamma\,\mathbb E\!\left[
V^\pi(S')\mid S=s,\mathfrak a=a
\right]-V^\pi(s).
\]
Thus, it compares the effect of taking the action $a$ with the value
generated by the policy. In the single-agent continuous-time setting,
the same structure is expressed through the infinitesimal generator, so
that the corresponding action-dependent term is given by the running
reward plus the generator applied to the value function, minus the
discount term.

The present mean-field setting preserves this structure. The
representative advantage function
$q^\pi_{\mathrm{rep}}$ is the natural extension of the
single-agent continuous-time term: it contains the entropy-regularized
instantaneous reward together with the contribution of the generator
acting on the state variable \cite{jia_qlearning2023, Baird1994ReinforcementLI,tallec2019, zhu2024phibe}. The population advantage function
$q^\pi_{\mathrm{pop}}$ is specific to the mean-field
framework and describes the additional contribution coming from the
dependence of $V^\pi$ on the population law through its Lions
derivatives. In this way, the classical advantage structure is
decomposed into a representative part and a population part. 

Finally, let us note that \cite{Wei2023ContinuousTQ} introduces two
q-functions in the McKean--Vlasov control setting, an integrated
q-function and an essential q-function. The integrated q-function is a
central object in the proposed learning procedure. Since its weak
martingale characterization is formulated with respect to all test
policies in a neighbourhood of the target policy, its practical
implementation relies on averaging the corresponding loss over nearby
test policies. The essential q-function plays the role of a
pointwise-in-action quantity for policy improvement, and in this
respect our kernels $q^\pi_{\mathrm{rep}}$ and
$q^\pi_{\mathrm{pop}}$ are closer in spirit to that object, although
our formulation is for the coupled representative-agent / population
system rather than for the population-only formulation. The main
difference is that, in our actor--critic pipeline, the policy gradient
theorem requires only the pointwise-in-action kernels
$q^\pi_{\mathrm{rep}}$ and $q^\pi_{\mathrm{pop}}$, which are obtained
directly from $V^\pi$ and its derivatives. Therefore, no auxiliary
integrated q-function is needed, nor does our approach require testing
over a neighbourhood of the target policy.

\end{remark}

\begin{definition}[Discounted occupancy measure]\label{def:discounted_measure}
Let $\pi\in\Pi_{\mathrm{add}}$ be a fixed policy, and let $(s_t^\pi,\mu_t^\pi)_{t\ge0}$ be the state-population process associated with \eqref{eq:rep_SDE}-\eqref{eq:MKV_SDE} starting from $s_0^\pi=s\sim \mu_0,$ and $\mu_0^\pi=\mu$. The discounted occupancy measure $\rho_{\mu,\mu_0}^\pi$ on $\R^d\times\mathcal P_2(\R^d)$ is defined by
\begin{align*}
\rho_{\mu,\mu_0}^\pi(B)
:=
\beta\, \mathbb{E}\!\left[\int_0^\infty e^{-\beta t}\,
\mathbf{1}_{\{(s_t^\pi,\mu_t^\pi)\in B\}}\,\d t\right],
\qquad
B\in\mathcal B\bigl(\R^d\times\mathcal P_2(\R^d)\bigr).
\end{align*}
\end{definition}
\begin{remark}[Equivalent representations of $\rho_{\mu,\mu_0}^\pi$]
For $t\ge0$, let $P_t^\pi(s_0,\mu;\cdot):=\Law(s_t^\pi\mid s_0,\mu)$.
Then, by Tonelli's theorem, $\rho_{\mu,\mu_0}^\pi$ admits the representation
\[
\rho_{\mu,\mu_0}^\pi(B)
=
\beta \int_0^\infty e^{-\beta t}
\left(
\int_{\mathbb R^d}\int_{\mathbb R^d}
\mathbf 1_B(x,\mu_t^\pi)\,
P_t^\pi(s_0,\mu;\d x)\,\mu_0(\d s_0)
\right)\d t,
\qquad
B\in\mathcal B(\mathbb R^d\times\mathcal P_2(\mathbb R^d)).
\]
Moreover, denoting by $\delta_{\mu_t^\pi}\in\mathcal P(\mathcal P_2(\mathbb R^d))$
the Dirac measure at $\mu_t^\pi$, the previous identity can be written
equivalently, in measure form, as
\[
\rho_{\mu,\mu_0}^\pi(\d x,\d\nu)
=
\beta\int_0^\infty e^{-\beta t}
\left(
\int_{\mathbb R^d}
P_t^\pi(s_0,\mu;\d x)\,\mu_0(\d s_0)
\right)\delta_{\mu_t^\pi}(\d\nu)\,\d t.
\]
\end{remark}

\begin{theorem}[G\^ateaux policy gradient]
\label{thm:policy-gradient-theorem}
Fix a baseline policy $\pi\in\Pi_{\mathrm{add}}$. Let $\varphi$ be a signed kernel
satisfying \eqref{eq:signed-kernel-zero-mass}, and let $\pi^\varepsilon$ be as
in \eqref{eq:epsilon-perturbed-policy}, admissible for all
$|\varepsilon|\leq\varepsilon_0$. Suppose Assumption \ref{A1}
holds for both $\pi$ and $\pi^\varepsilon$, with  $|\varepsilon|\leq\varepsilon_0$. Assume $\mu,\;\mu_0\in\mathcal{P}_2(\mathbb{R}^d)$ and $\beta > 2 \bar\beta_\pi$, where $\bar\beta_\pi$ from
Theorem \ref{thm:MF_eval_infinite}. Then
the map
$\varepsilon\mapsto J(\varepsilon):=\int_{\mathbb{R}^d}
V^{\pi^\varepsilon}(s,\mu)\,\mu_0(\d s)$
is differentiable at $\varepsilon=0$, with
\begin{equation}\label{eq:gateaux-pg}
    \begin{aligned}
       \frac{\d}{\d\varepsilon}J(\varepsilon)\bigg|_{\varepsilon=0}
=
\frac1\beta\mathbb{E}_{(s,\mu)\sim\rho_{\mu,\mu_0}^\pi}\!\left[
\int_{\mathcal A}
q^{\pi}_{\mathrm{rep}}(s,\mu,a)\,
\varphi(\d a\mid s,\mu)
+
\int_{\R^d}\int_{\mathcal A}
q^{\pi}_{\mathrm{pop}}(s,\mu,\xi,a)\,
\varphi(\d a\mid \xi,\mu)\,\mu(\d\xi)
\right].
    \end{aligned}
\end{equation}
\end{theorem}

The proof of \Cref{thm:policy-gradient-theorem} is provided in \Cref{APP:PROOFS_SEC_PG}. 

\begin{remark}\label{rmk:key diff}
The preceding theorem, together with Theorem~\ref{thm:parametric-policy-gradient},
is a key contribution of this paper and distinguishes our approach from existing
policy-gradient methods and related $q$-function-based algorithms. Compared with \cite{frikhaActorCritic}, we provide a parameterization-free formulation of the policy gradient, which describes the first-order variation of the objective under arbitrary admissible perturbations of the policy. Compared with
\cite{Wei2023ContinuousTQ}, we use the pointwise $q$-functions to derive the
policy-gradient formula directly, while their method requires computing the integrated q-functions for all test policies in a neighbourhood of the target policy.
\end{remark}

\subsection{Parametric policy gradient}
\label{subsec:parametric-pg}

The G\^ateaux gradient formula \eqref{eq:gateaux-pg} expresses the
directional derivative of $J$ in terms of an arbitrary signed kernel
$\varphi$.  In practice, policies are often parametrized by a finite
dimensional vector $\omega\in\mathbb{R}^p$, and one wishes to compute
$\nabla_\omega J(\omega)$ directly.  The parametric policy gradient
theorem below reduces this computation to a weighted expectation of
the advantage functions against the score function. To state the result precisely, we require the
following two assumptions on the parametric family.
\medbreak

\refstepcounter{Acond}\label{A2}%
\noindent\textbf{(\theAcond) Parametric differentiability of averaged
coefficients.}
For $\nu$-almost every $a\in\mathcal{A}$ and every
$(s,\mu)\in\mathbb{R}^d\times\mathcal{P}_2(\mathbb{R}^d)$, the map
$\omega\mapsto\log p_{\pi_\omega}(a\mid s,\mu)$ is continuously
differentiable. Moreover, for every compact set $K\subset\mathbb{R}^p$,
there exists $C_K>0$ and a function $g_K(\cdot,s,\mu)\in L^1(\mathcal{A},\nu)$ such that for all $\omega\in K$ and all
$(s,\mu)\in\mathbb{R}^d\times\mathcal{P}_2(\mathbb{R}^d)$, the maps $\omega\mapsto p_{\pi_\omega}(\cdot\mid s,\mu)$,
$\omega\mapsto b^{\pi_\omega}(s,\mu)$, $\omega\mapsto\Sigma^{\pi_\omega}(s,\mu)$,
and $\omega\mapsto r_\lambda^{\pi_\omega}(s,\mu)$ are continuously
differentiable, satisfying
\begin{align}    \label{eq:A2-Sigma_r}
\begin{aligned}
&\bigl\|\nabla_\omega p_{\pi_\omega}(a\mid s,\mu)\bigr\|
\le g_K(a,s,\mu)
\qquad\text{for $\nu$-a.e. }a\in\mathcal A.\\
    &\|\nabla_\omega b^{\pi_\omega}(s,\mu)\|
    \leq C_K\bigl(1+|s|+\sqrt{m_2(\mu)}\bigr),\\
    \|\nabla_\omega r_\lambda^{\pi_\omega}&(s,\mu)\|+\|\nabla_\omega\Sigma^{\pi_\omega}(s,\mu)\|
    \leq C_K\bigl(1+|s|^2+m_2(\mu)\bigr).
\end{aligned}
\end{align}

\refstepcounter{Acond}\label{A3}%
\noindent\textbf{(\theAcond) Uniform regularity of the parametric family.}
Assumption \ref{A1} holds for every $\pi_\omega$, with
Lipschitz constant $L_\pi$ and regularity constant $K_\pi$ independent
of $\omega\in\mathbb{R}^p$.

\begin{remark}[On the Assumptions \ref{A1}-\ref{A3}]
\label{rem:gaussian-policy}
The Gaussian policy $\pi_\omega(\cdot\mid s,\mu)
:=\mathcal{N}(f_\omega(s,\mu),\Sigma_0)$, with $\Sigma_0\in\mathbb{R}^{m\times m}$
positive definite, satisfies Assumptions \ref{A1}-\ref{A3} whenever
$f_\omega$ is Lipschitz with linear growth in $(s,\mu)$, belongs to
$\mathcal{C}^{2,2}(\mathbb{R}^d\times\mathcal{P}_2(\mathbb{R}^d))$
with all spatial and Lions derivatives of at most polynomial growth
in $(s,\mu,\xi)$, and $\omega\mapsto f_\omega(s,\mu)$ is continuously
differentiable with at most polynomial growth in $(s,\mu)$; all
uniformly in $\omega$ on compact sets. In particular, this holds when
$f_\omega$ is linear in $s$ and $m_1(\mu)$, or when $f_\omega$ is a
cylindrical neural network \cite[Section 2.2]{pham2022meanfield} with $\mathcal{C}^2$ activation
function of at most linear growth, such as $\tanh$ or $\mathrm{GELU}$.
\end{remark}

\begin{theorem}[Parametric policy gradient]
\label{thm:parametric-policy-gradient}
Let $\{\pi_\omega\}_{\omega\in\mathbb{R}^p}$ be a family of feedback
randomized policies satisfying Assumptions \ref{A1}-\ref{A3}. Fix
$\omega\in\mathbb{R}^p$ and set $\pi:=\pi_\omega$. Assume
$\mu,\mu_0\in\mathcal{P}_2(\mathbb{R}^d)$ and
$\beta>\widetilde\beta_\pi:=\max(\bar\beta_\pi,\,2\beta_0(2))$. Then the map $\omega\mapsto J(\omega)$, defined in \eqref{eq:parametric-objective-policy-iteration}, is differentiable at $\omega$, and
\begin{equation}\label{eq:parametric-pg-expectarion}
    \begin{aligned}
        \nabla_\omega J(\omega)
        &= \frac1\beta\mathbb{E}_{\substack{(s,\mu)\sim \rho_{\mu,\mu_0}^\pi\\a\sim \pi(\cdot|s,\mu)}}\!\Bigl[
            q^{\pi}_{\mathrm{rep}}(s,\mu,a)\,
            \nabla_\omega\log p_\pi(a|s,\mu)\Bigr]\!+\frac1\beta\mathbb{E}_{\substack{(s,\mu)\sim \rho_{\mu,\mu_0}^\pi\\ \xi\sim \mu\\a\sim \pi(\cdot|\xi,\mu)}}\!\Bigl[
            q^{\pi}_{\mathrm{pop}}(s,\mu,\xi,a)\,
            \nabla_\omega\log p_\pi(a|\xi,\mu)
        \Bigr].
    \end{aligned}
\end{equation}
Equivalently, this can be written as
\begin{equation}\label{eq:parametric-pg}
    \begin{aligned}
        \nabla_\omega J(\omega)
        &= \frac1\beta\mathbb{E}_{(s,\mu)\sim \rho_{\mu,\mu_0}^\pi}\!\left[
            \int_{\mathcal{A}}
            q^{\pi}_{\mathrm{rep}}(s,\mu,a)\,
            \nabla_\omega\log p_\pi(a\mid s,\mu)\,
            \pi(\d a\mid s,\mu)\right]\\
            &\quad\qquad\qquad+\frac1\beta\mathbb{E}_{(s,\mu)\sim \rho_{\mu,\mu_0}^\pi}\!\left[\int_{\mathbb{R}^d}\!\int_{\mathcal{A}}
            q^{\pi}_{\mathrm{pop}}(s,\mu,\xi,a)\,
            \nabla_\omega\log p_\pi(a\mid\xi,\mu)\,
            \pi(\d a\mid\xi,\mu)\,
            \mu(\d\xi)
        \right].
    \end{aligned}
\end{equation}
\end{theorem}
The proof of \Cref{thm:parametric-policy-gradient} is given in \Cref{APP:PROOFS_SEC_PG}.

\begin{remark} \label{rmk:para-gd}
The identity \eqref{eq:parametric-pg-expectarion} keeps the classical
policy-gradient form of \emph{advantage times score}, but now with two
terms. The first expectation involves
$q^\pi_{\mathrm{rep}}$ and corresponds to the representative-agent
contribution, whereas the second involves
$q^\pi_{\mathrm{pop}}$ and captures the population contribution.

From a computational viewpoint, once $V^\pi$ has been evaluated or
approximated, the quantities entering
\eqref{eq:parametric-pg-expectarion} are obtained directly from
$V^\pi$, $b^\pi$, $\Sigma^\pi$, and the score function
$\nabla_\omega \log p_\pi(a\mid s,\mu)$. The expectations with respect
to $\rho_{\mu,\mu_0}^\pi$ can then be approximated from trajectories of
the controlled dynamics under the current policy $\pi_\omega$, together
with samples of actions drawn from $\pi_\omega(\cdot\mid s,\mu)$.
This is precisely the mechanism underlying the actor--critic algorithm
described in Section \ref{sec:policy-iteration-algorithm}.

The formula \eqref{eq:parametric-pg-expectarion} should be contrasted
with the policy-gradient representation in \cite{frikhaActorCritic}. In
\cite{frikhaActorCritic}, the gradient is obtained by differentiating the HJB
characterization of the value function with respect to the policy parameter.
As a result, the gradient itself is characterized as the solution of another
linear HJB equation, whose source term is constructed from the value function
under the current policy. Thus, computing the ascent direction requires solving
an additional policy-evaluation-type problem. By contrast, \eqref{eq:parametric-pg-expectarion} gives an explicit gradient
formula once $V^\pi$ and its derivatives are available. This makes the actor
step more direct.

Moreover, the differentiated HJB equation in
\cite{frikhaActorCritic} contains an additional term $H_\omega[J]$, which
involves the derivatives $\nabla_\omega b^{\pi_\omega}$ and
$\nabla_\omega \Sigma^{\pi_\omega}$. As emphasized in
\cite[Remark~3.3]{frikhaActorCritic}, when the underlying dynamics are unknown, this term is explicitly computable only in special settings, such as separable models. In
contrast, our gradient formula is expressed directly through the
pointwise-in-action kernels derived from $V^\pi$ and its 
derivatives. This form avoids differentiating the aggregated drift and
diffusion with respect to the policy parameters, and is therefore better suited
for the model-free extensions developed in the follow-up paper \cite{secondpaper}.

Finally, we note that Theorem~3.1 of \cite{frikhaActorCritic} is equivalent to our parametric policy-gradient formula \eqref{eq:parametric-pg-expectarion}.
When our initial distribution $\mu_0$ of the representative state is taken to be a
Dirac measure $\mu_0=\delta_s$, our pointwise form of \eqref{eq:parametric-pg-expectarion} agrees with their gradient $G(s,\mu)$. For a general initial
distribution $\mu_0$, averaging their pointwise gradient $\mathbb{E}_{s\sim\mu_0}[G(s,\mu)]$ gives exactly our gradient $\nabla_\omega J$ defined in \eqref{eq:parametric-pg-expectarion}.

\end{remark}

\def\tcr{\textcolor{red}}


\section{Actor-Critic Algorithm}
\label{sec:policy-iteration-algorithm}

We now use the previous results to formulate an actor--critic algorithm. For a fixed policy $\pi_\omega$, the critic
provides a finite-dimensional approximation of the value function
$V^{\pi_\omega}$ solving \eqref{eq:HJB-eval}, while the actor updates
$\omega$ through the policy gradient formula derived in
Section~\ref{sec:policy-iteration}.

For the critic, we use a linear finite-dimensional ansatz on
$\R^d\times\mathcal P_2(\R^d)$, with cylindrical basis functions to
encode the dependence on the population law. This choice is consistent
with classical linear value-function approximation in reinforcement
learning for large state spaces
\cite{Szepesvari2010Algorithms,JinYangWangJordan2020,
AyoubJiaSzepesvariWangYang2020,YinHaoAbbasiYadkoriLazicSzepesvari2022},
as well as with related approximation methods in mean-field control
\cite{BayraktarKara2025LFA}.

Once such a finite-dimensional critic has been obtained, either by the
Galerkin procedure below or by any other method, the policy gradient
theorem can be applied directly. This yields an actor update that retains both the
representative-agent and the population-law contributions of the
coupled mean-field problem.

\subsection{Policy evaluation via Galerkin approximation}
\label{subsec:galerkin-policy-evaluation}
Before introducing the Galerkin approximation, we emphasize that our actor-critic algorithm is not tied to a particular numerical method for solving the linear HJB equation. In principle, the critic can be computed using other approximation schemes, such as nonlinear function approximation with deep neural networks \cite{pham2025actor}, or estimated through martingale-based characterizations of the value function \cite{Wei2023ContinuousTQ}. The Galerkin method presented here should therefore be viewed as one possible critic approximation scheme, corresponding to linear function approximation. We focus on the Galerkin approach for two reasons. First, it gives a transparent and stable way to approximate the value function through projection onto a finite-dimensional function space. Second, this linear-projection structure is well suited for our subsequent extension to the model-free reinforcement-learning setting \cite{secondpaper}, where the dynamics are unknown and the critic must be estimated from data. In that setting, nonlinear function approximation for Bellman-type equations can suffer from the double-sampling issue, which may introduce instability in policy evaluation. By contrast, the Galerkin formulation leads to a more tractable projected critic equation and provides a useful starting point for developing stable model-free algorithms.

\medbreak

Let $\{\phi_1,\dots,\phi_n\}$ be smooth functions on
$\R^d\times\mathcal P_2(\R^d)$ and define the trial space $\mathbb V_n:=\mathrm{span}\{\phi_1,\dots,\phi_n\}.$
We approximate the value function by
\begin{align}\label{eq:galerkin_apprimation}
V_\theta(s,\mu):=\sum_{i=1}^n\theta_i\,\phi_i(s,\mu)
=\theta^\top\Phi(s,\mu),
\qquad
\Phi:=(\phi_1,\dots,\phi_n)^\top,
\end{align}
where $\theta\in\R^n$ is the vector of critic coefficients. Although
the basis functions are defined on
$\R^d\times\mathcal P_2(\R^d)$, the approximation is finite-dimensional
because only finitely many of them are used. In practice, the
dependence on $\mu$ will be chosen through cylindrical functions, as
explained in Subsection \ref{subsec:cylindrical-basis-functions}.

For a fixed policy $\pi_\omega$, we impose the HJB residual associated
with $V_\theta$ to be orthogonal to the trial space with respect to the
reference measure $\mathfrak m$. Thus, we seek
$\theta^*(\omega)\in\R^n$ such that for every $j=1,\dots,n$,
\begin{equation}\label{eq:galerkin-projection-rho}
\int_{\R^d\times\mathcal P_2(\R^d)}
\Bigl(
(\mathcal L_{b,\Sigma}^{\pi_\omega}V_{\theta^*(\omega)})(s,\mu)
-\beta V_{\theta^*(\omega)}(s,\mu)
+r_\lambda^{\pi_\omega}(s,\mu)
\Bigr)\phi_j(s,\mu)\,
\mathfrak m(\d s,\d\mu)=0,
\end{equation}
where $\mathfrak m$ is a fixed Borel probability measure on $\R^d\times\mathcal P_2(\R^d)$. Since the equation is linear in $V$, \eqref{eq:galerkin-projection-rho}
is equivalent to the linear system
\begin{equation}\label{eq:linear-system-policy-evaluation}
A(\omega)\,\theta^*(\omega)=b(\omega),
\end{equation}
where
\begin{equation}\label{eq:definition_A_B_galerkin}
    \begin{aligned}
        A_{ji}(\omega)
&=
\int_{\R^d\times\mathcal P_2(\R^d)}
\phi_j(s,\mu)\,
\Bigl(
\beta\,\phi_i(s,\mu)
-(\mathcal L_{b,\Sigma}^{\pi_\omega}\phi_i)(s,\mu)
\Bigr)\,
\mathfrak m(\d s,\d\mu),
\\
b_j(\omega)
&=
\int_{\R^d\times\mathcal P_2(\R^d)}
r_\lambda^{\pi_\omega}(s,\mu)\,
\phi_j(s,\mu)\,
\mathfrak m(\d s,\d\mu).
    \end{aligned}
\end{equation}
Consequently, whenever $A$ is invertible, the solution of \eqref{eq:linear-system-policy-evaluation} is given by $\theta^*(\omega)= (A(\omega))^{-1}b(\omega)$. In any case, independently of the particular approximation used for
$A(\omega)$ and $b(\omega)$, the corresponding critic is obtained as the
Galerkin approximation
\begin{align}\label{eq:solution_gal_approx}
    V_n^{\pi_\omega}(s,\mu)
    :=
    V_{\theta^*(\omega)}(s,\mu).
\end{align}

\subsection{Cylindrical basis functions and moment features}
\label{subsec:cylindrical-basis-functions}

A central issue in the mean-field setting is the representation of the
measure variable. A convenient class of functions is given by
cylindrical functions, for which the Lions derivatives can be computed
explicitly. This makes them particularly attractive for Galerkin
approximations of the HJB equation.

Let $k\in\N$ and fix feature functions
$\varphi_1,\dots,\varphi_k\in C_b^2(\R^d)$. We define the associated
moment map
\[
m(\mu):=\bigl(m_1(\mu),\dots,m_k(\mu)\bigr)\in\R^k,
\qquad
m_i(\mu):=\int_{\R^d}\varphi_i(\xi)\,\mu(\d\xi),
\quad i=1,\dots,k.
\]
A function $F:\R^d\times\mathcal P_2(\R^d)\to\R$ is called cylindrical
with respect to the feature family $\{\varphi_i\}_{i=1}^k$ if there
exists a smooth map $\Psi:\R^d\times\R^k\to\R$ such that
\[
F(s,\mu)=\Psi\bigl(s,m(\mu)\bigr).
\]
Thus, the dependence on $\mu$ is reduced to finitely many moments or,
more generally, finitely many linear observables of $\mu$.

This construction is useful for the policy-evaluation step for two
reasons. First, it leads to a finite-dimensional parametrization of the
measure argument. Second, the derivatives entering the HJB operator
become explicit. Indeed, if $F(s,\mu)=\Psi(s,m(\mu))$, then
\begin{align}\label{eq:deerivatives_of_cilindrical_func}
\begin{aligned}
    \nabla_s F(s,\mu)
=
\nabla_s\Psi\bigl(s,m(\mu)\bigr),&
\qquad
D^2_{ss}F(s,\mu)
=
D^2_{ss}\Psi\bigl(s,m(\mu)\bigr),
\\
\partial_\mu F(s,\mu)(\xi)
=
\sum_{i=1}^k
\partial_{m_i}\Psi\bigl(s,m(\mu)\bigr)\,\nabla\varphi_i(\xi),&
\qquad
D_\xi\partial_\mu F(s,\mu)(\xi)
=
\sum_{i=1}^k
\partial_{m_i}\Psi\bigl(s,m(\mu)\bigr)\,D^2_{\xi\xi}\varphi_i(\xi).
\end{aligned}
\end{align}
Therefore, once the feature functions $\varphi_i$ are fixed, the
operator $\mathcal L_{b,\Sigma}^{\pi_\omega}F$ can be evaluated
explicitly for cylindrical functions. This motivates choosing, in the
Galerkin approximation \eqref{eq:galerkin-projection-rho}, basis
functions of the form
$\phi_j(s,\mu)=\psi_j\bigl(s,m(\mu)\bigr)$, $j=1,\dots,n$, with
$\psi_j$ smooth on $\R^d\times\R^k$. In this setting, the problem is
naturally expressed in terms of the reduced variables
$(s,m(\mu))\in\R^d\times\R^k$, and the reference measure may therefore
be chosen as a Borel probability measure $\overline{\mathfrak m}$ on
$\R^d\times\R^k$, built from $N>0$ observations
$\{(s_\tau,m_\tau)\}_{\tau=1}^{N}$, with $m_\tau=m(\mu_\tau)$ , sampled on the time grid
$t_\tau=(\tau-1)\Delta t$ along the trajectories of
\eqref{eq:rep_SDE}--\eqref{eq:MKV_SDE}. A natural choice is to take $\overline{\mathfrak m}$ as the empirical
discounted occupancy measure on the reduced space, namely
\begin{align}
\label{eq:uni_weighted_measure}
    \overline{\mathfrak m}
    =
    \beta\sum_{\tau=1}^{N}
    e^{-\beta t_\tau}\Delta t\,
    \delta_{(s_\tau,m(\mu_\tau))},
    \qquad\text{or}\qquad
    \overline{\mathfrak m}
    =
    \frac{1}{N}\sum_{\tau=1}^{N}
    \delta_{(s_\tau,m(\mu_\tau))},
\end{align}
the latter corresponding to the empirical undiscounted measure on the
finite observation window. In Section~\ref{subsec:systemic-risk}, we also implemented the undiscounted measure. Our numerical results indicate that the algorithm is robust with respect to the choice of measure $\overline{\mathfrak m}$.

In practice, one may average over $L>0$ independent trajectories in order
to reduce the variance of the empirical estimator. This is naturally
reflected in the definition of $\overline{\mathfrak m}$ by averaging the
empirical measure across trajectories. Consequently, if
$\{(s_\eta,m_\eta)\}_{\eta=1}^{N_D}$, with $N_D=N\cdot L$, denotes the
collection of $N$ observations obtained from these $L$ trajectories,
then one empirical approximation of \eqref{eq:definition_A_B_galerkin},
based on the undiscounted measure, is given by
\begin{equation*}
    \begin{aligned}
        \widehat A_{ji}(\omega)
        &:=
        \frac{1}{N_D}\sum_{\eta=1}^{N_D}
        \psi_j(s_\eta,m_\eta)\,
        \Bigl(
        \beta\,\psi_i(s_\eta,m_\eta)
        -
        (\mathcal L_{b,\Sigma}^{\pi_\omega}\phi_i)(s_\eta,m_\eta)
        \Bigr),
        \\
        \widehat b_j(\omega)
        &:=
        \frac{1}{N_D}\sum_{\eta=1}^{N_D}
        r_\lambda^{\pi_\omega}(s_\eta,m_\eta)\,
        \psi_j(s_\eta,m_\eta),
    \end{aligned}
\end{equation*}
for $i,j=1,\dots,n$. From a modeling viewpoint, the choice of the features $\varphi_1,\dots,\varphi_k$ should reflect the expected structure of the value function. In linear-quadratic problems, moment-based features are especially natural, because the value function depends on $\mu$ through a small number of low-order moments; see \Cref{thm:LQ_infinite}. For more general problems, one may enrich the family of features by adding nonlinear observables of the population distribution.

\subsection{Policy gradient}
\label{subsec:computational-policy-improvement}

Let $\pi_\omega$ be a parametrized policy, with parameter
$\omega\in\R^p$, and let $V_n^{\pi_\omega}$ be either the critic obtained
from \eqref{eq:solution_gal_approx}, or, more generally, any
finite-dimensional approximation of $V^{\pi_\omega}$ for which the
derivatives entering the policy gradient formula are well defined and
can be evaluated. The policy gradient theorem from
Section~\ref{sec:policy-iteration} provides an exact gradient formula
in terms of a representative contribution and a population
contribution. For implementation, we keep exactly the same structure,
but replace the exact value function $V^{\pi_\omega}$ by its
approximation $V_n^{\pi_\omega}$.

Accordingly, we define
$q_{\mathrm{rep},n}^{\pi_\omega}$ and
$q_{\mathrm{pop},n}^{\pi_\omega}$ as in
\Cref{def:essential-pointwise-kernels}, with $V^\pi$ replaced by
$V_n^{\pi_\omega}$. The resulting approximate policy-gradient
direction is
\begin{equation}\label{eq:approx-gradient-direction}
    \begin{aligned}
        \nabla_\omega J_n^{\pi_\omega}
        &:=
        \frac1\beta\int_{\R^d\times\mathcal P_2(\R^d)}
        \Biggl(
            \int_{\mathcal A}
            q_{\mathrm{rep},n}^{\pi_\omega}(s,\mu,a)\,
            \nabla_\omega \log p_{\pi_\omega}(a\mid s,\mu)\,
            \pi_\omega(\d a\mid s,\mu)
            \\
        &\qquad\qquad
            +
            \int_{\R^d}\!\int_{\mathcal A}
            q_{\mathrm{pop},n}^{\pi_\omega}(s,\mu,\xi,a)\,
            \nabla_\omega \log p_{\pi_\omega}(a\mid \xi,\mu)\,
            \pi_\omega(\d a\mid \xi,\mu)\,
            \mu(\d \xi)
        \Biggr)
        \rho^{\pi_\omega}(\d s,\d \mu).
    \end{aligned}
\end{equation}

When the critic is expanded in cylindrical basis functions,
$q_{\mathrm{rep},n}^{\pi_\omega}$ and
$q_{\mathrm{pop},n}^{\pi_\omega}$ can be evaluated explicitly. Indeed,
in that case the derivatives of $V_n^{\pi_\omega}$ are given directly by
the formulas in \eqref{eq:deerivatives_of_cilindrical_func}. For numerical purposes, \eqref{eq:approx-gradient-direction} may be
approximated from a finite collection of observations
$\{(s_\eta,\mu_\eta)\}_{\eta=1}^{N_D}\subset
\R^d\times\mathcal P_2(\R^d)$ together with nonnegative weights
$\{w_\eta\}_{\eta=1}^{N_D}$. This yields the empirical approximation
\begin{equation}\label{eq:general_empirical_gradient}
    \begin{aligned}
        \widehat{\nabla_\omega J_n^{\pi_\omega}}
        &:=
        \frac1\beta\sum_{\eta=1}^{N_D} w_\eta
        \Biggl(
            \int_{\mathcal A}
            q_{\mathrm{rep},n}^{\pi_\omega}(s_\eta,\mu_\eta,a)\,
            \nabla_\omega \log p_{\pi_\omega}(a\mid s_\eta,\mu_\eta)\,
            \pi_\omega(\d a\mid s_\eta,\mu_\eta)
            \\
        &\qquad\qquad\qquad
            +
            \int_{\R^d}\!\int_{\mathcal A}
            q_{\mathrm{pop},n}^{\pi_\omega}(s_\eta,\mu_\eta,\xi,a)\,
            \nabla_\omega \log p_{\pi_\omega}(a\mid \xi,\mu_\eta)\,
            \pi_\omega(\d a\mid \xi,\mu_\eta)\,
            \mu_\eta(\d \xi)
        \Biggr).
    \end{aligned}
\end{equation}
Several remarks on numerically computing the above gradient are in order. 
First, the integral can be approximated by empirical mean through sampling actions $a_j$ from the current policy and sampling states $s_i$ from the current $\mu_\eta(\xi)$.  Second, regarding the weight $w_\eta$,
when the measure in
\eqref{eq:approx-gradient-direction} is chosen as the discounted
occupancy measure $\rho^{\pi_\omega}$, we consider
$L$ independent trajectories
$\{(s_\tau^\ell,\mu_\tau^\ell)\}_{\tau=0}^{N-1}$,
$\ell=1,\dots,L$, of the controlled process
$(s_t^{\pi_\omega},\mu_t^{\pi_\omega})$ sampled on the time grid
$t_\tau=\tau\Delta t$, and identify the observations in
\eqref{eq:general_empirical_gradient} with these data, so that
$N_D=LN$ and
$w_\tau^\ell=\frac{1}{L}e^{-\beta t_\tau}\Delta t$. However, numerically we found that simply setting $w_\eta = \frac{1}{N_D}$ also provides a good approximation of the ascent direction. 
Finally, regardless of the specific approximation used for
$\nabla_\omega J_n^{\pi_\omega}$, the parameter update is of the form
\[
\omega^{(k+1)}
=
\omega^{(k)}+\alpha_k\,\widehat{\nabla_\omega J_n^{\pi_\omega}},
\]
where $\alpha_k>0$ is the stepsize. 
\smallbreak

The discussion above leads to \Cref{alg:computational-policy-iteration}.

\begin{algorithm}[H]
\caption{Computational policy iteration for the coupled system}
\label{alg:computational-policy-iteration}
\begin{algorithmic}[1]
\State Choose a parametrized policy family $\{\pi_\omega\}_{\omega\in\R^p}$.
\State Choose a finite-dimensional trial space
$\mathbb V_n=\mathrm{span}\{\phi_1,\dots,\phi_n\}$, preferably with
cylindrical basis functions in the measure variable.
\State Initialize the actor parameter $\omega^{(0)}$.
\For{$k=0,1,\dots,K-1$}
    \State \textbf{Policy evaluation:}
    solve the linear Galerkin system
    \eqref{eq:linear-system-policy-evaluation} and define
    \[
    V_n^{\pi_{\omega^{(k)}}}(s,\mu)
    :=(\theta^{(k)})^\top\Phi(s,\mu).
    \]
    \State \textbf{Policy improvement:}
    construct the approximate advantage functions
    $q_{\mathrm{rep},n}^{\pi_\omega}$ and $q_{\mathrm{pop},n}^{\pi_\omega}$ using
    $V_n^{\pi_{\omega^{(k)}}}$.
    \State Estimate the gradient direction
    $ \widehat J_n(\omega^{(k)})$ from \eqref{eq:general_empirical_gradient}.
    \State Update the actor by
    \[
    \omega^{(k+1)}
    =
    \omega^{(k)}+\alpha_k\, \widehat{\nabla_\omega J_n^{\pi_\omega}}.
    \]
\EndFor
\State Return the final policy $\pi_{\omega^{(K)}}$ and the associated
critic $V_n^{\pi_{\omega^{(K)}}}$.
\end{algorithmic}
\end{algorithm}

\section{Infinite-Time Mean-Field Linear Quadratic Regulator}
\label{sec:LQR}

In this section we analyze the infinite-horizon mean-field linear-quadratic regulator with stationary randomized feedback policies and entropy regularization. Our aim is to derive an explicit characterization of the optimal policy and of the associated value function in terms of a stationary algebraic Riccati system. Besides its intrinsic interest as a solvable class of mean-field control problems, this model is also relevant as a benchmark for entropy-regularized reinforcement learning, where randomized policies and Shannon regularization naturally arise in connection with exploration.

\subsection{Problem setting}
Let $d,m\in\mathbb N$ and set $\mathcal A=\mathbb R^m$. For notational simplicity we take a one-dimensional Brownian motion; the extension to a multi-dimensional Brownian motion is immediate by interpreting $\|\cdot\|_F$ and $\sigma\sigma^\top$ accordingly. Fix $(s,\mu)\in\mathbb R^d\times\mathcal P_2(\mathbb R^d)$ and consider the representative and population dynamics introduced in \eqref{eq:rep_SDE}-\eqref{eq:MKV_SDE}, where $\tilde s_0\sim\mu$ and $s_0=s$, with coefficients
\begin{equation}\label{eq:general_coeff_LQR}
    b(x,\mu,a)=Ax+\bar A\,\bar\mu+Ba,
    \qquad
    \sigma(x,\mu,a)=\gamma+Dx+\bar D\,\bar\mu+Fa,
    \qquad
    \bar\mu:=\int_{\mathbb R^d}x\,\mu(\d x),
\end{equation}
for fixed matrices $A,\bar A,D,\bar D\in\mathbb R^{d\times d}$, $B,F\in\mathbb R^{d\times m}$, and $\gamma\in\mathbb R^d$. Let $Q,\bar Q\in\mathbb S_+^d$, $N\in\mathbb S_{++}^m$, $I\in\mathbb R^{m\times d}$, and $M\in\mathbb R^d$, $H\in\mathbb R^m$. We consider the running reward
\begin{equation}\label{eq:LQ_running_reward}
    r(x,\mu,a)
    =
    -
    \Bigl(
        x^\top Qx
        +\bar\mu^\top\bar Q\,\bar\mu
        +a^\top Na
        +2a^\top I x
        +2M^\top x
        +2H^\top a
    \Bigr),
\end{equation}
and, for $\lambda>0$, the entropy-regularized reward $r_\lambda$ introduced in \eqref{eq:reg_reward_avg}. For $\beta>0$ we define
\begin{equation*}
    V^\pi(s,\mu)
    :=
    \mathbb E^{s,\mu,\pi}\!\left[
        \int_0^\infty e^{-\beta t}\,
        r_\lambda(s_t,\mu_t,a_t^s)\,\d t
    \right],
    \qquad
    V^*(s,\mu):=\sup_{\pi\in\Pi_{\mathrm{add}}}V^\pi(s,\mu),
\end{equation*}
where $\Pi_{\mathrm{add}}$ denotes the class of all stationary randomized feedback policies $\pi:\mathbb R^d\times\mathcal P_2(\mathbb R^d)\to\mathcal P(\mathbb R^m)$ such that $\pi(\cdot\mid x,\mu)$ admits a density $p_\pi(x,\mu,\cdot)$ with respect to Lebesgue measure, $V^\pi(s,\mu)<\infty$ for every $(s,\mu)$, and, for every initial condition $(s,\mu)$, $\lim_{t\to\infty}e^{-\beta t}\bigl(\mathbb E|s_t|^2+m_2(\mu_t)\bigr)=0.$

\begin{remark}\label{rem:sufficient_conditions_Piadd}
A sufficient, but not necessary, condition to ensure that  $\pi\in\Pi_{\mathrm{add}}$ is that $b^\pi$, $\sigma^\pi$, and $r_\lambda^\pi$ satisfy Assumption \ref{A1}, and that the dissipativity estimate \eqref{eq:dissipativity_remark} holds. Thus, $\pi\in\Pi_{\mathrm{add}}$ by \Cref{thm:MF_eval_infinite}. In the LQ setting, Assumption \ref{A1} is ensured by assuming that $m_\pi\in\mathcal C^{2,2}$ with bounded derivatives, that $q_\pi$ and $\mathcal  H\bigl(\pi(\cdot\mid x,\mu)\bigr)$ belong to $\mathcal C_{\mathrm{poly}}^{2,2}$, and that $m_\pi$ has at most linear growth in $(x,\sqrt{m_2(\mu)})$, where
\[
m_\pi(x,\mu):=\int_{\mathcal A} a\,\pi(\d a\mid x,\mu),
\qquad
q_\pi(x,\mu):=\int_{\mathcal A} a^\top N a\,\pi(\d a\mid x,\mu),
\]
and $\mathcal H\bigl(\pi(\cdot\mid x,\mu)\bigr)$ is defined in \eqref{eq:reg_reward_avg}. A sufficient condition for \eqref{eq:dissipativity_remark} is that $2\,x^\top A x+\|Dx\|_F^2\le c_0-\alpha|x|^2$ for all $x\in\mathbb R^d,$ for some $c_0\ge0$ and some $\alpha>0$ such that the remaining terms involving $\bar A,\bar D,B,F$ and $m_\pi$ can be absorbed by Young's inequality.
\end{remark}

\begin{definition}[Stabilizability, detectability, and mean-square stabilizability]
\label{def:stab_and_detect}
The pair $(A,B)\in\mathbb R^{d\times d}\times\mathbb R^{d\times m}$ is
\emph{stabilizable} if there exists $\Theta\in\mathbb R^{m\times d}$ such that
$A+B\Theta$ is Hurwitz. The pair
$(A,H)\in\mathbb R^{d\times d}\times\mathbb R^{q\times d}$ is
\emph{detectable} if there exists $L\in\mathbb R^{d\times q}$ such that
$A-LH$ is Hurwitz.
\end{definition}

\noindent
\refstepcounter{Acond}\label{ass:LQR}%
\noindent\textbf{(\theAcond) Discounted LQ conditions.}
Let $A_\beta:=A-\frac{\beta}{2}I_d,$ and $
\tilde A_\beta:=A+\bar A-\frac{\beta}{2}I_d.$ We assume that the following conditions hold:
\begin{enumerate}
    \item[(H1)] $N\in\mathbb S_{++}^m$ and
    \[
    Q-I^\top N^{-1}I\in\mathbb S_{++}^d,
    \qquad
    Q+\bar Q-I^\top N^{-1}I\in\mathbb S_{++}^d.
    \]

    \item[(H2)] The pair $(\tilde A_\beta,B)$ is stabilizable, and there exist
$\Theta\in\mathbb R^{m\times d}$ and $P_0\in\mathbb S_{++}^d$ such that
\[
(A_\beta+B\Theta)^\top P_0
+
P_0(A_\beta+B\Theta)
+
(D+F\Theta)^\top P_0(D+F\Theta)
\preceq -I_d.
\]
\end{enumerate}
\begin{remark}
Assumption {\rm(H1)} implies in particular that the pairs $(A_\beta,(Q-I^\top N^{-1}I)^{1/2})
$ and $
(\tilde A_\beta,(Q+\bar Q-I^\top N^{-1}I)^{1/2})$ are detectable. On the other hand, when $D=F=0$, condition {\rm(H2)} is equivalent to assuming that both pairs $(A_\beta,B)
$ and $
(\tilde A_\beta,B)$ are stabilizable.
\end{remark}

Let us consider the quadratic function 
\begin{equation}\label{eq:V_candidate_pf2}
    V(s,\mu)
    =
    -
    \Bigl(
        (s-\bar\mu)^\top K(s-\bar\mu)
        +\bar\mu^\top\Lambda\,\bar\mu
        +2Y^\top s
        +R
    \Bigr),
\end{equation}
where $K,\Lambda\in\mathbb S_+^d$, $Y\in\mathbb R^d$, and $R\in\mathbb R$ satisfy the algebraic system
\begin{empheq}[left=\empheqlbrace]{align}
    \beta K
    &= Q+KA+A^\top K+D^\top K D-U^\top S^{-1}U,
    \label{eq:ARE_P}\\
    \beta\Lambda
    &= Q+\bar Q+\Lambda(A+\bar A)+(A+\bar A)^\top\Lambda
    +(D+\bar D)^\top K(D+\bar D)-W^\top S^{-1}W,
    \label{eq:ARE_Pi}\\
    \beta Y
    & = M+A^\top Y+D^\top K\gamma-U^\top S^{-1}O,
    \label{eq:ARE_Y}\\
    \beta R
    &= \gamma^\top K\gamma-O^\top S^{-1}O
    -\frac{\lambda}{2}\Bigl(m\log(\pi_0\lambda)-\log\det S\Bigr),
    \label{eq:ARE_r}
\end{empheq}
where 
\begin{equation}\label{eq:tilde_notation}
    S:=N+F^\top K F,
    \qquad
    O:=H+B^\top Y+F^\top K\gamma,
\end{equation}
and
\begin{equation}\label{eq:tilde_notation_2}
    U:=I+B^\top K+F^\top K D,
    \qquad
    W:=I+B^\top \Lambda+F^\top K(D+\bar D).
\end{equation}
Here $\pi_0$ denotes the mathematical constant.

\begin{proposition}\label{thm:LQ_infinite}
Assume \Cref{ass:LQR} and let $(K,\Lambda,Y,R)$ be a solution of
\eqref{eq:ARE_P}-\eqref{eq:ARE_r}. Let $V$ be given by
\eqref{eq:V_candidate_pf2}, and let $\pi^*$ be defined as
\begin{equation}\label{eq:pi_star}
    \pi^*(\cdot\mid s,\mu)
    =
    \mathcal N\!\left(
        -S^{-1}\bigl(Us+(W-U)\bar\mu+O\bigr),\;
        \frac{\lambda}{2}S^{-1}
    \right).
\end{equation}
Then
\[
V^*(s,\mu)=V^{\pi^*}(s,\mu)=V(s,\mu)
\qquad
\text{for all }(s,\mu)\in\mathbb R^d\times\mathcal P_2(\mathbb R^d),
\]
and $\pi^*$ is the unique optimal stationary feedback policy in $\Pi_{\mathrm{add}}$.
\end{proposition}
The proof of \Cref{thm:LQ_infinite} is provided in \Cref{APP:PROOFS_SEC_LQR}.


\section{Numerical Examples}\label{sec:numerics}

\subsection{Example 1: Mean-field systemic risk model}
\label{subsec:systemic-risk}

Mean-field models of systemic risk describe a large population of interacting agents whose individual states are attracted to the current population average, while each agent can act through a control representing, for instance, borrowing or lending to the system; See \cite{CFS,carmona2018probabilistic1}. In our stationary discounted setting introduced in \Cref{sec:LQR}, this leads to a one-dimensional linear-quadratic specification in which the state is penalized through its deviation from the population mean and the control is penalized quadratically.

\subsubsection{Problem setup}

We consider the one-dimensional mean-field LQR dynamics in the particular systemic risk configuration, namely
\begin{equation*}
  \d s_t = \bigl(A (s_t-\bar\mu_t) + B a_t\bigr)\,\d t + \gamma\,\d B_t,\qquad a_t\sim \pi(\cdot\,|\,s_t,\mu_t),\qquad
  \bar\mu_t := \int_{\R} x\,\mu_t(\d x),
\end{equation*}
where $\mu_t=\Law(\tilde s_t)$ is the population law generated by the mean-field environment \eqref{eq:MKV_SDE}. This corresponds to the general framework of \Cref{sec:LQR} with $d=m=1,$ $\bar A=-A$, $D=\bar D=F=0,$ and $I=M=H=0.$ We take the running reward
\begin{equation}\label{eq:instant_reward_examples}
  r(x,\mu,a)
  =
  -\Bigl(
      Q(x-\bar\mu)^2+\bar Q\,\bar\mu^2+N a^2
    \Bigr),
\end{equation}
with $Q,\bar Q\ge0$ and $N>0$. Our goal is to maximize the expected long-run regularized reward
\begin{align}\label{eq:value_function_examples}
  V^{*}(s,\mu)
  \,:=\,
  \sup_{\pi\in \Pi_{\mathrm{add}}} \left\{ V^\pi(s,\mu):=\,
  \mathbb E\!\left[
    \int_0^\infty e^{-\beta t}\,
    r_\lambda^\pi(s_t,\mu_t)\,\d t
    \;\bigg|\;
    s_0=s,\ \mu_0=\mu
  \right]\right\} ,
\end{align}
for every $(s,\mu)\in\mathbb R^d\times\mathcal P_2(\mathbb R^d)$, 
where the entropy-regularized $r_\lambda^\pi(s,\mu)$ is as in \eqref{eq:reg_reward_avg} with $r(s,\mu,a)$ defined in \eqref{eq:instant_reward_examples}. Following \Cref{thm:LQ_infinite}, the optimal policy is given by
\begin{equation}\label{eq:optimal_policy_sistemic_risk}
  \pi^*(\cdot\mid s,\mu)
  =
  \mathcal N\!\left(
      -N^{-1}B^\top\bigl(K(s-\bar\mu)+\Lambda\,\bar\mu\bigr),
      \frac{\lambda}{2}N^{-1}
    \right).
\end{equation}
The corresponding optimal value function is
\begin{equation}\label{eq:optimal_value_sistemic_risk}
  V^*(s,\mu)=V^{\pi^*}(s,\mu)= -\Bigl((s-\bar\mu)^\top K(s-\bar\mu)+\bar\mu^\top\Lambda\,\bar\mu+R\Bigr),
\end{equation}
where $K,\Lambda\in\mathbb S_+^d$ and $R\in\mathbb R$ satisfy the reduced algebraic system
\begin{align*}
    \begin{cases}
         \beta K
  &= Q+KA+A^\top K-KBN^{-1}B^\top K,\\
  \beta\Lambda
  &= Q+\bar Q-\Lambda BN^{-1}B^\top\Lambda,\\
  \beta R
  &= \gamma^\top K\gamma-\frac{\lambda}{2}\Bigl(m\log(\pi_0\lambda)-\log\det N\Bigr).
    \end{cases}
\end{align*}

\subsubsection{Policy parametrization and value approximation}
Following the theoretical optimal policy \eqref{eq:optimal_policy_sistemic_risk}, we consider the family of Gaussian policies parametrized as 
\begin{equation}\label{eq:sysrisk-policy}
  \pi_\omega(\cdot \mid s,\mu)=\mathcal N\bigl(\omega_1(s-\bar\mu)+\omega_2\bar\mu,\,\lambda\bigr),
\end{equation}
where $\omega=(\omega_1,\omega_2)\in\R^2$. On the other hand, following \eqref{eq:optimal_value_sistemic_risk}, the value function can be approximated by the quadratic cylindrical ansatz
\begin{equation}\label{eq:sysrisk-Vn}
  V_\theta(s,\mu)
  =\theta_3\,(s-\bar\mu)^2 + \theta_2\,\bar\mu^2 + \theta_1,
  \qquad \theta=(\theta_1,\theta_2,\theta_3)\in\R^3,
\end{equation}
that is, following the notation of \eqref{eq:galerkin_apprimation}, we are taking $\phi_1(s,\mu)=1$, $\phi_2(s,\mu)=\bar\mu^2$, and $\phi_3(s,\mu)=(s-\bar\mu)^2$.

\subsubsection{Implementation and numerical results}

We implement the model-based actor--critic procedure described in \Cref{sec:policy-iteration-algorithm} for the scalar mean-field LQR example. The actor is parametrized by the Gaussian policy \eqref{eq:sysrisk-policy}, and the critic is sought in the quadratic space generated by \eqref{eq:sysrisk-Vn}. Throughout the experiment we use
$$
A=-1,\qquad B=1,\qquad Q=1,\qquad \bar Q=1,\qquad N=\tfrac12,\qquad \beta=1,\qquad \gamma=0.5,\qquad \lambda=0.2.
$$
The population is initialized with mean $\bar\mu_0=1$, and the initial representative state is sampled from $\mu_0=\mathcal N(1,1)$. Since the population dynamics are closed at the level of the mean $\bar\mu_t$, no auxiliary particle approximation of the population law is needed in this example. The representative state is sampled from the exact closed-loop Gaussian transition, while $\bar\mu_t$ is propagated through the corresponding deterministic mean equation. The discounted integrals are approximated on $[0,T]$, with $T=8$ and $\Delta t=0.05$. We use $100$ Monte Carlo trajectories per actor step, a constant learning rate $\alpha=5\times 10^{-2}$, and run $1000$ policy-improvement iterations.

For a fixed actor parameter $\omega=(\omega_1,\omega_2)$, the aggregated coefficients are
$$
b^{\pi_\omega}(s,\mu)=A(s-\bar\mu)+B\bigl(\omega_1(s-\bar\mu)+\omega_2\bar\mu\bigr),\qquad \Sigma^{\pi_\omega}(s,\mu)=\gamma^2.
$$
Thus, the Galerkin critic is computed by solving the linear system \eqref{eq:linear-system-policy-evaluation}, with $A(\omega)$ and $b(\omega)$ defined in \eqref{eq:definition_A_B_galerkin}, after replacing the general generator $\mathcal L_{b,\Sigma}^{\pi_\omega}$ by the present LQR generator. For the basis
$$
\phi_1(s,\mu)=1,\qquad \phi_2(s,\mu)=\bar\mu^2,\qquad \phi_3(s,\mu)=(s-\bar\mu)^2,
$$
one obtains
$$
\mathcal L_{b,\Sigma}^{\pi_\omega}\phi_1=0,\qquad \mathcal L_{b,\Sigma}^{\pi_\omega}\phi_2=2B\omega_2\,\bar\mu^2,\qquad \mathcal L_{b,\Sigma}^{\pi_\omega}\phi_3=2(A+B\omega_1)(s-\bar\mu)^2+\gamma^2.
$$  
At each actor step, the identities above are inserted into the Galerkin system \eqref{eq:linear-system-policy-evaluation}, with $A(\omega)$ and $b(\omega)$ defined in \eqref{eq:definition_A_B_galerkin}, to compute the critic coefficients $\theta^{(k)}$ and then update the actor through the policy-gradient direction in \eqref{eq:general_empirical_gradient}. From a computational viewpoint, in this benchmark the critic space has dimension three, so policy evaluation reduces to solving a $3\times3$ linear system. Moreover, the Gaussian structure of the policy allows the conditional expectation over the action variable in the actor update to be evaluated analytically. The remaining numerical cost comes from the empirical approximation of the discounted occupancy measure along simulated trajectories. Since the closed-loop dynamics are linear and Gaussian for each fixed actor parameter, these trajectories are sampled from the exact transition law, while the population mean is propagated through its deterministic closed equation.

The closed-form Riccati solution provides a benchmark for the optimal feedback and the optimal value function. \Cref{fig:value_LQR} compares the final learned policy mean and its induced value function with the Riccati benchmark for several fixed values of the population mean $\bar\mu$. \Cref{fig:actor_critic_training_LQR} reports the evolution of the critic and actor coefficients during training. The convergence of both sets of coefficients toward the Riccati values confirms that the HJB--Galerkin actor--critic iteration recovers the quadratic structure of the optimal solution in this benchmark example.

We also implement the undiscounted empirical measure $\overline{\mathfrak m}$ in~\eqref{eq:uni_weighted_measure}, which can be viewed as a biased empirical approximation of the discounted occupancy measure in~\eqref{eq:approx-gradient-direction}, since it assigns uniform rather than discounted weights along the finite time window. The corresponding results, shown in \Cref{fig:actor_critic_training_LQR_uni} with stepsize $\alpha=1\times 10^{-1}$, illustrate the same qualitative behavior as in the discounted case.

\begin{figure}
    \centering
    \includegraphics[width=0.8\linewidth]{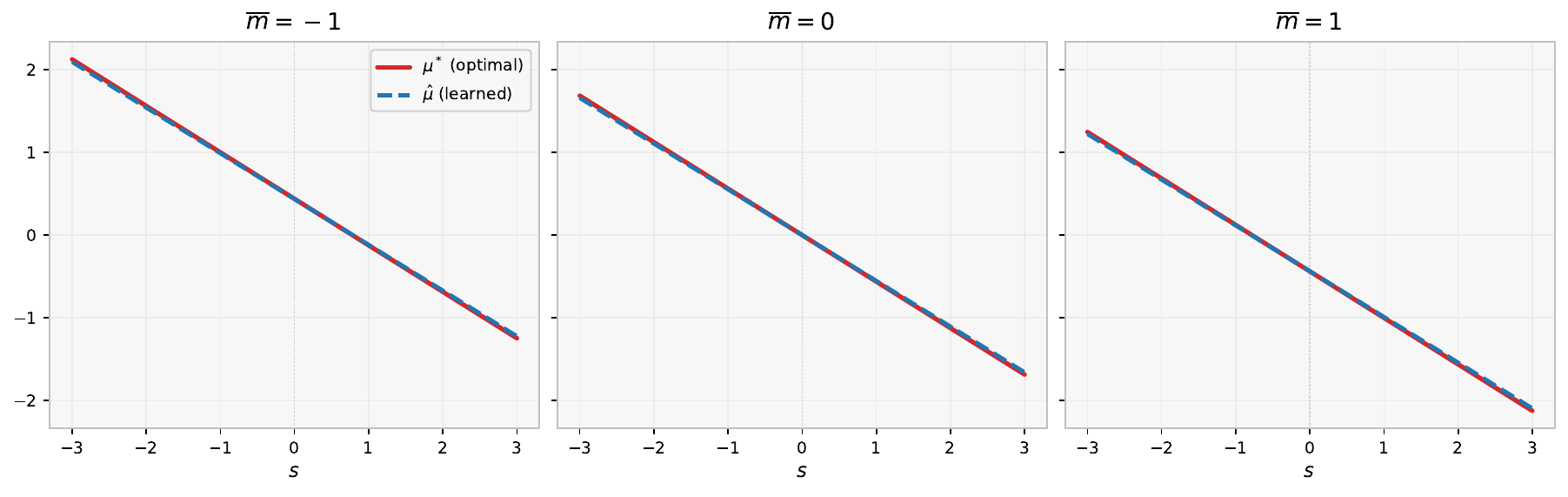}
    \includegraphics[width=0.8\linewidth]{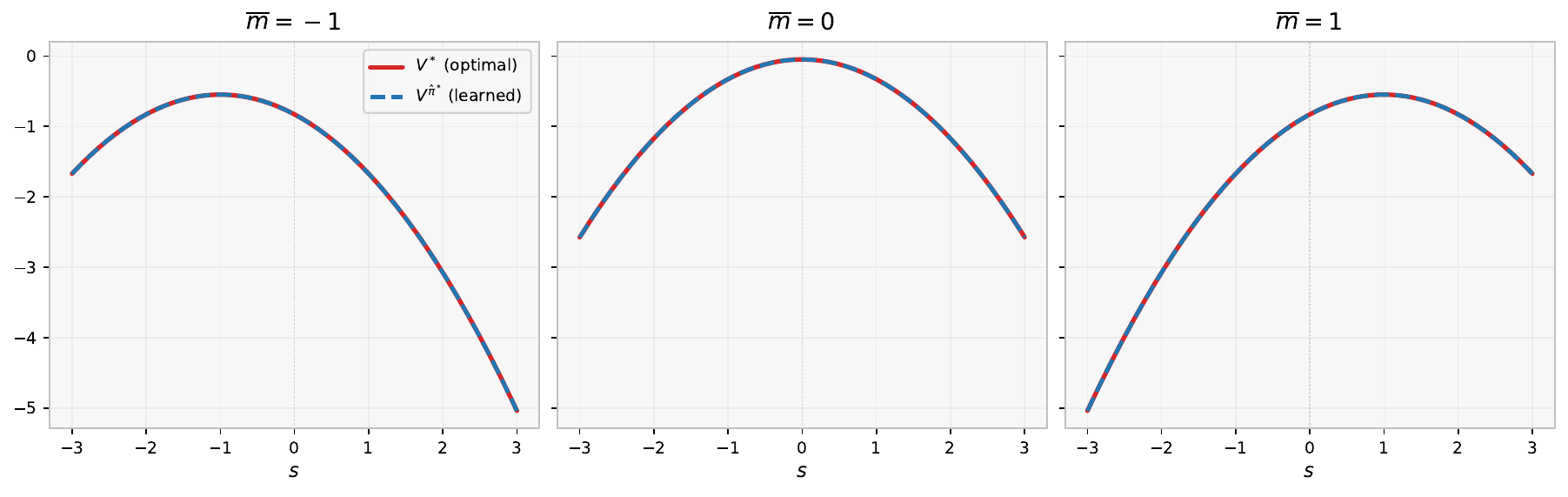}
    \caption{Comparison between the learned and optimal solutions for different values of the population mean $\bar\mu$. Top row: policy mean induced by the learned actor and optimal Riccati policy mean. Bottom row: value function associated with the learned policy and optimal Riccati value function.}
    \label{fig:value_LQR}
\end{figure}

\begin{figure}
    \centering
    \includegraphics[width=0.4\linewidth]{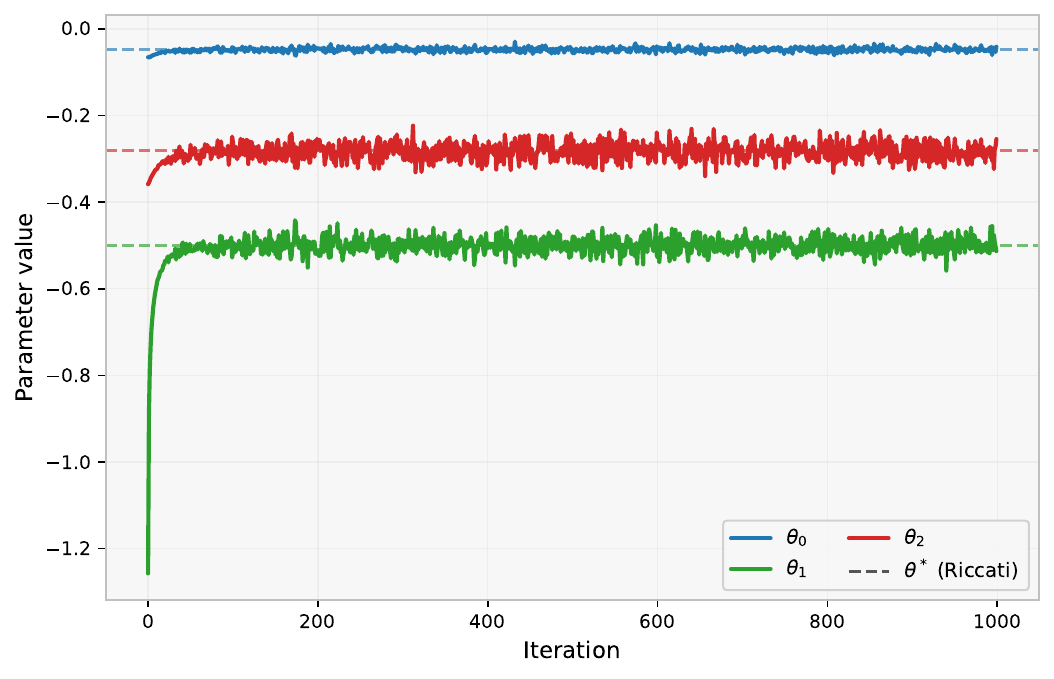}
    \includegraphics[width=0.4\linewidth]{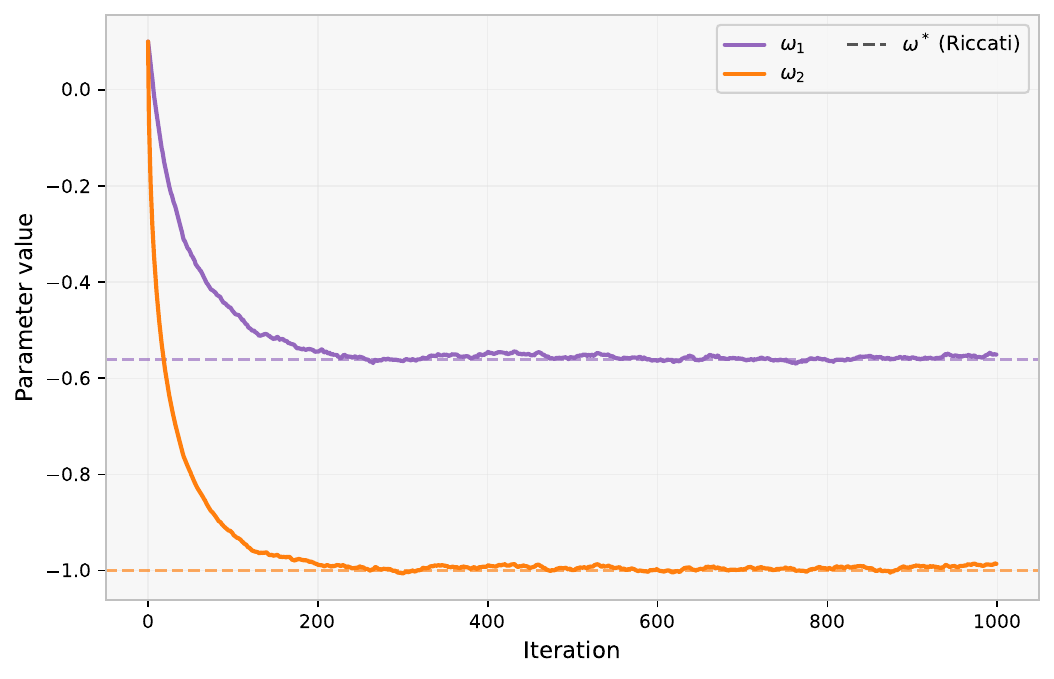}
    \caption{Convergence of the finite-dimensional parameters in the scalar mean-field LQR example. Left: Galerkin critic coefficients $\theta_k$ compared with the Riccati coefficients $\theta^*$. Right: actor coefficients $\omega_k$ compared with the optimal Riccati feedback coefficients $\omega^*$.}
    \label{fig:actor_critic_training_LQR}
\end{figure}

\begin{figure}
    \centering
    \includegraphics[width=0.4\linewidth]{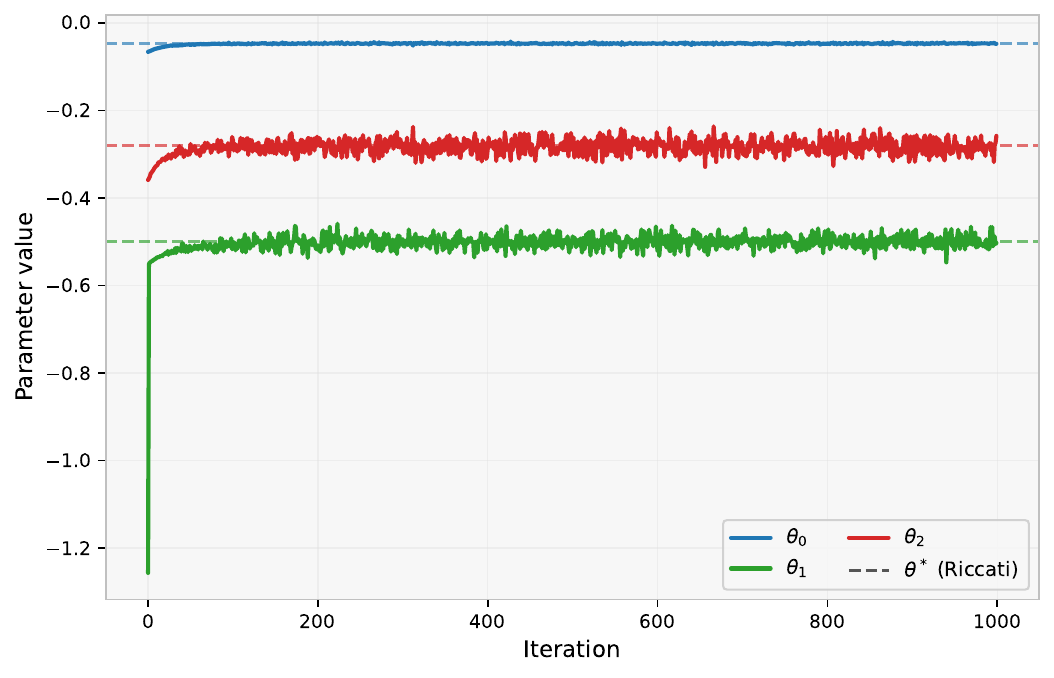}
    \includegraphics[width=0.4\linewidth]{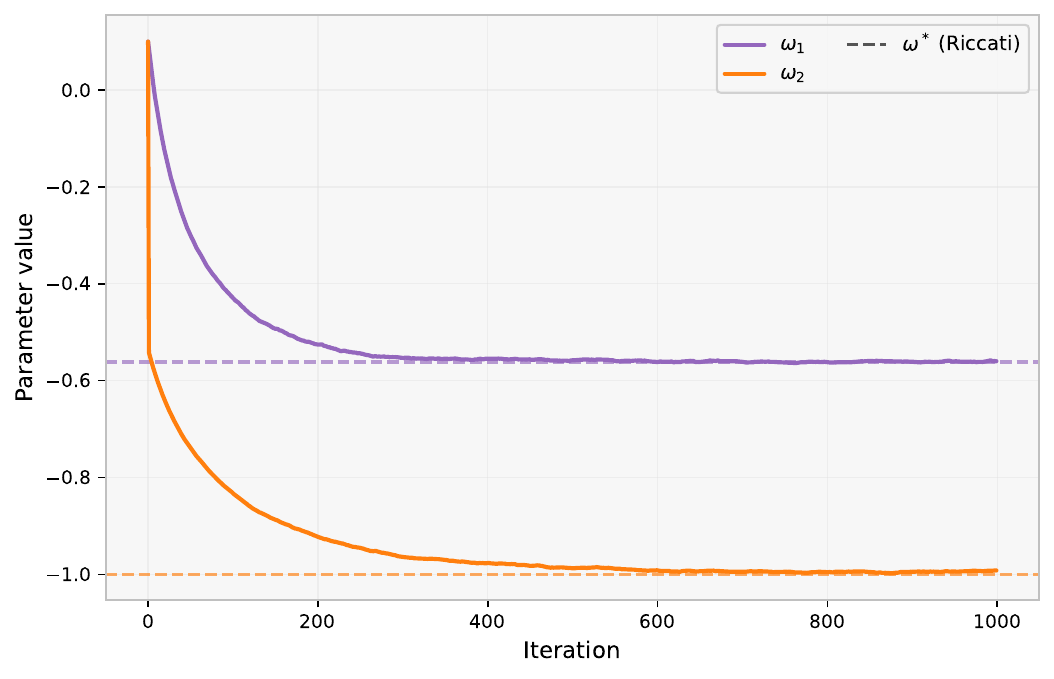}
    \caption{Convergence of the finite-dimensional parameters in the scalar mean-field LQR example with undiscounted measure $\overline{\mathfrak m}$. Left: Galerkin critic coefficients $\theta_k$ compared with the Riccati coefficients $\theta^*$. Right: actor coefficients $\omega_k$ compared with the optimal Riccati feedback coefficients $\omega^*$.}
    \label{fig:actor_critic_training_LQR_uni}
\end{figure}


\subsection{Example 2: Crowd-aversion transport}
The second example illustrates the actor--critic method in a nonlinear mean-field setting for which no closed-form Riccati benchmark is available. The objective is to transport a representative agent toward a prescribed target while penalizing trajectories that pass through regions of high crowd density. The crowd is described by a population law $\mu_t$, which is approximated numerically by an empirical measure generated by a finite particle system.

\subsubsection{Problem setup}

We consider the controlled diffusion on $\R^2$
\begin{equation*}
\d s_t = a_t\,\d t+\sigma\,\d B_t,
\qquad
a_t\sim\pi_\omega(\cdot\mid s_t,\mu_t),
\end{equation*}
where $a_t\in\R^2$ is the control, $\sigma>0$ is a constant volatility, and $\mu_t$ denotes the population distribution observed by the representative agent. The reward is defined by
\begin{equation}
\label{eq:running_cost_crowd_transport}
r(s,\mu,a)
=
-\left(
\frac12\|a\|^2
+\frac{\kappa}{2}\|s-s_{\mathrm{tar}}\|^2
+\gamma\,(K_\eta*\mu)(s)
+\frac{\rho}{2}\|s\|^2
\right),
\end{equation}
where $s_{\mathrm{tar}}\in\R^2$ is the target location, $\kappa>0$ controls the strength of the attraction to the target, $\gamma>0$ controls the crowd-aversion penalty, and $\rho\ge0$ is a confinement parameter. The crowd-density term is computed through the Gaussian kernel convolution
\begin{equation*}
(K_\eta*\mu)(s)
:=
\int_{\R^2}
\exp\left(-\frac{\|s-s'\|^2}{2\eta^2}\right)\,\mu(\d s'),
\end{equation*}
with bandwidth $\eta>0$. The goal is to maximize the corresponding infinite-horizon entropy-regularized objective \eqref{eq:value_fn} with reward given by~\eqref{eq:running_cost_crowd_transport}.

\begin{remark}
The four terms in \eqref{eq:running_cost_crowd_transport} have distinct roles. The first one penalizes large controls, the second one attracts the representative agent toward $s_{\mathrm{tar}}$, the third one penalizes positions located in high-density regions of the crowd, and the last one adds a mild confinement on the unbounded domain.
\end{remark}

\subsubsection{Policy parametrization and value approximation}
\label{subsec:policy_value_crowd_transport}

We use a Gaussian policy on $\R^2$ with fixed exploration variance,
\begin{equation*}
\pi_\omega(\cdot\mid s,\mu)
=
\mathcal N\bigl(m_\omega(s,\mu),\sigma_a^2 I_2\bigr),
\end{equation*}
where the mean is parametrized as $m_\omega(s,\mu)=Wf(s,\mu)$, with $\omega\equiv W\in\R^{2\times d_f}$. In order to generate trajectories without an Euler--Maruyama approximation, we use an affine feature map depending on the state and on the empirical population mean,
\begin{equation*}
f(s,\mu)
=
\bigl(
1,\,
s_1,\,
s_2,\,
s_1-s_{\mathrm{tar},1},\,
s_2-s_{\mathrm{tar},2},\,
\bar\mu_1,\,
\bar\mu_2
\bigr),
\qquad
\bar\mu:=\int_{\R^2} x\,\mu(\d x).
\end{equation*}
With this choice, the finite-particle approximation of the aggregated representative--population system is an affine linear diffusion. Hence, at each actor--critic iteration, the training trajectories are sampled from the exact Gaussian transition of this affine system, computed through the corresponding matrix exponential.

The critic is constructed as in the Galerkin policy-evaluation procedure of \Cref{subsec:galerkin-policy-evaluation}. To represent the dependence on the population law, we use kernel-embedding features
\begin{equation*}
u_j(\mu)
:=
\int_{\R^2}
\exp\left(-\frac{\|x-c_j\|^2}{2h^2}\right)\,\mu(\d x),
\qquad
j=1,\dots,M_\mu,
\end{equation*}
where $\{c_j\}_{j=1}^{M_\mu}\subset\R^2$ are fixed centers and $h>0$ is the bandwidth. For the state variable we use Gaussian radial basis functions $\psi_m(s)=\exp(-\|s-z_m\|^2/(2s_x^2))$, $m=1,\dots,M_s$. The critic basis is then given by
\begin{equation*}
\phi_{m,0}(s,\mu)=\psi_m(s),
\qquad
\phi_{m,j}(s,\mu)=\psi_m(s)\,u_j(\mu),
\qquad
m=1,\dots,M_s,\quad j=1,\dots,M_\mu.
\end{equation*}
Thus, with $n=M_s(M_\mu+1)$, the value function is approximated in $\mathbb V_n=\mathrm{span}\{\phi_1,\dots,\phi_n\}$ by $V_n^{\pi_\omega}(s,\mu)=\theta^\top\Phi(s,\mu)$.

\subsubsection{Implementation and numerical results}

The implementation follows the actor--critic algorithm of \Cref{sec:policy-iteration-algorithm}. For each actor parameter $W$, the critic is obtained by solving the Galerkin system \eqref{eq:linear-system-policy-evaluation}, with $A(W)$ and $b(W)$ defined as in \eqref{eq:definition_A_B_galerkin}, after replacing the abstract generator $\mathcal L_{b,\Sigma}^{\pi_\omega}$ by the known generator of the crowd-aversion dynamics. The resulting critic $V_n^{\pi_\omega}$ is then inserted into the empirical policy-gradient formula \eqref{eq:general_empirical_gradient} to update the actor.

The numerical parameters are
\begin{equation*}
\sigma=0.2,\qquad
\beta=0.1,\qquad
\kappa=1000,\qquad
\gamma=10,\qquad
\eta=0.8,\qquad
\rho=0.1.
\end{equation*}
The entropy parameter is $\lambda=2$, the exploration standard deviation is $\sigma_a=0.25$, and the target is $s_{\mathrm{tar}}=(2,0)$. In the reported experiment, both the initial population law and the initial representative law are centered near $(-2,0)$, with Gaussian perturbations of standard deviation $0.02$. The implementation also allows these two initial laws to be chosen independently.

The population law is approximated by an empirical measure
\begin{equation}\label{eq:empirical_estimation_mu_N}
\widehat\mu_t^N
=
\frac1N\sum_{i=1}^N\delta_{X_t^i}.
\end{equation}
Accordingly, the crowd-density term is evaluated as
\begin{equation*}
(K_\eta*\widehat\mu_t^N)(s_t)
=
\frac1N\sum_{i=1}^N
\exp\left(-\frac{\|s_t-X_t^i\|^2}{2\eta^2}\right).
\end{equation*}
From an implementation viewpoint, the main additional cost in this crowd-aversion example comes from the particle approximation of the population law. Since $\widehat\mu_t^N$ is represented by the empirical measure in \eqref{eq:empirical_estimation_mu_N}, several realizations of the coupled representative--population dynamics must be generated to assemble the Galerkin system and evaluate the policy-gradient direction. Nevertheless, the implementation exploits the affine structure of the actor with respect to the state and the empirical population mean, which allows us to sample the coupled system from its exact Gaussian transition.

At each actor--critic iteration we generate $L=24$ independent simulations, each with $N=64$ particles representing the crowd. The discounted objective is truncated at $T=5$ and evaluated on the time grid $\Delta t=0.05$. We run $500$ actor updates with constant stepsize $\alpha=5\times 10^{-4}$. The actor matrix is initialized at zero, its entries are clipped for stability, and the Galerkin system is solved with ridge parameter $10^{-2}$. The RBF centers for both the state and measure features are placed on a $5\times5$ Cartesian grid on $[-8,4]^2$. State and particle locations are clipped only when evaluating the RBF features; the simulated dynamics themselves are not clipped.

Unlike the LQR example, this nonlinear crowd-aversion problem does not admit a closed-form optimal policy or value function. We therefore assess the method through diagnostics that measure stability and qualitative performance. \Cref{fig:crowd_aversion_stability} reports the estimated regularized reward and the norms of the actor, critic, and policy-gradient direction along training. \Cref{fig:crowd_aversion_trajectories} shows the trajectories generated by the final learned policy. The reward curve provides a numerical indication of policy improvement, while the parameter and gradient norms monitor the stability of the actor--critic iteration. The trajectory plot illustrates whether the learned policy transports the representative agent from the initial region near $(-2,0)$ toward the target while accounting for the crowd-aversion term. Finally, \Cref{fig:crowd_aversion_stability_undiscounted} illustrate the convergence when the undiscounted measure $\bar{\mathrm{m}}$, defined in \eqref{eq:uni_weighted_measure}, is considered.  Note that its qualitative behavior is similar to the discounted case shown in \Cref{fig:crowd_aversion_stability}.

\begin{figure}
    \centering
    \includegraphics[width=0.4\linewidth]{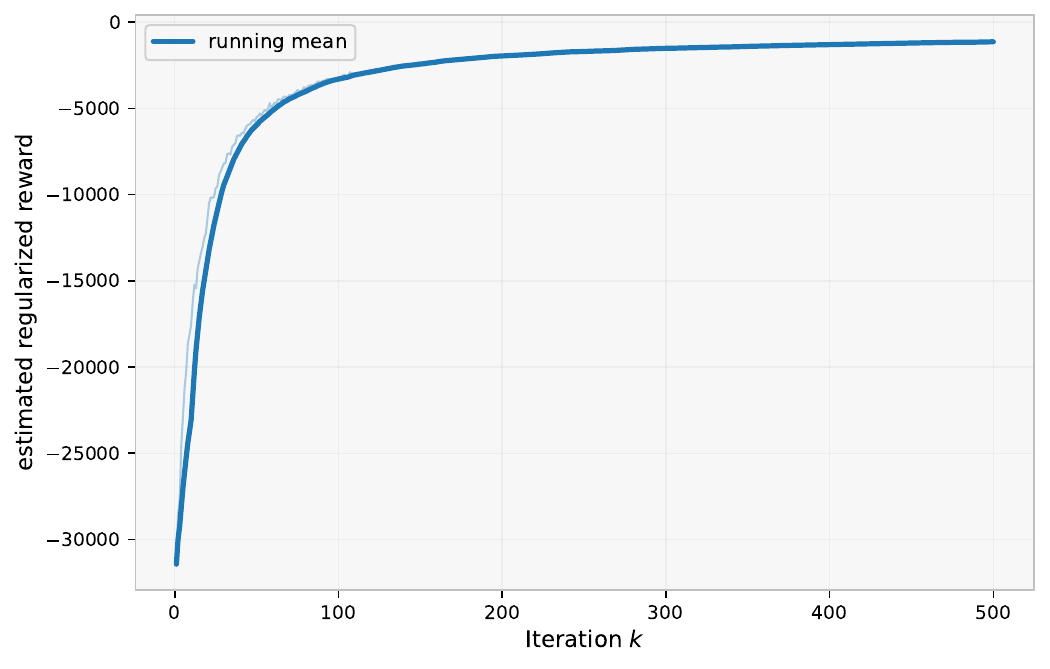}
    \includegraphics[width=0.4\linewidth]{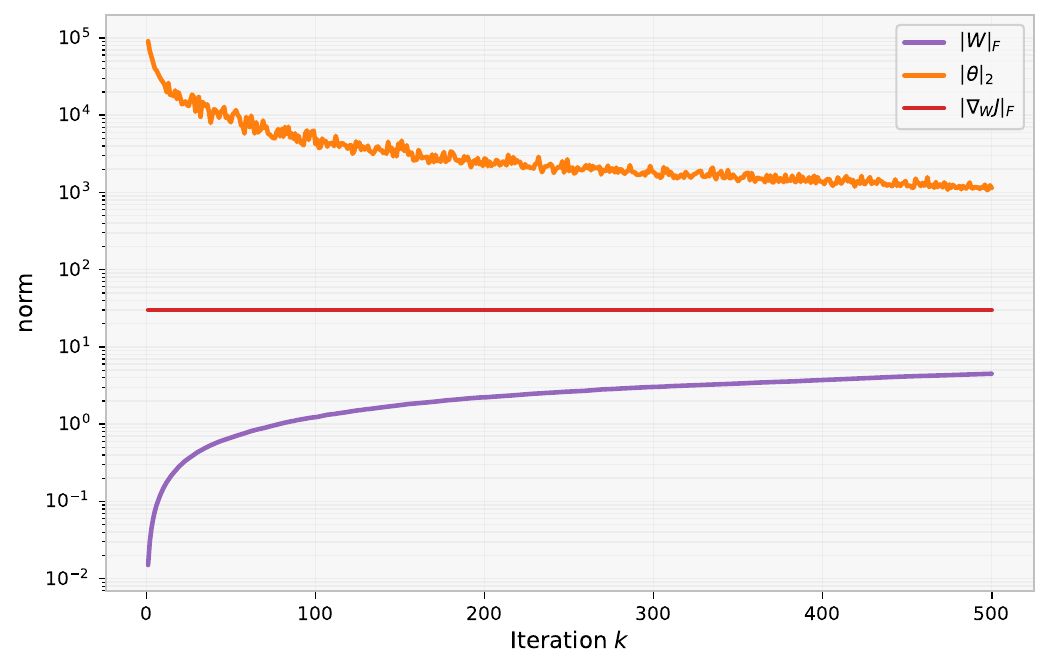}
    \caption{Training diagnostics for the crowd-aversion transport example. Left: estimated entropy-regularized reward along the actor--critic iterations. Right: norms of the actor parameters, critic coefficients, and policy-gradient direction.}
    \label{fig:crowd_aversion_stability}
\end{figure}

\begin{figure}
    \centering
    \includegraphics[width=0.4\linewidth]{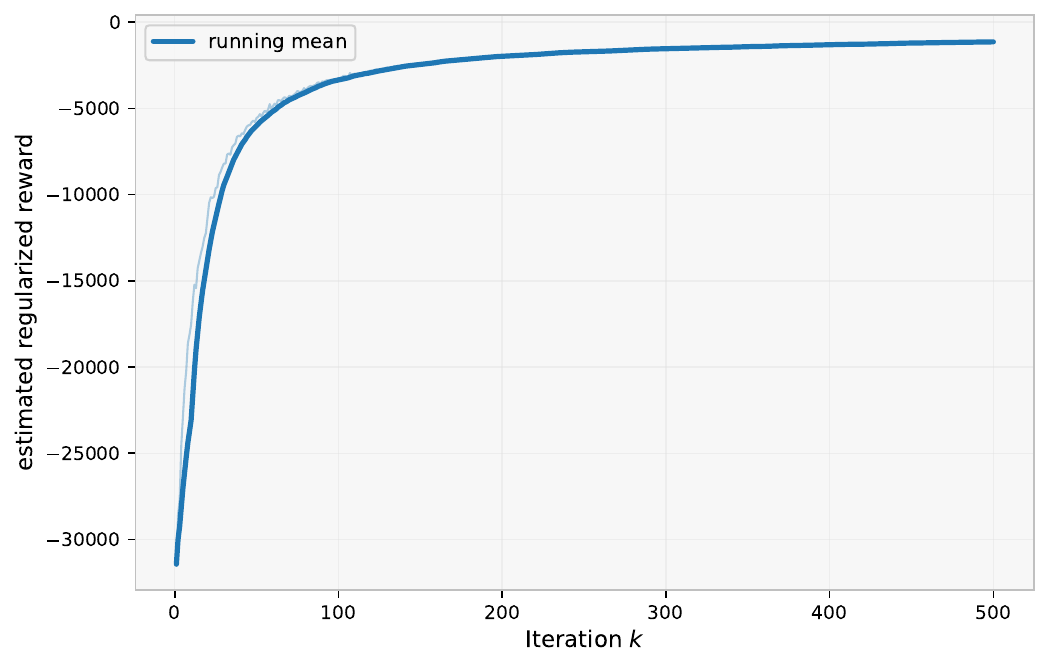}
    \includegraphics[width=0.4\linewidth]{images_paper_1/crowd_parameter_norms_known_dynamics_undiscounted.pdf}
    \caption{Training diagnostics for the crowd-aversion transport example using the undiscounted occupancy measure. Left: estimated entropy-regularized reward along the actor--critic iterations. Right: norms of the actor parameters, critic coefficients, and policy-gradient direction.}\label{fig:crowd_aversion_stability_undiscounted}
\end{figure}

\begin{figure}[H]
    \centering
    \includegraphics[width=0.4\linewidth]{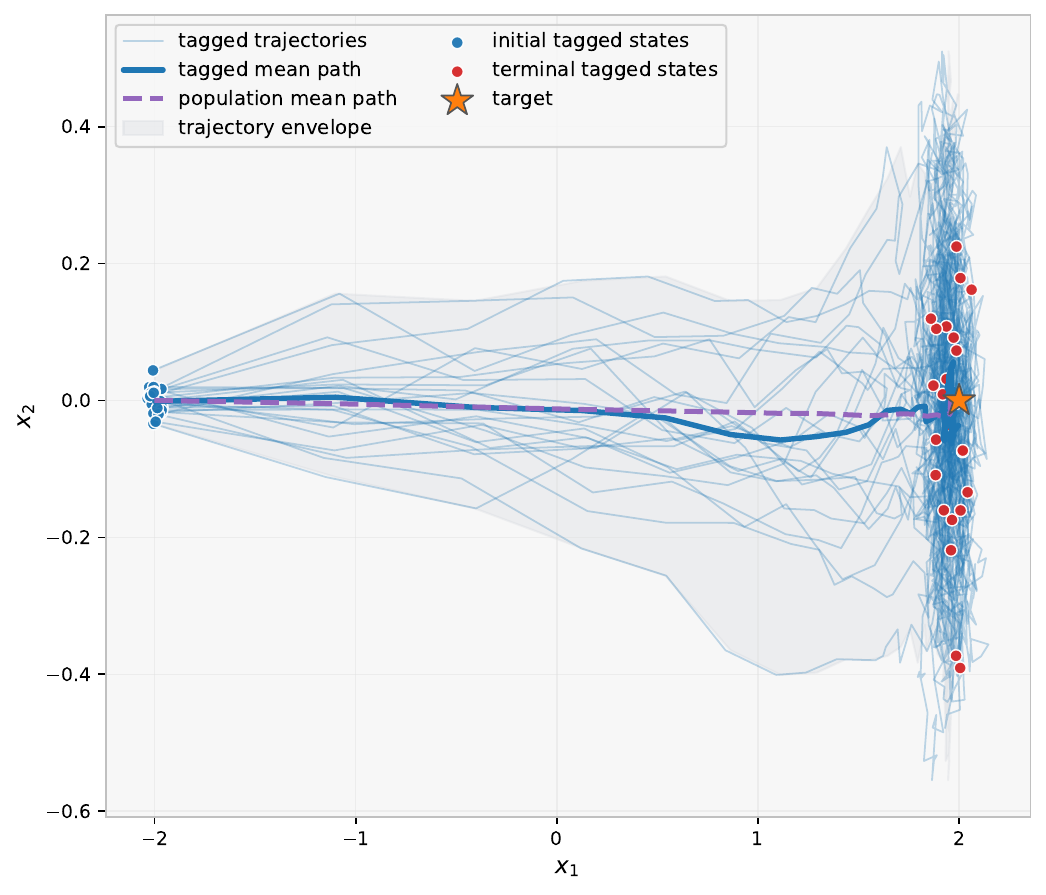}
    \caption{Simulated trajectories after training. The plot shows representative trajectories under the learned policy, together with the empirical population mean path and the target location.}\label{fig:crowd_aversion_trajectories}
\end{figure}

\section{Acknowledgements}
E. Bayraktar is supported in part by the NSF grants DMS-2507940, and 2406232, and in part by the Susan M. Smith Professorship. Y. Zhu and M. HERNANDEZ is supported in part by the NSF grants No 2529107.




\appendix

 \appendix
\section{Proofs of the Main Results}\label{proof_of_results}

\subsection{Proofs for \Cref{sec:HJB}}\label{APP:PROOFS_SEC_HJB}

We first establish well-posedness of the system together with moment
estimates that are explicit in time. Estimates of this type are
standard in the literature; see, for instance, \cite{BayraktarCossoPham2018,BuckdahnLiPengRainer2017}. However, for the infinite-horizon
argument we need bounds with an explicit exponential dependence on $t$, so that they can be combined with the discount factor $e^{-\beta t}$ when passing to the limit as $T\to\infty$. A generic finite-horizon
bound of the form $C_T$, with no explicit control of its dependence on
$T$, would not be sufficient for this purpose. We therefore state and
prove the corresponding estimates for general $p$-moments.

\begin{lemma}[Well-posedness and exponential moment growth]
  \label[lemma]{lem:moment-growth-clean}
  Suppose Assumption \ref{A1} holds.
  Fix an even integer $p \geq 2$.
  Then, for every $(s,\mu) \in \mathbb{R}^d \times
  \mathcal{P}_p(\mathbb{R}^d)$, the system
  \eqref{eq:rep_SDE}-\eqref{eq:MKV_SDE} admits a unique strong
  solution on $[0,\infty)$. Moreover, there exist constants
  $C_p \geq 1$ and $\beta_0(p) > 0$, depending only on $p$ and the
  structural constants in Assumption \ref{A1}, such that for all
  $t \geq 0$,
  \begin{equation}\label{eq:moment-exp}
    \mathbb{E}\!\left[|s_t|^p\right]
    + \mathbb{E}\!\left[|\tilde{s}_t|^p\right]
    \leq C_p\,e^{\beta_0(p)\,t}
    \bigl(1 + |s|^p + m_p(\mu)\bigr).
  \end{equation}
  In particular, for $p = 4$, for each $T > 0$ there exists $C_T > 0$
  such that
  \begin{align}\label{eq:finite_T_estimation_s}
    \sup_{0 \leq t \leq T}
    \Bigl(
    \mathbb{E}^{s,\mu,\pi}\!\left[|s_t|^4\right]
    + \mathbb{E}^{s,\mu,\pi}\!\left[|\tilde{s}_t|^4\right]
    \Bigr)
    \leq C_p e^{\beta_0(p)T}\bigl(1 + |s|^4 + m_4(\mu)\bigr),
  \end{align}
  and $\sup_{0 \leq t \leq T} m_2(\mu_t) \leq C_T(1 + m_4(\mu)^{1/2})$.
\end{lemma}

\begin{proof}

\noindent\textbf{Step 1 (Existence and uniqueness).}
Under Assumption \ref{A1}, the McKean-Vlasov SDE \eqref{eq:MKV_SDE}
has globally Lipschitz and linearly growing coefficients
$b^\pi$ and $\sigma^\pi$; see Remark \ref{rem:linear_growth}.
Hence, by the standard well-posedness theory for McKean-Vlasov SDEs,
\eqref{eq:MKV_SDE} admits a unique strong solution $\tilde s$ for
every initial law $\mu \in \mathcal{P}_2(\mathbb{R}^d)$.
The population flow $\mu_t := \mathcal{L}(\tilde s_t)$ is therefore
well defined.
Conditioning on the deterministic flow $(\mu_t)_{t\ge0}$, the
representative equation \eqref{eq:rep_SDE} is an SDE with
progressively measurable coefficients, globally Lipschitz in $s$ and
with linear growth, so it also admits a unique strong solution.

\smallbreak
\noindent\textbf{Step 2 ($p$-moment bound for $\tilde s_t$).}
Let $f(x)=|x|^p$. Since $p\ge2$ is even, $f\in C^2(\mathbb{R}^d)$ and
\[
  \nabla f(x)=p|x|^{p-2}x,
  \qquad
  D^2f(x)=p(p-2)|x|^{p-4}xx^\top+p|x|^{p-2}I_d .
\]
For $n\ge1$, define the stopping time $\tau_n:=\inf\{t\ge0:\,|\tilde s_t|>n\}.$
Applying It\^o's formula to $f(\tilde s_t)$ on $[0,t\wedge\tau_n]$
gives
\begin{align}\label{eq:ito_tothe_powerp}
  |\tilde{s}_{t \wedge \tau_n}|^p
  = |\tilde{s}_0|^p
  + \int_0^{t \wedge \tau_n}
    \Bigl[
      \nabla f(\tilde{s}_u) \cdot b^\pi(\tilde{s}_u,\mu_u)
      + \tfrac{1}{2} D^2 f(\tilde{s}_u) : \Sigma^\pi(\tilde{s}_u,\mu_u)
    \Bigr] \d u
  + M_{t \wedge \tau_n},
\end{align}
where $ M_t$ is a local martingale,
hence $\mathbb E[M_{t\wedge\tau_n}]=0$. Using \eqref{eq:linear_growth_bpi}, Cauchy-Schwarz, and Young's
inequality, there exists a constant $C_p>0$, depending only on
$p$ and $K_\pi$, such that for all $(x,\mu)\in\mathbb{R}^d\times
\mathcal{P}_2(\mathbb{R}^d)$,
\begin{equation}\label{eq:estima_f_nabla-D2_f}
  \begin{aligned}
    |\nabla f(x)\cdot b^\pi(x,\mu)|
    \le C_p\bigl(1+|x|^p+m_2(\mu)^{p/2}\bigr),\quad
    D^2f(x):\Sigma^\pi(x,\mu)
    \le C_p\bigl(1+|x|^p+m_2(\mu)^{p/2}\bigr).
  \end{aligned}
\end{equation}
Taking expectations in \eqref{eq:ito_tothe_powerp} yields
\[
  \mathbb E\!\left[|\tilde s_{t\wedge\tau_n}|^p\right]
  \le m_p(\mu)
  + C_p\int_0^t
    \Bigl(
      1+\mathbb E\!\left[|\tilde s_{u\wedge\tau_n}|^p\right]
      + m_2(\mu_u)^{p/2}
    \Bigr)\,\d u.
\]
Set $ y_n(t):=\mathbb E\!\left[|\tilde s_{t\wedge\tau_n}|^p\right]$, $y(t):=\mathbb E\!\left[|\tilde s_t|^p\right],$ and  $z(t):=\sup_{n\ge1} y_n(t).$ Since $\mu_u=\mathcal{L}(\tilde s_u)$, H\"older's inequality gives
\[
  m_2(\mu_u)^{p/2}
  =\Bigl(\mathbb E[|\tilde s_u|^2]\Bigr)^{p/2}
  \le \mathbb E[|\tilde s_u|^p]
  = y(u).
\]
Therefore, for every $n\ge1$,
\begin{align}\label{eq:diff_estimation_stild}
  y_n(t)
  \le m_p(\mu)
  + C_p\int_0^t \bigl(1+y_n(u)+y(u)\bigr)\,\d u.
\end{align}
Now, for each fixed $u\ge0$, Fatou's lemma implies $y(u)
  = \mathbb E\!\left[\lim_{n\to\infty} |\tilde s_{u\wedge\tau_n}|^p\right]
  \le \liminf_{n\to\infty} y_n(u)
  \le z(u).$ Hence \eqref{eq:diff_estimation_stild} implies
\[
  y_n(t)
  \le m_p(\mu) + C_p\int_0^t \bigl(1+2z(u)\bigr)\,\d u,
  \qquad n\ge1.
\]
Taking the supremum over $n$ gives
\[
  z(t)
  \le m_p(\mu) + C_p\int_0^t \bigl(1+2z(u)\bigr)\,\d u.
\]
By Gr\"onwall's inequality, $z(t)\le \bigl(m_p(\mu)+1\bigr)e^{2C_p t}$ for $t\ge0$. Since $y(t)\le z(t)$, we conclude that
\begin{align}\label{eq:estimation_tilds_moment}
  \mathbb E\!\left[|\tilde s_t|^p\right]
  \le \bigl(m_p(\mu)+1\bigr)e^{\beta_0 t},
  \qquad \beta_0:=2C_p.
\end{align}

\smallbreak
\noindent\textbf{Step 3 ($p$-moment bound for $s_t$).}
Define $ \sigma_n:=\inf\{t\ge0:\,|s_t|>n\},$ and $ v_n(t):=\mathbb E\!\left[|s_{t\wedge\sigma_n}|^p\right].$ Repeating the same localized It\^o argument for \eqref{eq:rep_SDE},
and using \eqref{eq:estimation_tilds_moment} to control
$m_2(\mu_t)^{p/2}$, we obtain
\begin{align*}
  v_n(t)
  \le |s|^p
  + C_p\int_0^t \bigl(1+v_n(u)\bigr)\,\d u
  + C_p\int_0^t \mathbb E\!\left[|\tilde s_u|^p\right]\,\d u.
\end{align*}
Using \eqref{eq:estimation_tilds_moment}, it follows that
\[
  v_n(t)
  \le |s|^p
  + C_p\int_0^t \bigl(1+v_n(u)\bigr)\,\d u
  + C_p\bigl(m_p(\mu)+1\bigr)\int_0^t e^{\beta_0 u}\,\d u .
\]
Another application of Gr\"onwall's inequality yields
\[
  v_n(t)
  \le C_p e^{\beta_0 t}\bigl(1+|s|^p+m_p(\mu)\bigr),
  \qquad t\ge0,
\]
after enlarging $C_p$ if necessary. Letting $n\to\infty$ and using
Fatou's lemma gives
\begin{align}\label{eq:moment_bound_st}
  \mathbb E\!\left[|s_t|^p\right]
  \le C_p e^{\beta_0 t}\bigl(1+|s|^p+m_p(\mu)\bigr).
\end{align}

\smallbreak
\noindent\textbf{Step 4 (Conclusion).}
Combining \eqref{eq:estimation_tilds_moment} and
\eqref{eq:moment_bound_st}, and enlarging the constant if necessary,
we obtain \eqref{eq:moment-exp}.
For $p=4$, the same argument gives the finite-horizon estimate
\eqref{eq:finite_T_estimation_s}. Finally, for $t\in[0,T]$,
\[
  m_2(\mu_t)
  = \mathbb E[|\tilde s_t|^2]
  \le \Bigl(\mathbb E[|\tilde s_t|^4]\Bigr)^{1/2}
  \le C_T\bigl(1+m_4(\mu)^{1/2}\bigr),
\]
which completes the proof.
\end{proof}

\begin{remark}[Explicit value of $\beta_0(2)$]\label{remark_explicit_beta_02}
  Taking $p = 2$ in Lemma \ref{lem:moment-growth-clean} and applying
  Young's inequality to \eqref{eq:estima_f_nabla-D2_f}, using
  $(u+v+w)^2 \leq 3(u^2+v^2+w^2)$ and the bound
  \eqref{eq:linear_growth_bpi}, one obtains
  \[
    2s \cdot b^\pi(s,\mu) \leq K_\pi + 4K_\pi|s|^2 + K_\pi\,m_2(\mu),
    \qquad
    \|\sigma^\pi(s,\mu)\|^2 \leq 3K_\pi^2\bigl(1 + |s|^2 + m_2(\mu)\bigr),
  \]
  from which \eqref{eq:diff_estimation_stild} with $p = 2$ becomes
  \[
    \frac{\d}{\d t}\mathbb{E}[|\tilde{s}_t|^2]
    \leq (K_\pi + 3K_\pi^2) + (5K_\pi + 6K_\pi^2)\,\mathbb{E}[|\tilde{s}_t|^2].
  \]
  Therefore $\beta_0 := \beta_0(2) = 5K_\pi + 6K_\pi^2$.
\end{remark}

\begin{remark}[A dissipative case]
  \label{rem:dissipative_second_moment}
  Assume that there exist constants $\kappa>C_{\mathrm{diss}}>0$ and
  $C_0\ge0$ such that, for all
  $(s,\mu)\in\mathbb R^d\times\mathcal P_2(\mathbb R^d)$, \eqref{eq:dissipativity_remark} holds. 
  Applying It\^o's formula to $|\tilde s_t|^2$, one obtains 
  \[
    \frac{\d}{\d t}m_2(\mu_t)
    \le C_0-(\kappa-C_{\mathrm{diss}})\,m_2(\mu_t),
  \]
  hence
  \[
    m_2(\mu_t)
    \le e^{-(\kappa-C_{\mathrm{diss}})t}m_2(\mu)
    +\frac{C_0}{\kappa-C_{\mathrm{diss}}}
      \bigl(1-e^{-(\kappa-C_{\mathrm{diss}})t}\bigr),
    \qquad t\ge0.
  \]
  In particular, $m_2(\mu_t)$ is uniformly bounded on $[0,\infty)$.
  Applying It\^o's formula to $|s_t|^2$ and using the previous bound, we deduce that there exists  $C>0$ independent of $t$ such that
  \[
    \sup_{t\ge0}\mathbb E^{s,\mu,\pi}[|s_t|^2]
    \le C\bigl(1+|s|^2+m_2(\mu)\bigr).
  \]
\end{remark}

We next establish differentiability properties of the decoupled flow
with respect to the state and measure variables. While such properties
are standard in the McKean-Vlasov literature, for our purposes we also
need explicit exponential-in-time bounds for the corresponding
derivatives. As noted before
Lemma \ref{lem:moment-growth-clean}, these explicit bounds are
essential for the infinite-horizon analysis.

\begin{lemma}[Exponential bounds for the derivatives of the decoupled flow]\label[lemma]{lem:variational-growth}
  Suppose Assumption \ref{A1} holds.
  Fix an even integer $p\ge2$.
  For $\mu\in\mathcal{P}_2(\mathbb{R}^d)$, let $(\mu_t)_{t\ge0}$ be the
  measure flow generated by \eqref{eq:MKV_SDE}. For each
  $x\in\mathbb{R}^d$, let $X_t^{x,\mu}$ denote the unique strong
  solution of
  \begin{align*}
    \d X_t^{x,\mu}
    = b^\pi(X_t^{x,\mu},\mu_t)\,\d t
    + \sigma^\pi(X_t^{x,\mu},\mu_t)\,\d B_t,
    \qquad
    X_0^{x,\mu}=x.
  \end{align*}
  Then, for every $\xi\in\mathbb{R}^d$, the derivatives
  $\nabla_x X_t^{x,\mu}$, $D_{xx}^2X_t^{x,\mu}$,
  $\partial_\mu X_t^{x,\mu}(\xi)$, and
  $D_\xi\partial_\mu X_t^{x,\mu}(\xi)$ exist. Moreover, there exists a
  constant $C_{p,d}>0$, depending only on $p$ and $d$, such that with $\beta_{\mathrm{var}}(p):=C_{p,d}K_\pi^2,$ one has, for all $t\ge0$, $x\in\mathbb{R}^d$,
  $\mu\in\mathcal{P}_2(\mathbb{R}^d)$, and $\xi\in\mathbb{R}^d$,
  \begin{align}\label{eq:variational-growth}
    \mathbb{E}\!\left[
      \|\nabla_x X_t^{x,\mu}\|^p
      + \|D^2_{xx}X_t^{x,\mu}\|^p
      + \|\partial_\mu X_t^{x,\mu}(\xi)\|^p
      + \|D_\xi\partial_\mu X_t^{x,\mu}(\xi)\|^p
    \right]
    \le C_{p,d}(1+K_\pi^p)e^{\beta_{\mathrm{var}}(p)t}.
  \end{align}
\end{lemma}

\begin{proof}
Under Assumption \ref{A1}, the coefficients $b^\pi$ and $\sigma^\pi$ are twice differentiable with respect to the state variable and Lions differentiable with respect to the measure argument, with bounded derivatives up to second order. Therefore, the differentiability results of \cite[Theorem 3.1 and Proposition 4.1]{BuckdahnLiPengRainer2017} apply, and the derivatives $\nabla_x X_t^{x,\mu}$, $D^2_{xx}X_t^{x,\mu}$, $\partial_\mu X_t^{x,\mu}(\xi)$, and $D_\xi\partial_\mu X_t^{x,\mu}(\xi)$ are well defined. It remains to prove the estimate \eqref{eq:variational-growth}. For simplicity, write
  \[
    J_t^{x,\mu}:=\nabla_x X_t^{x,\mu},\qquad
    H_t^{x,\mu}:=D^2_{xx}X_t^{x,\mu},\qquad
    U_t^{x,\mu}(\xi):=\partial_\mu X_t^{x,\mu}(\xi),\qquad
    V_t^{x,\mu}(\xi):=D_\xi\partial_\mu X_t^{x,\mu}(\xi).
  \]
  We denote by $\sigma^\pi_\ell$ the $\ell$-th column of $\sigma^\pi$,
  $\ell=1,\dots,d$.
  Since only upper bounds are relevant, we may enlarge $K_\pi$ once and
  for all and assume $K_\pi\ge1$.

  \smallskip
  \noindent\textbf{Step 1 (Bounds for the first spatial derivative).}
  Differentiating the decoupled SDE with respect to $x$ gives
  \begin{align}\label{eq:variational_J}
    \d J_t^{x,\mu}
    =
    \nabla_s b^\pi(X_t^{x,\mu},\mu_t)J_t^{x,\mu}\,\d t
    + \sum_{\ell=1}^d
      \nabla_s \sigma_\ell^\pi(X_t^{x,\mu},\mu_t)J_t^{x,\mu}\,\d B_t^\ell,
    \qquad
    J_0^{x,\mu}=I_d.
  \end{align}
  Since each entry of $\nabla_s b^\pi$ and $\nabla_s\sigma^\pi_\ell$
  is bounded by $K_\pi$, we have
  \begin{align}\label{eq:bound_nabla_b_sigma}
    \|\nabla_s b^\pi(X_t^{x,\mu},\mu_t)\|
    \le \sqrt d\,K_\pi,\qquad  \sum_{\ell=1}^d
    \|\nabla_s \sigma_\ell^\pi(X_t^{x,\mu},\mu_t)\|^2
    \le d^2K_\pi^2.
  \end{align}
  Applying It\^o's formula to $\|J_t^{x,\mu}\|^p$, integrating over
  $[0,t]$, and taking expectations, we obtain
  \begin{align*}
    \mathbb{E}\!\left[\|J_t^{x,\mu}\|^p\right]
    &=
    \|I_d\|^p
    + p\int_0^t
      \mathbb{E}\!\left[
        \|J_r^{x,\mu}\|^{p-2}
        \left\langle
          J_r^{x,\mu},
          \nabla_s b^\pi(X_r^{x,\mu},\mu_r)J_r^{x,\mu}
        \right\rangle
      \right]\d r
 \\
    &\quad
    + \frac p2 \sum_{\ell=1}^d \int_0^t
      \mathbb{E}\!\left[
        \|J_r^{x,\mu}\|^{p-2}
        \|\nabla_s \sigma_\ell^\pi(X_r^{x,\mu},\mu_r)J_r^{x,\mu}\|^2
      \right]\d r
    \\
    &\quad
    + \frac{p(p-2)}2 \sum_{\ell=1}^d \int_0^t
      \mathbb{E}\!\left[
        \|J_r^{x,\mu}\|^{p-4}
        \left\langle
          J_r^{x,\mu},
          \nabla_s \sigma_\ell^\pi(X_r^{x,\mu},\mu_r)J_r^{x,\mu}
        \right\rangle^2
      \right]\d r.
  \end{align*}
  Using \eqref{eq:bound_nabla_b_sigma}, and
  $\langle z,Az\rangle\le \|A\|\,\|z\|^2$, and recalling $K_\pi\ge1$,
  we infer that
  \begin{align*}
    \mathbb{E}\!\left[\|J_t^{x,\mu}\|^p\right]
    \le d^{p/2}
    + a_{p,d}K_\pi^2
      \int_0^t \mathbb{E}\!\left[\|J_r^{x,\mu}\|^p\right]\d r .
  \end{align*}
  Gr\"onwall's inequality gives
  \begin{align}\label{eq:J_growth}
    \mathbb{E}\!\left[\|J_t^{x,\mu}\|^p\right]
    \le d^{p/2}e^{a_{p,d}K_\pi^2 t}.
  \end{align}

  \smallskip
  \noindent\textbf{Step 2 (Bounds for the second spatial derivative).}
  Differentiating \eqref{eq:variational_J} once more with respect to $x$
  gives
  \begin{align}\label{eq:variational_H}
    \d H_t^{x,\mu}
    &=
    \Bigl[
      \nabla_s b^\pi(X_t^{x,\mu},\mu_t)H_t^{x,\mu}
      + D^2_{ss}b^\pi(X_t^{x,\mu},\mu_t)
        [J_t^{x,\mu},J_t^{x,\mu}]
    \Bigr]\d t
    \notag\\
    &\quad
    +
    \sum_{\ell=1}^d
    \Bigl[
      \nabla_s \sigma_\ell^\pi(X_t^{x,\mu},\mu_t)H_t^{x,\mu}
      + D^2_{ss}\sigma_\ell^\pi(X_t^{x,\mu},\mu_t)
        [J_t^{x,\mu},J_t^{x,\mu}]
    \Bigr]\d B_t^\ell,
    \qquad
    H_0^{x,\mu}=0.
  \end{align}
  By Assumption \ref{A1},
  \begin{align}\label{eq:H_source_bound}
    \|D^2_{ss}b^\pi(X_t^{x,\mu},\mu_t)[J_t^{x,\mu},J_t^{x,\mu}]\|
    + \sum_{\ell=1}^d
      \|D^2_{ss}\sigma_\ell^\pi(X_t^{x,\mu},\mu_t)
        [J_t^{x,\mu},J_t^{x,\mu}]\|
    \le (d+1)K_\pi\|J_t^{x,\mu}\|^2.
  \end{align}
  Applying It\^o's formula to $\|H_t^{x,\mu}\|^p$ in
  \eqref{eq:variational_H}, integrating over $[0,t]$, taking
  expectations, and using \eqref{eq:bound_nabla_b_sigma},
  \eqref{eq:bound_nabla_b_sigma}, \eqref{eq:H_source_bound}, and Young's
  inequality, we obtain
  \begin{align*}
    \mathbb{E}\!\left[\|H_t^{x,\mu}\|^p\right]
    \le
    a_{p,d}K_\pi^2
    \int_0^t \mathbb{E}\!\left[\|H_r^{x,\mu}\|^p\right]\d r
    + C_{p,d}K_\pi^p
    \int_0^t \mathbb{E}\!\left[\|J_r^{x,\mu}\|^{2p}\right]\d r.
  \end{align*}
  Using \eqref{eq:J_growth} with $2p$ in place of $p$ and then
  Gr\"onwall's inequality, we get
  \begin{align}\label{eq:H_growth}
    \mathbb{E}\!\left[\|H_t^{x,\mu}\|^p\right]
    \le C_{p,d}(1+K_\pi^p)e^{C_{p,d}K_\pi^2 t}.
  \end{align}
  \smallskip
  \noindent\textbf{Step 3 (Bounds for the Lions derivative).}
  Let $(\widetilde\Omega,\widetilde{\mathcal F},\widetilde{\mathbb P})$
  be an independent copy of the original probability space, and denote
  by $\widetilde X_t^{\xi,\mu}$ and $\widetilde J_t^{\xi,\mu}:=\nabla_x X_t^{x,\mu}\big|_{x=\xi},$ the corresponding independent copies of the decoupled flow and its
  first spatial derivative, started from $\xi$. Then
  $U_t^{x,\mu}(\xi)$ satisfies
  \begin{align}\label{eq:variational_U}
    \d U_t^{x,\mu}(\xi)
    &=
    \Bigl[
      \nabla_s b^\pi(X_t^{x,\mu},\mu_t)U_t^{x,\mu}(\xi)
      + F_t^{x,\mu}(\xi)
    \Bigr]\d t
    +
    \sum_{\ell=1}^d
    \Bigl[
      \nabla_s \sigma_\ell^\pi(X_t^{x,\mu},\mu_t)U_t^{x,\mu}(\xi)
      + G_{t,\ell}^{x,\mu}(\xi)
    \Bigr]\d B_t^\ell
  \end{align}
  where $U_0^{x,\mu}(\xi)=0,$ and 
  \begin{align*}
    F_t^{x,\mu}(\xi)
    :=
    \widetilde{\mathbb E}\!\left[
      \partial_\mu b^\pi(X_t^{x,\mu},\mu_t)(\widetilde X_t^{\xi,\mu})
      \,\widetilde J_t^{\xi,\mu}
    \right],\qquad
    G_{t,\ell}^{x,\mu}(\xi)
    :=
    \widetilde{\mathbb E}\!\left[
      \partial_\mu \sigma_\ell^\pi(X_t^{x,\mu},\mu_t)(\widetilde X_t^{\xi,\mu})
      \,\widetilde J_t^{\xi,\mu}
    \right].
  \end{align*}
  Since $\partial_\mu b^\pi$ and $\partial_\mu\sigma^\pi_\ell$ are
  bounded by $K_\pi$, Jensen's inequality yields
  \[
    \|F_t^{x,\mu}(\xi)\|^p
    \le K_\pi^p\,
    \widetilde{\mathbb E}\!\left[\|\widetilde J_t^{\xi,\mu}\|^p\right],
    \qquad
    \|G_{t,\ell}^{x,\mu}(\xi)\|^p
    \le K_\pi^p\,
    \widetilde{\mathbb E}\!\left[\|\widetilde J_t^{\xi,\mu}\|^p\right].
  \]
  Taking expectations and using \eqref{eq:J_growth} with $x=\xi$, we obtain
  \begin{align}\label{eq:FG_growth}
    \mathbb{E}\!\left[\|F_t^{x,\mu}(\xi)\|^p\right]
    + \sum_{\ell=1}^d
      \mathbb{E}\!\left[\|G_{t,\ell}^{x,\mu}(\xi)\|^p\right]
    \le C_{p,d}K_\pi^p e^{C_{p,d}K_\pi^2 t}.
  \end{align}
  Applying It\^o's formula to $\|U_t^{x,\mu}(\xi)\|^p$ in
  \eqref{eq:variational_U}, integrating over $[0,t]$, taking
  expectations, and using \eqref{eq:FG_growth}, we infer
  \begin{align*}
    \mathbb{E}\!\left[\|U_t^{x,\mu}(\xi)\|^p\right]
    \le
    a_{p,d}K_\pi^2
    \int_0^t \mathbb{E}\!\left[\|U_r^{x,\mu}(\xi)\|^p\right]\d r
    + C_{p,d}K_\pi^p
    \int_0^t e^{C_{p,d}K_\pi^2 r}\,\d r .
  \end{align*}
  Gr\"onwall's inequality gives
  \begin{align}\label{eq:U_growth}
    \mathbb{E}\!\left[\|U_t^{x,\mu}(\xi)\|^p\right]
    \le C_{p,d}(1+K_\pi^p)e^{C_{p,d}K_\pi^2 t}.
  \end{align}

  \smallskip
  \noindent\textbf{Step 4 (Bounds for $D_\xi\partial_\mu X_t^{x,\mu}(\xi)$).}
  Let $ \widetilde H_t^{\xi,\mu}
    :=D^2_{xx}X_t^{x,\mu}\big|_{x=\xi}$ denote the independent copy of the second spatial derivative started
  from $\xi$. Differentiating \eqref{eq:variational_U} with respect to
  $\xi$ yields
  \begin{align}\label{eq:variational_V}
    \d V_t^{x,\mu}(\xi)
    &=
    \Bigl[
      \nabla_s b^\pi(X_t^{x,\mu},\mu_t)V_t^{x,\mu}(\xi)
      + \bar F_t^{x,\mu}(\xi)
    \Bigr]\d t
    +
    \sum_{\ell=1}^d
    \Bigl[
      \nabla_s \sigma_\ell^\pi(X_t^{x,\mu},\mu_t)V_t^{x,\mu}(\xi)
      + \bar G_{t,\ell}^{x,\mu}(\xi)
    \Bigr]\d B_t^\ell,
  \end{align}
  where $V_0^{x,\mu}(\xi)=0,$ and
  \begin{align*}
    \bar F_t^{x,\mu}(\xi)
    &:=
    \widetilde{\mathbb E}\!\left[
      D_\xi\partial_\mu b^\pi(X_t^{x,\mu},\mu_t)(\widetilde X_t^{\xi,\mu})
      [\widetilde J_t^{\xi,\mu},\widetilde J_t^{\xi,\mu}]
      + \partial_\mu b^\pi(X_t^{x,\mu},\mu_t)(\widetilde X_t^{\xi,\mu})
      \widetilde H_t^{\xi,\mu}
    \right],\\
    \bar G_{t,\ell}^{x,\mu}(\xi)
    &:=
    \widetilde{\mathbb E}\!\left[
      D_\xi\partial_\mu \sigma_\ell^\pi(X_t^{x,\mu},\mu_t)(\widetilde X_t^{\xi,\mu})
      [\widetilde J_t^{\xi,\mu},\widetilde J_t^{\xi,\mu}]
      + \partial_\mu \sigma_\ell^\pi(X_t^{x,\mu},\mu_t)(\widetilde X_t^{\xi,\mu})
      \widetilde H_t^{\xi,\mu}
    \right].
  \end{align*}
  Since $D_\xi\partial_\mu b^\pi$, $D_\xi\partial_\mu\sigma^\pi_\ell$,
  $\partial_\mu b^\pi$, and $\partial_\mu\sigma^\pi_\ell$ are all
  bounded by $K_\pi$, Jensen's inequality gives
  \[
    \|\bar F_t^{x,\mu}(\xi)\|^p
    \le C_{p,d}K_\pi^p\,
    \widetilde{\mathbb E}\!\left[
      \|\widetilde J_t^{\xi,\mu}\|^{2p}
      + \|\widetilde H_t^{\xi,\mu}\|^p
    \right],
  \]
  and similarly for $\bar G_{t,\ell}^{x,\mu}(\xi)$. Taking expectations,
  and using \eqref{eq:J_growth} with $2p$ in place of $p$ together with
  \eqref{eq:H_growth}, we obtain
  \begin{align}\label{eq:FbarGbar_growth}
    \mathbb{E}\!\left[\|\bar F_t^{x,\mu}(\xi)\|^p\right]
    + \sum_{\ell=1}^d
      \mathbb{E}\!\left[\|\bar G_{t,\ell}^{x,\mu}(\xi)\|^p\right]
    \le C_{p,d}(1+K_\pi^p)e^{C_{p,d}K_\pi^2 t}.
  \end{align}
  Applying It\^o's formula to $\|V_t^{x,\mu}(\xi)\|^p$ in
  \eqref{eq:variational_V}, integrating over $[0,t]$, taking
  expectations, and using \eqref{eq:FbarGbar_growth}, we infer
  \begin{align*}
    \mathbb{E}\!\left[\|V_t^{x,\mu}(\xi)\|^p\right]
    \le
    a_{p,d}K_\pi^2
    \int_0^t \mathbb{E}\!\left[\|V_r^{x,\mu}(\xi)\|^p\right]\d r
    + C_{p,d}(1+K_\pi^p)
    \int_0^t e^{C_{p,d}K_\pi^2 r}\,\d r .
  \end{align*}
  Gr\"onwall's inequality yields
  \begin{align}\label{eq:V_growth}
    \mathbb{E}\!\left[\|V_t^{x,\mu}(\xi)\|^p\right]
    \le C_{p,d}(1+K_\pi^p)e^{C_{p,d}K_\pi^2 t}.
  \end{align}
  Combining \eqref{eq:J_growth}, \eqref{eq:H_growth},
  \eqref{eq:U_growth}, and \eqref{eq:V_growth}, we obtain
  \eqref{eq:variational-growth}.
\end{proof}

\begin{remark}[Explicit value of $\beta_{\mathrm{var}}(2)$]\label{remark_explicit_beta_var02}
A careful analysis shows that the constant $C_{p,d}$ from \Cref{lem:variational-growth} is given by 
  \begin{align*}
    C_{p,d}:=2^{2p}(1+d^p)\bigl(1+2p\sqrt d+\frac{2p(2p-1)}{2}d^2+(d+1)^p\bigr),
  \end{align*}
In particular, we have that $\beta_{\mathrm{var}}(2)=(32+32d+64\sqrt d+144d^{2}+32d^{3}+64d^{5/2}+112d^{4})K_\pi^2$.
\end{remark}
We can now prove \Cref{thm:MF_eval_infinite}.

\begin{proof}[\textbf{Proof of Theorem \ref{thm:MF_eval_infinite}}]
  We proceed in several steps.

  \smallskip
  \noindent\textbf{Step 1 (Finite-horizon problem).}
  For each $T>0$, define
  \begin{align*}
    V_T^\pi(t,s,\mu)
    :=\mathbb{E}^{t,s,\mu,\pi}\!\left[
      \int_t^T e^{-\beta(\tau-t)}\,r_\lambda^\pi(s_\tau,\mu_\tau)\,\d\tau
    \right],
    \qquad
    V_T^\pi(T,\cdot,\cdot)\equiv0.
  \end{align*}
  Under Assumption \ref{A1}, the finite-horizon
  policy-evaluation problem with coefficients $b^\pi$, $\sigma^\pi$,
  running reward $r_\lambda^\pi$, and terminal cost $g\equiv0$ falls
  under the time-homogeneous version of
  \cite[Assumption 2.1]{frikhaActorCritic}. Therefore, by
  \cite[Proposition 2.1]{frikhaActorCritic}, $V_T^\pi\in
    C^{1,2,2}\bigl([0,T)\times\mathbb{R}^d\times\mathcal{P}_2(\mathbb{R}^d)\bigr)$
  and $V_T^\pi$ solves
  \begin{equation}\label{eq:backward_kolmogorov}
    \partial_t V_T^\pi
    + \mathcal{L}_{b,\Sigma}^\pi V_T^\pi
    - \beta V_T^\pi
    + r_\lambda^\pi
    =0,
    \qquad
    0\le t<T,
  \end{equation}
  with terminal condition $V_T^\pi(T,\cdot,\cdot)=0$.

  Set
  \begin{align*}
    U_T^\pi(s,\mu):=V_T^\pi(0,s,\mu)
    = \mathbb{E}\!\left[
      \int_0^T e^{-\beta t}\,r_\lambda^\pi(X_t^{s,\mu},\mu_t)\,\d t
    \right].
  \end{align*}

  \smallskip
  \noindent\textbf{Step 2 (Well-definedness of $V^\pi$ and convergence of $U_T^\pi$).}
  Since
  $r_\lambda^\pi\in\mathcal{C}^{2,2}_{\mathrm{poly}}
  (\mathbb{R}^d\times\mathcal{P}_2(\mathbb{R}^d))$,
  Definition \ref{def:C22poly} yields
  \begin{align}\label{eq:r_growth}
    |r_\lambda^\pi(s,\mu)|
    \le C_r\bigl(1+|s|^2+m_2(\mu)\bigr)
  \end{align}
  for some constant $C_r>0$. By
  Lemma \ref{lem:moment-growth-clean} with $p=2$,
  \begin{align}\label{eq:moment_growth_representative}
    \mathbb{E}^{s,\mu,\pi}[|s_t|^2]
    + \mathbb{E}^{s,\mu,\pi}[|\tilde s_t|^2]
    \le C_2 e^{\beta_0(2)t}\bigl(1+|s|^2+m_2(\mu)\bigr),
    \qquad
    t\ge0.
  \end{align}
  Since $m_2(\mu_t)=\mathbb{E}^{s,\mu,\pi}[|\tilde s_t|^2]$, it follows
  from \eqref{eq:r_growth} and \eqref{eq:moment_growth_representative}
  that
  \begin{align}\label{eq:dominating_r_final}
    e^{-\beta t}\,
    \mathbb{E}^{s,\mu,\pi}\!\left[
      |r_\lambda^\pi(s_t,\mu_t)|
    \right]
    \le
    C e^{-(\beta-\beta_0(2))t}
    \bigl(1+|s|^2+m_2(\mu)\bigr).
  \end{align}
  Therefore, if $\beta>\bar{\beta}_\pi$, then the right-hand side of
  \eqref{eq:dominating_r_final} is integrable on $[0,\infty)$, and the
  value function $V^\pi(s,\mu)$ is well defined. Moreover, dominated convergence yields
  \begin{align}\label{eq:UT_to_V_pointwise}
    U_T^\pi(s,\mu)\xrightarrow[T\to\infty]{}V^\pi(s,\mu)
  \end{align}
  for every $(s,\mu)\in\mathbb{R}^d\times\mathcal{P}_2(\mathbb{R}^d)$.
  If $K_R:=\bigl\{(s,\mu)\in\mathbb{R}^d\times\mathcal{P}_2(\mathbb{R}^d):
    |s|^2+m_2(\mu)\le R\bigr\},$ then \eqref{eq:dominating_r_final} also gives
  \begin{align}\label{eq:UT_to_V_uniform}
    \sup_{(s,\mu)\in K_R}|V^\pi(s,\mu)-U_T^\pi(s,\mu)|
    \le C_R\int_T^\infty e^{-(\beta-\beta_0(2))t}\,\d t
    \xrightarrow[T\to\infty]{}0.
  \end{align}

  \smallskip
  \noindent\textbf{Step 3 (Uniform finite-horizon bounds in $\mathcal{C}^{2,2}_{\mathrm{poly}}$).}
  We first record the corresponding moment bound for the decoupled flow
  $X_t^{x,\mu}$. By the same It\^o-Gr\"onwall argument used in the proof
  of Lemma \ref{lem:moment-growth-clean}, together with
  \eqref{eq:linear_growth_bpi} and the estimate on $m_2(\mu_t)$ coming
  from \eqref{eq:moment_growth_representative}, there exists a constant
  $C>0$, depending only on the structural constants in
  Assumption \ref{A1}, such that
  \begin{align*}
    \mathbb{E}\!\left[|X_t^{x,\mu}|^2\right]
    \le C e^{\beta_0(2)t}\bigl(1+|x|^2+m_2(\mu)\bigr),
    \qquad
    t\ge0.
  \end{align*}
  
  We now estimate the derivatives of $U_T^\pi$. Since $U_T^\pi$ is the finite-horizon value function, Step 2 of \cite[Appendix A.1]{frikhaActorCritic} provides explicit representation formulas for $\nabla_s U_T^\pi$, $D^2_{ss}U_T^\pi$, $\partial_\mu U_T^\pi$, and $D_\xi\partial_\mu U_T^\pi$. These formulas show that each derivative is given by an integral over $[0,T]$ of
$e^{-\beta t}$ times the expectation of a finite linear combination of terms involving derivatives of $r_\lambda^\pi$ evaluated at
$(X_t^{s,\mu},\mu_t)$ and $(X_t^{\xi,\mu},\mu_t)$, together with the
corresponding derivatives of the flow. To avoid an unnecessarily long proof, we refer to \cite[Appendix A.1, Step 2]{frikhaActorCritic} for the explicit expressions. Therefore, by using Lemma \ref{lem:variational-growth} with $p=2$, and the fact that $r_\lambda^\pi\in\mathcal{C}^{2,2}_{\mathrm{poly}}(\mathbb{R}^d\times\mathcal{P}_2(\mathbb{R}^d))$ we derive the estimates
  \begin{align}
    \mathbb{E}\!\left[
      \left|
        \nabla_s r_\lambda^\pi(X_t^{s,\mu},\mu_t)\,
        \nabla_x X_t^{s,\mu}
      \right|
    \right]+\mathbb{E}\!\left[
      \left|
        \nabla_s r_\lambda^\pi(X_t^{s,\mu},\mu_t)\,
        \partial_\mu X_t^{s,\mu}(\xi)
      \right|
    \right]
    &\le
    C e^{\bar{\beta}_\pi t}\bigl(1+|s|+m_2(\mu)\bigr),
    \label{eq:bound_type_grad_1}
  \end{align}
  and 
  \begin{equation*}
    \begin{aligned}
         &\mathbb{E}\!\left[
      \left|
        \nabla_s r_\lambda^\pi(X_t^{s,\mu},\mu_t)\,
        D^2_{xx}X_t^{s,\mu}
      \right|
    \right]
   +\mathbb{E}\!\left[
      \left|
        \nabla_s r_\lambda^\pi(X_t^{s,\mu},\mu_t)\,
        D_\xi\partial_\mu X_t^{s,\mu}(\xi)
      \right|
    \right] \\
    &\quad+\mathbb{E}\!\left[
      \left|
        D^2_{ss}r_\lambda^\pi(X_t^{s,\mu},\mu_t)
        [\nabla_x X_t^{s,\mu},\nabla_x X_t^{s,\mu}]
      \right|
    \right]+\mathbb{E}\!\left[
      \left|
        D_\xi\partial_\mu r_\lambda^\pi(X_t^{s,\mu},\mu_t)(X_t^{\xi,\mu})
        \,\nabla_x X_t^{\xi,\mu}
      \right|
    \right]\\
    &\hspace{10cm}\le
    C e^{\bar{\beta}_\pi t}\bigl(1+m_2(\mu)\bigr),
    \end{aligned}
  \end{equation*}
and 
  \begin{align}
    \mathbb{E}\!\left[
      \left|
        \partial_\mu r_\lambda^\pi(X_t^{s,\mu},\mu_t)(X_t^{\xi,\mu})
      \right|
    \right]
    +\mathbb{E}\!\left[
      \left|
        \partial_\mu r_\lambda^\pi(X_t^{s,\mu},\mu_t)(X_t^{\xi,\mu})
        \,D^2_{xx}X_t^{\xi,\mu}
      \right|
    \right]\le
    C e^{\beta_0(2)t}\bigl(1+|s|+|\xi|+m_2(\mu)\bigr)
    \label{eq:bound_type_grad_2}.
  \end{align}

  Integrating \eqref{eq:bound_type_grad_1}-\eqref{eq:bound_type_grad_2}
  over $[0,T]$, multiplying by $e^{-\beta t}$, and using
  $\beta>\bar{\beta}_\pi$, we conclude that there exists a constant
  $C>0$, independent of $T$, such that for all
  $(s,\mu,\xi)\in\mathbb{R}^d\times\mathcal{P}_2(\mathbb{R}^d)\times\mathbb{R}^d$,
  \begin{align}
    |U_T^\pi(s,\mu)|
    &\le C\bigl(1+|s|^2+m_2(\mu)\bigr),
    \label{eq:UT_growth}\\
    |\nabla_s U_T^\pi(s,\mu)|
    + |\partial_\mu U_T^\pi(s,\mu)(\xi)|
    &\le C\bigl(1+|s|+|\xi|+m_2(\mu)\bigr),
    \label{eq:UT_first_growth}\\
    \|D^2_{ss}U_T^\pi(s,\mu)\|
    + \|D_\xi\partial_\mu U_T^\pi(s,\mu)(\xi)\|
    &\le C\bigl(1+m_2(\mu)\bigr).
    \label{eq:UT_second_growth}
  \end{align}

  \smallskip
  \noindent\textbf{Step 4 (Passage to the limit in the derivatives).}
  Let
  \[
    K_R:=\bigl\{(s,\mu)\in\mathbb{R}^d\times\mathcal{P}_2(\mathbb{R}^d):
    |s|^2+m_2(\mu)\le R\bigr\},
    \qquad
    B_R:=\{\xi\in\mathbb{R}^d: |\xi|\le R\}.
  \]
  By \eqref{eq:UT_to_V_uniform}, the family $(U_T^\pi)_{T>0}$ is
  uniformly Cauchy on $K_R$. Using the representation formulas from \cite[Appendix A.1, Step 2]{frikhaActorCritic} together with the estimations
  \eqref{eq:bound_type_grad_1}-\eqref{eq:bound_type_grad_2}, we
  obtain, for every $D_\alpha\in \{1,\nabla_s,\partial_\mu, D_{ss}^2, D_\xi \partial_\mu\}$, for all $T'>T\ge0$, all $(s,\mu)\in K_R$, and all $\xi\in B_R$,
  \begin{align*}
    |D_\alpha U_{T'}^\pi(s,\mu)-D_\alpha U_T^\pi(s,\mu)|
    &\le
    C_R\int_T^{T'} e^{-(\beta-\bar{\beta}_\pi)t}\,\d t,
  \end{align*}
  Since $\beta>\bar{\beta}_\pi$, the right-hand sides tend to zero as
  $T\to\infty$. Hence, for each $D_\alpha$, the families $\{D_\alpha U_T^\pi\}_{T\geq 0}$ are locally uniformly Cauchy and thus converge locally uniformly to
  continuous limits.

  By \eqref{eq:UT_to_V_pointwise}, the limit of $U_T^\pi$ is precisely
  $V^\pi$. The local uniform convergence of $U_T^\pi$,
  $\nabla_s U_T^\pi$, and $D^2_{ss}U_T^\pi$ implies that $V^\pi$ is twice
  continuously differentiable in $s$, with
  \[
    \nabla_s V^\pi=\lim_{T\to\infty}\nabla_s U_T^\pi,
    \qquad
    D^2_{ss}V^\pi=\lim_{T\to\infty}D^2_{ss}U_T^\pi.
  \]
  For the measure derivative, let $\vartheta,\vartheta'$ be square-integrable
  random variables with laws $\mu$ and $\mu'$, and set
  $\vartheta_\lambda:=\vartheta+\lambda(\vartheta'-\vartheta)$,
  $\mu_\lambda:=\Law(\vartheta_\lambda)$. Since each $U_T^\pi$ is Lions
  differentiable, the mean-value formula on
  $\mathcal{P}_2(\mathbb{R}^d)$ gives
  \begin{align}\label{eq:mean_value_formula_UT}
    U_T^\pi(s,\mu')-U_T^\pi(s,\mu)
    =
    \int_0^1
    \mathbb{E}\!\left[
      \partial_\mu U_T^\pi(s,\mu_\lambda)(\vartheta_\lambda)
      \cdot(\vartheta'-\vartheta)
    \right]\d\lambda.
  \end{align}
  Passing to the limit $T\to\infty$ in \eqref{eq:mean_value_formula_UT},
  using the local uniform convergence of $U_T^\pi$ and
  $\partial_\mu U_T^\pi$ together with dominated convergence, we obtain
  that $V^\pi$ is Lions differentiable and
  \[
    \partial_\mu V^\pi
    =
    \lim_{T\to\infty}\partial_\mu U_T^\pi.
  \]
  Applying the same argument to
  $\partial_\mu U_T^\pi(s,\mu)(\xi)$ as a function of $\xi$, and using
  the local uniform convergence of
  $D_\xi\partial_\mu U_T^\pi(s,\mu)(\xi)$, we conclude that
  $D_\xi\partial_\mu V^\pi$ exists and is jointly continuous. Therefore, $  V^\pi\in
    \mathcal{C}^{2,2}(\mathbb{R}^d\times\mathcal{P}_2(\mathbb{R}^d)).$
  Letting $T\to\infty$ in \eqref{eq:UT_growth},
  \eqref{eq:UT_first_growth}, and \eqref{eq:UT_second_growth}, we obtain that $V^\pi\in
    \mathcal{C}^{2,2}_{\mathrm{poly}}
    (\mathbb{R}^d\times\mathcal{P}_2(\mathbb{R}^d)).$

  \smallskip
  \noindent\textbf{Step 5 (The stationary HJB equation).}
  Evaluating \eqref{eq:backward_kolmogorov} at $t=0$ yields
  \begin{align}\label{eq:kolmogorov_at_t0}
    \bigl(\mathcal{L}_{b,\Sigma}^\pi-\beta\bigr)U_T^\pi(s,\mu)
    +r_\lambda^\pi(s,\mu)
    = -\partial_t V_T^\pi(0,s,\mu).
  \end{align}
  Since the problem is time-homogeneous, $V_T^\pi(t,s,\mu)=U_{T-t}^\pi(s,\mu),$
  and therefore
  \begin{align*}
    -\partial_t V_T^\pi(0,s,\mu)
    =
    \frac{\d}{\d T}U_T^\pi(s,\mu)
    =
    \mathbb{E}\!\left[
      e^{-\beta T}r_\lambda^\pi(X_T^{s,\mu},\mu_T)
    \right].
  \end{align*}
  Thus, by \eqref{eq:dominating_r_final},
  \begin{align}\label{eq:boundary_term_to_zero}
    \left|\partial_t V_T^\pi(0,s,\mu)\right|\leq \left|
      \mathbb{E}\!\left[
        e^{-\beta T}r_\lambda^\pi(X_T^{s,\mu},\mu_T)
      \right]
    \right|
    \le
    C e^{-(\beta-\beta_0(2))T}\bigl(1+|s|^2+m_2(\mu)\bigr)
    \xrightarrow[T\to\infty]{}0.
  \end{align}

  It remains to pass to the limit in
  $\bigl(\mathcal{L}_{b,\Sigma}^\pi-\beta\bigr)U_T^\pi(s,\mu)$.
  The terms involving $\nabla_s U_T^\pi$ and $D^2_{ss}U_T^\pi$ converge
  pointwise by Step 4. For the Lions terms, using
  \eqref{eq:linear_growth_bpi}, \eqref{eq:bounds_Simga_pi},
  \eqref{eq:UT_first_growth}, and \eqref{eq:UT_second_growth}, we obtain
  \begin{align*}
    |b^\pi(\xi,\mu)\cdot\partial_\mu U_T^\pi(s,\mu)(\xi)|
    &\le C_{s,\mu}(1+|\xi|^2),\\
    |\Sigma^\pi(\xi,\mu):D_\xi\partial_\mu U_T^\pi(s,\mu)(\xi)|
    &\le C_{s,\mu}(1+|\xi|^2),
  \end{align*}
  with constants independent of $T$. Since
  $\mu\in\mathcal{P}_2(\mathbb{R}^d)$, the dominating function is
  $\mu$-integrable, and dominated convergence yields
  \begin{align}\label{eq:generator_convergence}
    \bigl(\mathcal{L}_{b,\Sigma}^\pi-\beta\bigr)U_T^\pi(s,\mu)
    \xrightarrow[T\to\infty]{}
    \bigl(\mathcal{L}_{b,\Sigma}^\pi-\beta\bigr)V^\pi(s,\mu).
  \end{align}
  Combining \eqref{eq:kolmogorov_at_t0}, \eqref{eq:boundary_term_to_zero},
  and \eqref{eq:generator_convergence}, and letting $T\to\infty$, we
  obtain
  \[
    (\mathcal{L}_{b,\Sigma}^{\pi}-\beta)V^\pi(s,\mu)
    +r_\lambda^\pi(s,\mu)=0.
  \]
  Thus $V^\pi$ is a classical solution of \eqref{eq:HJB-eval}.

  \smallskip
  \noindent\textbf{Step 6 (Uniqueness).}
  Let $U\in
    \mathcal{C}^{2,2}_{\mathrm{poly}}
    (\mathbb{R}^d\times\mathcal{P}_2(\mathbb{R}^d))$  be another classical solution of \eqref{eq:HJB-eval}. Fix
  $(s,\mu)\in\mathbb{R}^d\times\mathcal{P}_2(\mathbb{R}^d)$ and define
  the stopping times $\tau_n:=\inf\{t\ge0: |s_t|+|\tilde s_t|\ge n\},$ with $ n\in\mathbb{N}.$ Applying the It\^o-Lions formula to
  $e^{-\beta(t\wedge\tau_n)}U(s_{t\wedge\tau_n},\mu_{t\wedge\tau_n})$ on
  $[0,T]$, taking expectations, and using \eqref{eq:HJB-eval}, we obtain
  \begin{align}\label{eq:U_stoping_times}
    U(s,\mu)
    &=
    \mathbb{E}^{s,\mu,\pi}\!\left[
      e^{-\beta(T\wedge\tau_n)}
      U(s_{T\wedge\tau_n},\mu_{T\wedge\tau_n})
    \right]
    +
    \mathbb{E}^{s,\mu,\pi}\!\left[
      \int_0^{T\wedge\tau_n}
      e^{-\beta t}r_\lambda^\pi(s_t,\mu_t)\,\d t
    \right].
  \end{align}
  Since
  $U\in\mathcal{C}^{2,2}_{\mathrm{poly}}
  (\mathbb{R}^d\times\mathcal{P}_2(\mathbb{R}^d))$, together with \eqref{eq:moment_growth_representative}, we have
  \[
    e^{-\beta T}
    \mathbb{E}^{s,\mu,\pi}\!\left[|U(s_T,\mu_T)|\right]
    \le
    C e^{-(\beta-\beta_0(2))T}\bigl(1+|s|^2+m_2(\mu)\bigr),
  \]
  which tends to zero as $T\to\infty$. Letting first $n\to\infty$ in \eqref{eq:U_stoping_times}, and
  then $T\to\infty$, and using dominated convergence together with
  \eqref{eq:dominating_r_final}, we conclude that
  \[
    U(s,\mu)
    =
    \mathbb{E}^{s,\mu,\pi}\!\left[
      \int_0^\infty e^{-\beta t}r_\lambda^\pi(s_t,\mu_t)\,\d t
    \right]
    =
    V^\pi(s,\mu).
  \]
  Hence $U=V^\pi$, and the classical solution of \eqref{eq:HJB-eval} is
  unique in the class
  $\mathcal{C}^{2,2}_{\mathrm{poly}}
  (\mathbb{R}^d\times\mathcal{P}_2(\mathbb{R}^d))$.
\end{proof}
    
    \subsection{Proofs for \Cref{sec:policy-iteration}}\label{APP:PROOFS_SEC_PG}
    
We turn to the main policy gradient theorem. A technical ingredient
is needed: a stability estimate that controls the first-order response
of the controlled system to a policy perturbation.
\begin{lemma}[Stability with respect to policy perturbations]
\label[lemma]{lem:stability-policy}
Suppose \Cref{A1} holds. Let $\pi$ be a fixed feedback policy and
$\varphi$ a signed kernel satisfying \eqref{eq:signed-kernel-zero-mass}
with $M_\psi := \|\psi\|_\infty < \infty$. For $|\varepsilon|$
small enough that $\pi^\varepsilon$ is admissible, let
$(s_t^\varepsilon,\mu_t^\varepsilon)$ and $(s_t^\pi,\mu_t^\pi)$ be the
solutions of \eqref{eq:rep_SDE}-\eqref{eq:MKV_SDE} under $\pi^\varepsilon$
and $\pi$ respectively, with the same initial condition $s_0\sim\mu_0$,
$\mu_0^\varepsilon=\mu_0^\pi=\mu$, and the same driving Brownian motions.
Then there exist constants $C>0$ and $\beta_{\mathrm{stab}}>0$,
depending only on $L_\pi$, $M_\psi$, and $\beta_0(2)$, such that
for every $t\geq0$,
\begin{equation}\label{eq:stability-policy}
    \mathbb{E}\bigl[|s_t^\varepsilon-s_t^\pi|^2\bigr]
    +\mathcal{W}_2^2(\mu_t^\varepsilon,\mu_t^\pi)
    \leq C\,\varepsilon^2\,e^{\beta_{\mathrm{stab}}\,t}
    \bigl(1+|s|^2+m_2(\mu)\bigr).
\end{equation}
In particular, $\mathbb{E}[|s_t^\varepsilon-s_t^\pi|^2]
+\mathcal{W}_2^2(\mu_t^\varepsilon,\mu_t^\pi)\to0$
as $\varepsilon\to0$, for each fixed $t\geq0$.
\end{lemma}

\begin{proof}
Set $\delta s_t := s_t^\varepsilon - s_t^\pi$ and
$\delta\tilde{s}_t := \tilde{s}_t^\varepsilon - \tilde{s}_t^\pi$.
Since both pairs of processes are driven by the same Brownian motions
$B$ and $\tilde{B}$, the differences satisfy
\begin{equation*}
    \d(\delta s_t)
    = \bigl(b^{\pi^\varepsilon}(s_t^\varepsilon,\mu_t^\varepsilon)
           - b^\pi(s_t^\pi,\mu_t^\pi)\bigr)\,\d t
    + \bigl(\sigma^{\pi^\varepsilon}(s_t^\varepsilon,\mu_t^\varepsilon)
           - \sigma^\pi(s_t^\pi,\mu_t^\pi)\bigr)\,\d B_t,
\end{equation*}
and analogously for $\delta\tilde{s}_t$.

\medskip
\noindent\textbf{Step 1: Decomposition of the coefficient differences.} By linearity of the averaged coefficients in the control measure:
\begin{equation}\label{eq:drift-decomp}
    b^{\pi^\varepsilon}(s^\varepsilon,\mu^\varepsilon) - b^\pi(s^\pi,\mu^\pi)
    = \varepsilon\,b^\varphi(s^\varepsilon,\mu^\varepsilon)
    + \bigl[b^\pi(s^\varepsilon,\mu^\varepsilon)
            - b^\pi(s^\pi,\mu^\pi)\bigr],
\end{equation}
where $b^\varphi(s,\mu) := \int_\mathcal{A} b(s,\mu,a)\,\varphi(\d a|s,\mu)$.
By the linear growth of $b$ in \Cref{A1} and the zero-mass
condition \eqref{eq:signed-kernel-zero-mass},
\begin{equation*}
    |b^\varphi(s,\mu)| \leq C_b\bigl(1+|s|+\sqrt{m_2(\mu)}\bigr),
\end{equation*}
and by the Lipschitz condition of \Cref{A1}, the second term
of \eqref{eq:drift-decomp} is bounded by
$L_\pi(|\delta s|+\mathcal{W}_2(\mu^\varepsilon,\mu^\pi))$.

For the diffusion coefficient, write
\begin{equation}\label{eq:sigma-decomp}
    \sigma^{\pi^\varepsilon}(s^\varepsilon,\mu^\varepsilon)
    - \sigma^\pi(s^\pi,\mu^\pi)
    = \bigl[\sigma^\pi(s^\varepsilon,\mu^\varepsilon)
            - \sigma^\pi(s^\pi,\mu^\pi)\bigr]
    + \bigl[\sigma^{\pi^\varepsilon}(s^\varepsilon,\mu^\varepsilon)
            - \sigma^\pi(s^\varepsilon,\mu^\varepsilon)\bigr].
\end{equation}
The first bracket is bounded by
$L_\pi(|\delta s|+\mathcal{W}_2(\mu^\varepsilon,\mu^\pi))$
by \Cref{A1}. For the second bracket, since
$\pi^\varepsilon = (1+\varepsilon\psi)\pi$ with $|\varepsilon|M_\psi < 1$,
we have for every $(s,\mu)$ and every $z\in\mathbb{R}^d$:
\[
z^\top\Sigma^{\pi^\varepsilon}(s,\mu)z
= \int_{\mathcal{A}}z^\top\Sigma(s,\mu,a)z\,
\bigl(1+\varepsilon\psi(s,\mu,a)\bigr)\,\pi(\d a\mid s,\mu),
\]
from which
\begin{equation}\label{eq:sigma-sandwich}
    (1-M_\psi|\varepsilon|)\,\Sigma^\pi(s,\mu)
    \leq \Sigma^{\pi^\varepsilon}(s,\mu)
    \leq (1+M_\psi|\varepsilon|)\,\Sigma^\pi(s,\mu).
\end{equation}
Since the map $A\mapsto A^{1/2}$ is operator-monotone on the cone of
symmetric positive semidefinite matrices (a consequence of the
Löwner-Heinz inequality, which holds for any semidefinite matrix), taking square roots
in \eqref{eq:sigma-sandwich} gives
\[
\sqrt{1-M_\psi|\varepsilon|}\,\sigma^\pi(s,\mu)
\leq \sigma^{\pi^\varepsilon}(s,\mu)
\leq \sqrt{1+M_\psi|\varepsilon|}\,\sigma^\pi(s,\mu).
\]
Using the elementary inequalities $1-\sqrt{1-x}\leq x$ and
$\sqrt{1+x}-1\leq x$ for $x\in[0,1)$, we conclude that
\[
\|\sigma^{\pi^\varepsilon}(s,\mu)-\sigma^\pi(s,\mu)\|
\leq M_\psi|\varepsilon|\,\|\sigma^\pi(s,\mu)\|.
\]
By the linear growth of $\sigma^\pi$ from \Cref{A1}:
\begin{equation}\label{eq:sigma-eps-diff}
    \|\sigma^{\pi^\varepsilon}(s,\mu)-\sigma^\pi(s,\mu)\|^2
    \leq C\,\varepsilon^2\bigl(1+|s|^2+m_2(\mu)\bigr).
\end{equation}
Note that \eqref{eq:sigma-eps-diff} holds without any ellipticity
assumption on $\Sigma^\pi$; the sandwich argument requires only that
$\Sigma^\pi$ be positive semidefinite, which holds automatically.
Combining both brackets in \eqref{eq:sigma-decomp}, we have
\begin{equation}\label{eq:sigma-full-diff}
    \bigl\|\sigma^{\pi^\varepsilon}(s^\varepsilon,\mu^\varepsilon)
    -\sigma^\pi(s^\pi,\mu^\pi)\bigr\|^2
    \leq C\Bigl(|\delta s|^2+\mathcal{W}_2^2(\mu^\varepsilon,\mu^\pi)
    +\varepsilon^2\bigl(1+|s^\varepsilon|^2+m_2(\mu^\varepsilon)\bigr)\Bigr).
\end{equation}

\medskip
\noindent\textbf{Step 2: It\^{o} formula, Gr\"{o}nwall, and exponential bound.} Applying It\^{o}'s formula to $|\delta s_t|^2$, taking expectations
(the stochastic integral is a true martingale by the moment estimates
of Lemma \ref{lem:moment-growth-clean}), and using Young's inequality
together with \eqref{eq:drift-decomp}-\eqref{eq:sigma-full-diff}, we obtain
\begin{align*}
    \frac{\d}{\d t}\mathbb{E}\bigl[|\delta s_t|^2\bigr]
    \leq C_1\,\mathbb{E}\bigl[|\delta s_t|^2\bigr]
    + C_1\,\mathcal{W}_2^2(\mu_t^\varepsilon,\mu_t^\pi)
    + C_1\,\varepsilon^2\,
      \mathbb{E}\bigl[1+|s_t^\pi|^2+m_2(\mu_t^\pi)\bigr],
\end{align*}
where $C_1>0$ depends only on $L_\pi$ and $M_\psi$. An identical estimate
holds for $\delta\tilde{s}_t$. Since
$\mathcal{W}_2^2(\mu_t^\varepsilon,\mu_t^\pi)\leq\mathbb{E}[|\delta\tilde{s}_t|^2]$,
setting $\Phi_t := \mathbb{E}[|\delta s_t|^2] + \mathbb{E}[|\delta\tilde{s}_t|^2]$
and summing the two estimates gives
\begin{equation}\label{eq:Phi-gronwall}
    \frac{\d}{\d t}\Phi_t
    \leq C_2\,\Phi_t
    + C_2\,\varepsilon^2\,
      \mathbb{E}\bigl[1+|s_t^\pi|^2+|\tilde{s}_t^\pi|^2+m_2(\mu_t^\pi)\bigr],
\end{equation}
where $C_2$ depends only on $C_1$. By Lemma \ref{lem:moment-growth-clean}
with $p=2$, the factor on the right-hand side satisfies
\[
    \mathbb{E}\bigl[1+|s_t^\pi|^2+|\tilde{s}_t^\pi|^2+m_2(\mu_t^\pi)\bigr]
    \leq Ce^{\beta_0(2)t}\bigl(1+|s|^2+m_2(\mu)\bigr).
\]
Since $\Phi_0=0$, Gr\"{o}nwall's inequality applied
to \eqref{eq:Phi-gronwall} gives
\[
    \Phi_t
    \leq C_2\varepsilon^2(1+|s|^2+m_2(\mu))
    \int_0^t e^{C_2(t-r)}\,C\,e^{\beta_0(2)r}\,\d r
    \leq C'\varepsilon^2 e^{\max(C_2,\,\beta_0(2))\,t}
    \bigl(1+|s|^2+m_2(\mu)\bigr),
\]
where the last step uses $e^{C_2(t-r)}e^{\beta_0(2)r}\leq
e^{\max(C_2,\beta_0(2))t}$ for $r\in[0,t]$.
Setting $\beta_{\mathrm{stab}}:=\max(C_2,\,\beta_0(2))$ and recalling
$\mathcal{W}_2^2(\mu_t^\varepsilon,\mu_t^\pi)\leq\mathbb{E}[|\delta\tilde{s}_t|^2]
\leq\Phi_t$ gives \eqref{eq:stability-policy}.
\end{proof}

We are now in a position to prove \Cref{thm:policy-gradient-theorem}.

\begin{proof}[\textbf{Proof of Theorem \ref{thm:policy-gradient-theorem}}]

Since $\beta > \widetilde\beta_\pi \geq \bar\beta_\pi$, Theorem \ref{thm:MF_eval_infinite} applies to both $\pi$ and $\pi^\varepsilon$, giving $V^\pi, V^{\pi^\varepsilon} \in \mathcal{C}^{2,2}_{\mathrm{poly}}$ with polynomial-growth constant bounded by $C/(\beta - \bar\beta_\pi)$ uniformly in $|\varepsilon| \leq \varepsilon_0$.

We fix $s \in \mathbb{R}^d$ and write $(s_t^\pi, \mu_t^\pi)_{t \geq 0}$ and $(s_t^\varepsilon, \mu_t^\varepsilon)_{t \geq 0}$ for the representative process and population flow under $\pi$ and $\pi^\varepsilon$ respectively, both starting from $(s,\mu)$. Both value functions admit the stochastic representations
\begin{align}\label{eq:rep-Vpi}
    V^{\pi^\varepsilon}(s,\mu) &= \int_0^\infty e^{-\beta t}
    \mathbb{E}^{s,\mu,\pi^\varepsilon}\!\bigl[r_\lambda^{\pi^\varepsilon}
    (s_t^\varepsilon,\mu_t^\varepsilon)\bigr]\,\d t,\qquad
    V^\pi(s,\mu) = \int_0^\infty e^{-\beta t}
    \mathbb{E}^{s,\mu,\pi}\!\bigl[r_\lambda^\pi(s_t^\pi,\mu_t^\pi)\bigr]\,\d t,
\end{align}
both convergent since $\beta > \beta_0(2)$. We split the difference quotient as
\begin{equation}\label{eq:split}
    \frac{V^{\pi^\varepsilon}(s,\mu) - V^\pi(s,\mu)}{\varepsilon}
    = A(\varepsilon;s) + B(\varepsilon;s),
\end{equation}
where
\begin{align*}
    A(\varepsilon;s) &:= \int_0^\infty e^{-\beta t}
    \frac{\mathbb{E}^{s,\mu,\pi^\varepsilon}\!\bigl[
    r_\lambda^{\pi^\varepsilon}(s_t^\varepsilon,\mu_t^\varepsilon)\bigr]
    -\mathbb{E}^{s,\mu,\pi^\varepsilon}\!\bigl[
    r_\lambda^\pi(s_t^\varepsilon,\mu_t^\varepsilon)\bigr]}
    {\varepsilon}\,\d t,
   \\
    B(\varepsilon;s) &:= \int_0^\infty e^{-\beta t}
    \frac{\mathbb{E}^{s,\mu,\pi^\varepsilon}\!\bigl[
    r_\lambda^\pi(s_t^\varepsilon,\mu_t^\varepsilon)\bigr]
    -\mathbb{E}^{s,\mu,\pi}\!\bigl[
    r_\lambda^\pi(s_t^\pi,\mu_t^\pi)\bigr]}
    {\varepsilon}\,\d t.
\end{align*}
Intuitively, $A(\varepsilon;s)$ captures the sensitivity of the
reward functional to the change in policy. In contrast,
$B(\varepsilon;s)$ captures the effect of the change in dynamics.

\medskip
\noindent\textbf{Step 1: We compute the limit of $A(\varepsilon;s)$.} We first identify the derivative of $\varepsilon\mapsto r_\lambda^{\pi^\varepsilon}(x,\nu)$ at $\varepsilon=0$. Writing $p_{\pi^\varepsilon}(a|x,\nu) = (1+\varepsilon\psi(x,\nu,a))\,p_\pi(a|x,\nu)$, the reward part differentiates as
\[
    \frac{d}{d\varepsilon}\bigg|_{\varepsilon=0}
    \int_\mathcal{A} r(x,\nu,a)\,\pi^\varepsilon(\d a|x,\nu)
    = \int_\mathcal{A} r(x,\nu,a)\,\varphi(\d a|x,\nu).
\]
For the entropy part, the identity $\frac{\d}{\d f}(f\log f) = \log f + 1$ applied to $f = p_{\pi^\varepsilon}(a|x,\nu)$ gives
\[
    \frac{\d}{\d\varepsilon}\bigg|_{\varepsilon=0}
    \left[-\lambda\int_\mathcal{A} p_{\pi^\varepsilon}\log p_{\pi^\varepsilon}\,\nu(\d a)\right]
    = -\lambda\int_\mathcal{A}\bigl(\log p_\pi(a|x,\nu)+1\bigr)\,\varphi(\d a|x,\nu).
\]
Since $\int_\mathcal{A}\varphi(da|x,\nu) = 0$ by \eqref{eq:signed-kernel-zero-mass}, the constant term vanishes. Combining both parts yields
\begin{equation}\label{eq:r-lambda-deriv}
    \frac{\d}{\d\varepsilon}\bigg|_{\varepsilon=0}
    r_\lambda^{\pi^\varepsilon}(x,\nu)
    = \int_\mathcal{A}\bigl(r(x,\nu,a)-\lambda\log p_\pi(a|x,\nu)\bigr)\,
    \varphi(\d a|x,\nu)
    =: r_\lambda^\varphi(x,\nu).
\end{equation}
By the mean value theorem applied to the differentiable map $\varepsilon\mapsto r_\lambda^{\pi^\varepsilon}(x,\nu)$ and \Cref{A1} (with constants uniform in $|\varepsilon|\leq\varepsilon_0$), the ratio satisfies
\begin{equation*}
    \left|\frac{r_\lambda^{\pi^\varepsilon}(x,\nu)
    -r_\lambda^\pi(x,\nu)}{\varepsilon}\right|
    \leq M_\psi C_r\bigl(1+|x|^2+m_2(\nu)\bigr),
    \qquad |\varepsilon|\leq\varepsilon_0,
\end{equation*}
with $C_r > 0$ independent of $\varepsilon$. Taking expectations against the $\pi^\varepsilon$-dynamics, using Lemma \ref{lem:moment-growth-clean} with $p=2$, and noting that $m_2(\mu_t^\varepsilon)$ is deterministic since there is no common noise, we obtain
\begin{equation}\label{eq:A-dominated}
    e^{-\beta t}\left|\frac{\mathbb{E}^{s,\mu,\pi^\varepsilon}[r_\lambda^{\pi^\varepsilon}
    (s_t^\varepsilon,\mu_t^\varepsilon)]
    -\mathbb{E}^{s,\mu,\pi^\varepsilon}[r_\lambda^\pi
    (s_t^\varepsilon,\mu_t^\varepsilon)]}{\varepsilon}\right|
    \leq M_\psi C_r\,
    e^{-(\beta-\beta_0(2))t}\bigl(1+|s|^2+m_2(\mu)\bigr),
\end{equation}
which is integrable since $\beta > \beta_0(2)$. We now split $A(\varepsilon;s)$ as $A(\varepsilon;s) = A_1(\varepsilon;s) + A_2(\varepsilon;s),$ where
\begin{align*}
    A_1(\varepsilon;s) &:= \int_0^\infty e^{-\beta t}
    \mathbb{E}^{s,\mu,\pi^\varepsilon}\!\left[
    \frac{r_\lambda^{\pi^\varepsilon}(s_t^\varepsilon,\mu_t^\varepsilon)
    -r_\lambda^\pi(s_t^\varepsilon,\mu_t^\varepsilon)}{\varepsilon}
    -r_\lambda^\varphi(s_t^\varepsilon,\mu_t^\varepsilon)
    \right]\d t,\\
    A_2(\varepsilon;s) &:= \int_0^\infty e^{-\beta t}
    \mathbb{E}^{s,\mu,\pi^\varepsilon}\!\bigl[
    r_\lambda^\varphi(s_t^\varepsilon,\mu_t^\varepsilon)\bigr]\,\d t.
\end{align*}
The integrand of $A_1$ converges pointwise to zero as $\varepsilon\to0$ by \eqref{eq:r-lambda-deriv}, and is dominated by twice the right-hand side of \eqref{eq:A-dominated}, so dominated convergence gives $A_1(\varepsilon;s)\to0$ as $\varepsilon\to0$.

For $A_2$, fix $t\geq 0$. Then, for each $\theta\in[0,1]$, define $\tilde s_t^\theta:=(1-\theta)\tilde s_t^\pi+\theta \tilde s_t^\varepsilon,
$ and $
\mu_t^\theta:=\Law(\tilde s_t^\theta).$ Since $r_\lambda^\varphi\in\mathcal{C}^{2,2}_{\mathrm{poly}}$, we have
\begin{align*}
    \bigl|\mathbb{E}[&r_\lambda^\varphi(s_t^\varepsilon,\mu_t^\varepsilon)]
    -\mathbb{E}[r_\lambda^\varphi(s_t^\pi,\mu_t^\pi)]\bigr|
    \leq \mathbb{E}\Bigl[\bigl|r_\lambda^\varphi(s_t^\varepsilon,\mu_t^\varepsilon)
    -r_\lambda^\varphi(s_t^\pi,\mu_t^\varepsilon)\bigr|\Bigr]
    +\mathbb{E}\Bigl[\bigl|r_\lambda^\varphi(s_t^\pi,\mu_t^\varepsilon)
    -r_\lambda^\varphi(s_t^\pi,\mu_t^\pi)\bigr|\Bigr]\\
    &\leq \mathbb{E}\Bigl[
    |s_t^\varepsilon-s_t^\pi|
    \int_0^1
    \bigl|\nabla_s r_\lambda^\varphi\bigl(
    s_t^\pi+\theta(s_t^\varepsilon-s_t^\pi),\mu_t^\varepsilon
    \bigr)\bigr|\,\d\theta\Bigr]+\int_0^1
    \mathbb{E}\Bigl[
    \bigl|
    \partial_\mu r_\lambda^\varphi(s_t^\pi,\mu_t^\theta)(\tilde s_t^\theta)
    \bigr|
    \,|\tilde s_t^\varepsilon-\tilde s_t^\pi|
    \Bigr]\,\d\theta\\
    &\leq C\,\mathbb{E}\Bigl[
    |s_t^\varepsilon-s_t^\pi|
    \bigl(1+|s_t^\varepsilon|+|s_t^\pi|+m_2(\mu_t^\varepsilon)^{1/2}\bigr)\Bigr] +C\int_0^1
    \mathbb{E}\Bigl[
    \bigl(1+|s_t^\pi|+|\tilde s_t^\theta|+m_2(\mu_t^\theta)^{1/2}\bigr)
    |\tilde s_t^\varepsilon-\tilde s_t^\pi|
    \Bigr]\,\d\theta\\
    &\leq C\,\mathbb{E}[|s_t^\varepsilon-s_t^\pi|^2]^{1/2}
    \Bigl(
    1+\mathbb{E}[|s_t^\varepsilon|^2]^{1/2}
    +\mathbb{E}[|s_t^\pi|^2]^{1/2}
    +m_2(\mu_t^\varepsilon)^{1/2}
    \Bigr)\\
    &\qquad +C\int_0^1
    \Bigl(
    1+\mathbb{E}[|s_t^\pi|^2]^{1/2}
    +2\,m_2(\mu_t^\theta)^{1/2}
    \Bigr)
    \mathcal W_2(\mu_t^\varepsilon,\mu_t^\pi)\,\d\theta\\
    &\leq C\,\mathbb{E}[|s_t^\varepsilon-s_t^\pi|^2]^{1/2}
    \Bigl(
    1+\mathbb{E}[|s_t^\varepsilon|^2]^{1/2}
    +\mathbb{E}[|s_t^\pi|^2]^{1/2}
    +m_2(\mu_t^\varepsilon)^{1/2}
    \Bigr)\\
    &\qquad +C\Bigl(
    1+\mathbb{E}[|s_t^\pi|^2]^{1/2}
    +m_2(\mu_t^\pi)^{1/2}
    +m_2(\mu_t^\varepsilon)^{1/2}
    \Bigr)\mathcal W_2(\mu_t^\varepsilon,\mu_t^\pi).
\end{align*}
By Lemma \ref{lem:stability-policy}, for each fixed $t\geq0$ the right-hand side tends to $0$ as $\varepsilon\to0$. Hence $\mathbb{E}^{s,\mu,\pi^\varepsilon}\!\bigl[
    r_\lambda^\varphi(s_t^\varepsilon,\mu_t^\varepsilon)\bigr]
    \to
    \mathbb{E}^{s,\mu,\pi}\!\bigl[
    r_\lambda^\varphi(s_t^\pi,\mu_t^\pi)\bigr]$ as $\varepsilon\to0$. The dominating function for dominated convergence over $t$ is the same as in \eqref{eq:A-dominated} (with $r_\lambda^\varphi$ in place of the ratio), and combining both terms gives
\begin{equation}\label{eq:A-limit}
    \lim_{\varepsilon\to0} A(\varepsilon;s)
    = \int_0^\infty e^{-\beta t}
    \mathbb{E}^{s,\mu,\pi}\!\bigl[r_\lambda^\varphi(s_t^\pi,\mu_t^\pi)\bigr]\,\d t.
\end{equation}

\medskip
\noindent\textbf{Step 2: We derive an exact formula for $B(\varepsilon;s)$.} We apply the It\^o-Lions formula to $V^\pi$ along the $\pi^\varepsilon$-dynamics. For $n\geq1$, let $\tau_n^\varepsilon := \inf\{t\geq0 : |s_t^\varepsilon|\geq n\}$. Applying the It\^o-Lions formula to $t\mapsto e^{-\beta(t\wedge\tau_n^\varepsilon)} V^\pi(s_{t\wedge\tau_n^\varepsilon}^\varepsilon,\mu_{t\wedge\tau_n^\varepsilon}^\varepsilon)$ and taking expectations (the stochastic integral has zero expectation since $|s_u^\varepsilon|\leq n$ on $[0,T\wedge\tau_n^\varepsilon]$ ensures the integrand is bounded) yields
\begin{align*}
    \mathbb{E}\bigl[e^{-\beta(T\wedge\tau_n^\varepsilon)}
    V^\pi(s_{T\wedge\tau_n^\varepsilon}^\varepsilon,
    \mu_{T\wedge\tau_n^\varepsilon}^\varepsilon)\bigr]
    - V^\pi(s,\mu)
    = \mathbb{E}\!\left[\int_0^{T\wedge\tau_n^\varepsilon}
    e^{-\beta t}(\mathcal{L}^{\pi^\varepsilon}-\beta)
    V^\pi(s_t^\varepsilon,\mu_t^\varepsilon)\,\d t\right].
\end{align*}
Since $\mathcal{L}^{\pi^\varepsilon} = \mathcal{L}^\pi + \varepsilon\mathcal{L}^\varphi$ by linearity of the generator in the averaged coefficients, and since $(\mathcal{L}^\pi-\beta)V^\pi = -r_\lambda^\pi$ by the HJB equation \eqref{eq:HJB-eval}, this simplifies to
\begin{equation*}
    \mathbb{E}\bigl[e^{-\beta(T\wedge\tau_n^\varepsilon)}
    V^\pi(s_{T\wedge\tau_n^\varepsilon}^\varepsilon,
    \mu_{T\wedge\tau_n^\varepsilon}^\varepsilon)\bigr]
    - V^\pi(s,\mu)
    = \mathbb{E}\!\left[\int_0^{T\wedge\tau_n^\varepsilon}
    e^{-\beta t}\bigl(-r_\lambda^\pi
    +\varepsilon\mathcal{L}^\varphi V^\pi\bigr)
    (s_t^\varepsilon,\mu_t^\varepsilon)\,\d t\right].
\end{equation*}
As $n\to\infty$, $\tau_n^\varepsilon\to\infty$ almost surely. Since $|V^\pi(x,\nu)|\leq C(1+|x|^2+m_2(\nu))$ and $\mathbb{E}[\sup_{t\leq T}|s_t^\varepsilon|^2]<\infty$ by \eqref{eq:finite_T_estimation_s}, dominated convergence gives
\begin{equation}\label{eq:ito-no-tau}
    \mathbb{E}\bigl[e^{-\beta T}V^\pi(s_T^\varepsilon,\mu_T^\varepsilon)\bigr]
    - V^\pi(s,\mu)
    = \int_0^T e^{-\beta t}\mathbb{E}\bigl[(-r_\lambda^\pi
    +\varepsilon\mathcal{L}^\varphi V^\pi)
    (s_t^\varepsilon,\mu_t^\varepsilon)\bigr]\,\d t.
\end{equation}
Since $V^\pi\in\mathcal{C}^{2,2}_{\mathrm{poly}}$ and $\beta>\beta_0(2)$, the boundary term satisfies
\[
    e^{-\beta T}\mathbb{E}\bigl[|V^\pi(s_T^\varepsilon,\mu_T^\varepsilon)|\bigr]
    \leq Ce^{-(\beta-\beta_0(2))T}(1+|s|^2+m_2(\mu))
    \to 0
    \qquad\text{as }T\to\infty.
\]
For the right-hand side, since $m_2(\mu_t^\varepsilon)$ is deterministic and Lemma \ref{lem:moment-growth-clean} gives $\mathbb{E}[|s_t^\varepsilon|^2]\leq Ce^{\beta_0(2)t}(1+|s|^2+m_2(\mu))$ and $m_2(\mu_t^\varepsilon)\leq Ce^{\beta_0(2)t}(1+m_2(\mu))$, the generator term satisfies
\begin{equation}\label{eq:LphiV-bound}
    e^{-\beta t}\mathbb{E}\bigl[|\mathcal{L}^\varphi V^\pi
    (s_t^\varepsilon,\mu_t^\varepsilon)|\bigr]
    \leq Ce^{-\beta t}\bigl(
    e^{\beta_0(2)t}(1+|s|^2+m_2(\mu))
    +e^{2\beta_0(2)t}(1+m_2(\mu))^2
    \bigr),
\end{equation}
which is integrable over $[0,\infty)$ since $\beta>2\beta_0(2)$. Similarly for $r_\lambda^\pi$. Letting $T\to\infty$ in \eqref{eq:ito-no-tau} by dominated convergence gives
\begin{equation}\label{eq:identity-final}
    V^\pi(s,\mu)
    = \int_0^\infty e^{-\beta t}
    \mathbb{E}^{s,\mu,\pi^\varepsilon}\!\bigl[
    r_\lambda^\pi(s_t^\varepsilon,\mu_t^\varepsilon)- \varepsilon
    \mathcal{L}^\varphi V^\pi(s_t^\varepsilon,\mu_t^\varepsilon)\bigr]\,\d t.
\end{equation}
From the definition of $B(\varepsilon;s)$ and the representations
\eqref{eq:identity-final} and \eqref{eq:rep-Vpi}, we obtain
\begin{equation}\label{eq:B-exact}
\begin{aligned}
    B(\varepsilon;s)
    &= \int_0^\infty e^{-\beta t}
    \frac{
    \mathbb{E}^{s,\mu,\pi^\varepsilon}\!\bigl[
    r_\lambda^\pi(s_t^\varepsilon,\mu_t^\varepsilon)\bigr]
    -\mathbb{E}^{s,\mu,\pi}\!\bigl[
    r_\lambda^\pi(s_t^\pi,\mu_t^\pi)\bigr]}
    {\varepsilon}\,\d t\\
    &= \int_0^\infty e^{-\beta t}
    \frac{
    \mathbb{E}^{s,\mu,\pi^\varepsilon}\!\bigl[
    r_\lambda^\pi(s_t^\varepsilon,\mu_t^\varepsilon)\bigr]
    -\mathbb{E}^{s,\mu,\pi^\varepsilon}\!\bigl[
    r_\lambda^\pi(s_t^\varepsilon,\mu_t^\varepsilon)
    -\varepsilon \mathcal{L}^\varphi V^\pi(s_t^\varepsilon,\mu_t^\varepsilon)
    \bigr]}
    {\varepsilon}\,\d t\\
    &= \int_0^\infty e^{-\beta t}
    \mathbb{E}^{s,\mu,\pi^\varepsilon}\!\bigl[
    \mathcal{L}^\varphi V^\pi(s_t^\varepsilon,\mu_t^\varepsilon)\bigr]\,\d t,
\end{aligned}
\end{equation}
where in the second equality we used \eqref{eq:identity-final}. This yields \eqref{eq:B-exact},
valid for all $\varepsilon\neq0$. For each fixed $t\geq0$, the same argument applied to $A_2$ above (fundamental theorem of calculus in state, Lions mean value formula in measure) applies here to $\mathcal{L}^\varphi V^\pi\in\mathcal{C}^{2,2}_{\mathrm{poly}}$, using the convergence $\mathbb{E}[|s_t^\varepsilon-s_t^\pi|^2] + \mathcal{W}_2^2(\mu_t^\varepsilon,\mu_t^\pi)\to0$ as $\varepsilon\to0$ from Lemma \ref{lem:stability-policy}, to give
\[
    \mathbb{E}^{s,\mu,\pi^\varepsilon}\!\bigl[
    \mathcal{L}^\varphi V^\pi(s_t^\varepsilon,\mu_t^\varepsilon)\bigr]
    \to
    \mathbb{E}^{s,\mu,\pi}\!\bigl[
    \mathcal{L}^\varphi V^\pi(s_t^\pi,\mu_t^\pi)\bigr]
    \qquad\text{as }\varepsilon\to0.
\]
Since the dominating function from \eqref{eq:LphiV-bound} is integrable over $[0,\infty)$, dominated convergence in \eqref{eq:B-exact} gives
\begin{equation}\label{eq:B-limit}
    \lim_{\varepsilon\to0} B(\varepsilon;s)
    = \int_0^\infty e^{-\beta t}
    \mathbb{E}^{s,\mu,\pi}\!\bigl[
    \mathcal{L}^\varphi V^\pi(s_t^\pi,\mu_t^\pi)\bigr]\,\d t.
\end{equation}

\medskip
\noindent\textbf{Step 3: We identify the limit and conclude.} Combining \eqref{eq:split}, \eqref{eq:A-limit}, and \eqref{eq:B-limit} gives
\begin{equation}\label{eq:pointwise-limit}
    \lim_{\varepsilon\to0}
    \frac{V^{\pi^\varepsilon}(s,\mu)-V^\pi(s,\mu)}{\varepsilon}
    = \int_0^\infty e^{-\beta t}
    \mathbb{E}^{s,\mu,\pi}\!\bigl[
    r_\lambda^\varphi(s_t^\pi,\mu_t^\pi)
    +(\mathcal{L}^\varphi V^\pi)(s_t^\pi,\mu_t^\pi)
    \bigr]\,\d t.
\end{equation}
Note that by definition of $\mathcal{L}^\varphi$ we deduce by Fubini's theorem (justified by polynomial growth and \Cref{A1}), that
\begin{align}\label{eq:charc_hamiltoniana_mixed}
    r_\lambda^\varphi(s,\mu)
    &+(\mathcal{L}^\varphi V^\pi)(s,\mu) \notag\\
    &= \int_\mathcal{A}\!\Bigl[r(s,\mu,a)-\lambda\log p_\pi(a|s,\mu)
    +b(s,\mu,a)\cdot\nabla_s V^\pi(s,\mu)
    +\tfrac{1}{2}\Sigma(s,\mu,a):D^2_{ss}V^\pi(s,\mu)\Bigr]
    \varphi(\d a|s,\mu) \notag\\
    &\quad +\int_{\mathbb{R}^d}\!\int_\mathcal{A}\!\Bigl[
    b(\xi,\mu,a)\cdot\partial_\mu V^\pi(s,\mu)(\xi)
    +\tfrac{1}{2}\Sigma(\xi,\mu,a):D_\xi\partial_\mu V^\pi(s,\mu)(\xi)
    \Bigr]\varphi(\d a|\xi,\mu)\,\mu(\d\xi). \nonumber\\
    &=\int_\mathcal{A}q^{\pi}_{\mathrm{rep}}(s,\mu,a)\,\varphi(\d a|s,\mu)+\int_{\mathbb{R}^d}\int_\mathcal{A}q^{\pi}_{\mathrm{pop}}(s,\mu,\xi,a)\,\varphi(\d a|\xi,\mu)\,\mu(\d\xi)
\end{align}
It remains to integrate over $\mu_0$. From \eqref{eq:A-dominated} integrated over $[0,\infty)$ and \eqref{eq:B-exact} bounded via \eqref{eq:LphiV-bound}, for all $|\varepsilon|\leq\varepsilon_0$:
\begin{align*}
    \left|\frac{V^{\pi^\varepsilon}(s,\mu)-V^\pi(s,\mu)}{\varepsilon}\right|
    &\leq |A(\varepsilon;s)|+|B(\varepsilon;s)| \notag\\
    &\leq \frac{M_\psi C_r(1+|s|^2+m_2(\mu))}{\beta-\beta_0(2)}
    + C\left[\frac{1+|s|^2+m_2(\mu)}{\beta-\beta_0(2)}
    +\frac{(1+m_2(\mu))^2}{\beta-2\beta_0(2)}\right] \notag\\
    &\leq K_\mu(1+|s|^2),
\end{align*}
where $K_\mu > 0$ depends on the fixed measure $\mu$ and the structural constants but not on $s$ or $\varepsilon$. The factor $(1+m_2(\mu))^2$ is absorbed into $K_\mu$ because $m_2(\mu_t^\varepsilon)$ is deterministic and was already used to bound the expectation; the remaining $s$-dependence is thus at most quadratic. Since $\mu_0\in\mathcal{P}_2(\mathbb{R}^d)$, we have $\int_{\mathbb{R}^d}K_\mu(1+|s|^2)\,\mu_0(ds) = K_\mu(1+m_2(\mu_0)) < \infty$. Dominated convergence over $\mu_0(ds)$ then gives
\[
    \frac{d}{d\varepsilon}J(\varepsilon)\bigg|_{\varepsilon=0}
    = \frac{1}{\beta}\left(\beta\int_{\mathbb{R}^d}
    \lim_{\varepsilon\to0}
    \frac{V^{\pi^\varepsilon}(s,\mu)-V^\pi(s,\mu)}{\varepsilon}
    \,\mu_0(\d s)\right),
\]
and substituting the pointwise limit \eqref{eq:pointwise-limit}, and \eqref{eq:charc_hamiltoniana_mixed}, and using \Cref{def:discounted_measure}, gives \eqref{eq:gateaux-pg}.
\end{proof}

We now provide the proof \Cref{thm:parametric-policy-gradient}. 

\begin{proof}[\textbf{Proof of Theorem \ref{thm:parametric-policy-gradient}}]

Fix $v\in\mathbb{R}^p$. We show that
$\frac{d}{d\varepsilon}J(\omega+\varepsilon v)|_{\varepsilon=0}$ equals
the right-hand side of \eqref{eq:parametric-pg} dotted with $v$; since
$v$ is arbitrary, this identifies $\nabla_\omega J(\omega)$. 

\medskip
\noindent\textbf{Directional kernel.} Define the signed kernel
\begin{equation*}
    \varphi_v(\d a\mid s,\mu)
    :=\bigl(\nabla_\omega\log p_\pi(a\mid s,\mu)\cdot v\bigr)\,
    \pi(\d a\mid s,\mu).
\end{equation*}
The zero-mass condition $\int_\mathcal{A}\varphi_v(da\mid s,\mu)=0$
holds since $\int_\mathcal{A}\nabla_\omega\log p_{\pi_\omega}(a\mid s,\mu)
\pi_\omega(da\mid s,\mu)=\nabla_\omega\int_\mathcal{A}p_{\pi_\omega}\nu(da)=0$,
which follows from \Cref{A2} and differentiation under the integral sign.
We write $(s_t^{\varepsilon v},\mu_t^{\varepsilon v})_{t\ge0}$ for the
dynamics under $\pi_{\omega+\varepsilon v}$ starting from $(s,\mu)$, and
$(s_t^\pi,\mu_t^\pi)_{t\ge0}$ for those under $\pi_\omega$. By \Cref{A3},
Lemma \ref{lem:stability-policy} applies and gives, for every $t\ge0$,
\begin{equation}\label{eq:param-stability}
    \mathbb{E}\bigl[|s_t^{\varepsilon v}-s_t^\pi|^2\bigr]
    +\mathcal{W}_2^2(\mu_t^{\varepsilon v},\mu_t^\pi)
    \leq C\varepsilon^2 e^{\beta_{\rm stab}\,t}
    \bigl(1+|s|^2+m_2(\mu)\bigr),
\end{equation}
so both terms tend to zero as $\varepsilon\to0$ for each fixed $t\ge0$.
For each fixed $s\in\mathbb{R}^d$, we split the difference quotient as
\begin{equation}\label{eq:param-split}
    \frac{V^{\pi_{\omega+\varepsilon v}}(s,\mu)-V^{\pi}(s,\mu)}{\varepsilon}
    = A(\varepsilon;s) + B(\varepsilon;s),
\end{equation}
where
\begin{align*}
    A(\varepsilon;s)
    &:= \int_0^\infty e^{-\beta t}
    \frac{\mathbb{E}^{s,\mu,\pi_{\omega+\varepsilon v}}\!
    \bigl[r_\lambda^{\pi_{\omega+\varepsilon v}}(s_t^{\varepsilon v},
    \mu_t^{\varepsilon v})\bigr]
    -\mathbb{E}^{s,\mu,\pi_{\omega+\varepsilon v}}\!
    \bigl[r_\lambda^{\pi}(s_t^{\varepsilon v},\mu_t^{\varepsilon v})\bigr]}
    {\varepsilon}\,\d t,\\
    B(\varepsilon;s)
    &:= \int_0^\infty e^{-\beta t}
    \frac{\mathbb{E}^{s,\mu,\pi_{\omega+\varepsilon v}}\!
    \bigl[r_\lambda^{\pi}(s_t^{\varepsilon v},\mu_t^{\varepsilon v})\bigr]
    -\mathbb{E}^{s,\mu,\pi}\!
    \bigl[r_\lambda^{\pi}(s_t^{\pi},\mu_t^{\pi})\bigr]}
    {\varepsilon}\,\d t.
    \end{align*}

\medskip
\noindent\textbf{Step 1: We compute the limit of $A(\varepsilon;s)$.} By \Cref{A2}, the map $\varepsilon\mapsto r_\lambda^{\pi_{\omega+\varepsilon v}}(x,\nu)$
is differentiable at $\varepsilon=0$ with derivative $\nabla_\omega
r_\lambda^{\pi_\omega}(x,\nu)\cdot v$. By the zero-mass property
of $\varphi_v$, this derivative equals $r_\lambda^{\varphi_v}(x,\nu)$
as defined in \eqref{eq:r-lambda-deriv} (with $\varphi_v$ in place of $\varphi$).
By the mean value theorem and the bound \eqref{eq:A2-Sigma_r} on $\nabla_\omega
r_\lambda^{\pi_\omega}$, uniformly in $\omega$ on compact sets:
\begin{equation}\label{eq:ratio-r-bound-param}
    \left|\frac{r_\lambda^{\pi_{\omega+\varepsilon v}}(x,\nu)
    -r_\lambda^\pi(x,\nu)}{\varepsilon}\right|
    \leq C_K\|v\|\bigl(1+|x|^2+m_2(\nu)\bigr),
    \qquad |\varepsilon|\leq\varepsilon_0.
\end{equation}
The rest of the argument for $A$ follows the same $A_1+A_2$ decomposition
as in Step 1 of Theorem \ref{thm:policy-gradient-theorem}, using
\eqref{eq:ratio-r-bound-param} in place of the previous uniform bound
and \eqref{eq:param-stability} for the convergence of the dynamics.
The dominating function is $C_K\|v\|e^{-(\beta-\beta_0(2))t}
(1+|s|^2+m_2(\mu))$, integrable since $\beta>\beta_0(2)$, giving
\begin{equation}\label{eq:A-limit-param}
    \lim_{\varepsilon\to0}A(\varepsilon;s)
    =\int_0^\infty e^{-\beta t}
    \mathbb{E}^{s,\mu,\pi}\!\bigl[
    r_\lambda^{\varphi_v}(s_t^\pi,\mu_t^\pi)\bigr]\,\d t.
\end{equation}

\medskip
\noindent\textbf{Step 2: We derive a formula for $B(\varepsilon;s)$
  and compute its limit.}

We apply the It\^o-Lions formula to $V^\pi$ along the
$\pi_{\omega+\varepsilon v}$-dynamics. For $n\ge1$, let
$\tau_n^{\varepsilon v}:=\inf\{t\ge0:|s_t^{\varepsilon v}|\ge n\}$.
Applying the formula to $t\mapsto e^{-\beta(t\wedge\tau_n^{\varepsilon v})}
V^\pi(s_{t\wedge\tau_n^{\varepsilon v}}^{\varepsilon v},
\mu_{t\wedge\tau_n^{\varepsilon v}}^{\varepsilon v})$, taking expectations,
decomposing $\mathcal{L}^{\pi_{\omega+\varepsilon v}}
=\mathcal{L}^\pi+(\mathcal{L}^{\pi_{\omega+\varepsilon v}}-\mathcal{L}^\pi)$,
and substituting the HJB equation $(\mathcal{L}^\pi-\beta)V^\pi=-r_\lambda^\pi$
from \eqref{eq:HJB-eval} gives
\begin{equation}\label{eq:ito-param}
    \mathbb{E}\bigl[e^{-\beta(T\wedge\tau_n^{\varepsilon v})}
    V^\pi(s_{T\wedge\tau_n^{\varepsilon v}}^{\varepsilon v},
    \mu_{T\wedge\tau_n^{\varepsilon v}}^{\varepsilon v})\bigr]
    -V^\pi(s,\mu)
    = \mathbb{E}\!\left[\int_0^{T\wedge\tau_n^{\varepsilon v}}
    e^{-\beta t}\bigl(-r_\lambda^\pi
    +(\mathcal{L}^{\pi_{\omega+\varepsilon v}}-\mathcal{L}^{\pi})V^\pi\bigr)
    (s_t^{\varepsilon v},\mu_t^{\varepsilon v})\,\d t\right].
\end{equation}
We now bound the generator difference term. By \Cref{A2} and the mean
value theorem, the ratios $\frac{b^{\pi_{\omega+\varepsilon v}}-b^\pi}{\varepsilon}$
and $\frac{\Sigma^{\pi_{\omega+\varepsilon v}}-\Sigma^\pi}{\varepsilon}$
are bounded by $C_K\|v\|(1+|x|+\sqrt{m_2(\nu)})$ and
$C_K\|v\|(1+|x|^2+m_2(\nu))$ respectively, uniformly in
$|\varepsilon|\leq\varepsilon_0$. Using the polynomial bounds on the
derivatives of $V^\pi$ from Definition \ref{def:C22poly} and the
growth of $b^\pi$, $\Sigma^\pi$ from \Cref{A1,A3}, each of the four
terms in $(\mathcal{L}^{\pi_{\omega+\varepsilon v}}-\mathcal{L}^\pi)V^\pi$
is bounded in absolute value. The dominant contribution comes from the
term involving $\frac{\Sigma^{\pi_{\omega+\varepsilon v}}-\Sigma^\pi}{\varepsilon}
:D^2_{ss}V^\pi$, which satisfies
\begin{equation}\label{eq:generator-diff-bound-param}
    \left|\frac{(\mathcal{L}^{\pi_{\omega+\varepsilon v}}-\mathcal{L}^{\pi})
    V^\pi(x,\nu)}{\varepsilon}\right|
    \leq C_K\|v\|\bigl(1+|x|^2+m_2(\nu)^2\bigr),
    \qquad |\varepsilon|\leq\varepsilon_0.
\end{equation}
The passages $n\to\infty$ and then $T\to\infty$ in \eqref{eq:ito-param},
with the subsequent recovery of \eqref{eq:rep-Vpi}, follow by dominated
convergence using the same arguments as in Step 2 of
Theorem \ref{thm:policy-gradient-theorem} (the growth bounds are the same
with $C_K\|v\|$ in place of $M_\psi$, and integrability holds since
$\beta>2\beta_0(2)$). Subtracting the stochastic
representation \eqref{eq:rep-Vpi} and dividing by $\varepsilon\neq0$:
\begin{equation}\label{eq:B-formula-param}
    B(\varepsilon;s)
    = \int_0^\infty e^{-\beta t}
    \frac{1}{\varepsilon}
    \mathbb{E}^{s,\mu,\pi_{\omega+\varepsilon v}}\!\bigl[
    (\mathcal{L}^{\pi_{\omega+\varepsilon v}}-\mathcal{L}^{\pi})
    V^\pi(s_t^{\varepsilon v},\mu_t^{\varepsilon v})\bigr]\,\d t.
\end{equation}
We now pass $\varepsilon\to0$ in \eqref{eq:B-formula-param}. By \Cref{A2},
the ratios $\frac{b^{\pi_{\omega+\varepsilon v}}-b^\pi}{\varepsilon}\to
\nabla_\omega b^{\pi_\omega}\cdot v = b^{\varphi_v}$ and
$\frac{\Sigma^{\pi_{\omega+\varepsilon v}}-\Sigma^\pi}{\varepsilon}\to
\nabla_\omega\Sigma^{\pi_\omega}\cdot v = \Sigma^{\varphi_v}$ pointwise
as $\varepsilon\to0$. Applying these limits to each
of the four terms of the generator, using the polynomial growth of
$V^\pi$ and dominated convergence over $\xi$ in the integral terms
(justified by the $\mathcal{P}_2$ condition on $\mu$), gives
\begin{equation}\label{eq:generator-ratio-limit}
    \frac{(\mathcal{L}^{\pi_{\omega+\varepsilon v}}-\mathcal{L}^{\pi})
    V^\pi(x,\nu)}{\varepsilon}
    \xrightarrow[\varepsilon\to0]{}
    (\mathcal{L}^{\varphi_v}V^\pi)(x,\nu)
    \quad\text{for every }(x,\nu).
\end{equation}
Combining \eqref{eq:generator-ratio-limit} with \eqref{eq:param-stability}
and the same continuity argument applied to the polynomial-growth function
$\mathcal{L}^{\varphi_v}V^\pi$, the integrand
in \eqref{eq:B-formula-param} converges for each fixed $t\ge0$. The
bound \eqref{eq:generator-diff-bound-param} provides a dominating
function integrable over $[0,\infty)$ since $\beta>2\beta_0(2)$, so
dominated convergence gives
\begin{equation}\label{eq:B-limit-param}
    \lim_{\varepsilon\to0}B(\varepsilon;s)
    =\int_0^\infty e^{-\beta t}
    \mathbb{E}^{s,\mu,\pi}\!\bigl[
    \mathcal{L}^{\varphi_v}V^\pi(s_t^\pi,\mu_t^\pi)\bigr]\,\d t.
\end{equation}

\medskip
\noindent\textbf{Step 3: We identify the limit and integrate over $\mu_0$.}

Combining \eqref{eq:param-split}, \eqref{eq:A-limit-param},
and \eqref{eq:B-limit-param}:
\begin{equation}\label{eq:pointwise-param}
    \lim_{\varepsilon\to0}
    \frac{V^{\pi_{\omega+\varepsilon v}}(s,\mu)-V^\pi(s,\mu)}{\varepsilon}
    =\int_0^\infty e^{-\beta t}
    \mathbb{E}^{s,\mu,\pi}\!\bigl[
    r_\lambda^{\varphi_v}(s_t^\pi,\mu_t^\pi)
    +(\mathcal{L}^{\varphi_v}V^\pi)(s_t^\pi,\mu_t^\pi)
    \bigr]\,\d t.
\end{equation}
We expand $r_\lambda^{\varphi_v}+\mathcal{L}^{\varphi_v}V^\pi$ using
the definitions of $r_\lambda^{\varphi_v}$, $\mathcal{L}^{\varphi_v}$,
and $\varphi_v$, and Fubini's theorem (justified by \Cref{A1,A2} and
the polynomial growth of $V^\pi$), we deduce \eqref{eq:charc_hamiltoniana_mixed} with $\varphi$ replaced by $\varphi_v$.  Substituting $\varphi_v(da|x,\nu)=(\nabla_\omega\log p_\pi(a|x,\nu)\cdot v)
\,\pi(da|x,\nu)$ gives
\[
    r_\lambda^{\varphi_v}(s,\mu)+(\mathcal{L}^{\varphi_v}V^\pi)(s,\mu)
    = v\cdot\int_\mathcal{A}
    q^{\pi}_{\mathrm{rep}}\,
    \nabla_\omega\log p_\pi\,\pi(\d a)
    +v\cdot\int_{\mathbb{R}^d}\!\int_\mathcal{A}
    q^{\pi}_{\mathrm{pop}}\,
    \nabla_\omega\log p_\pi\,\pi(\d a)\,\mu(\d\xi),
\]
so \eqref{eq:pointwise-param} gives
$\frac{d}{d\varepsilon}V^{\pi_{\omega+\varepsilon v}}(s,\mu)|_{\varepsilon=0}
= v\cdot R(s)$, where $R(s)$ denotes the right-hand side
of \eqref{eq:parametric-pg} at $s$.

It remains to integrate over $\mu_0$. From \eqref{eq:ratio-r-bound-param}
and \eqref{eq:generator-diff-bound-param}, the difference quotient satisfies
\[
    \left|\frac{V^{\pi_{\omega+\varepsilon v}}(s,\mu)-V^\pi(s,\mu)}{\varepsilon}\right|
    \leq K_\mu\bigl(1+|s|^2\bigr),
    \qquad |\varepsilon|\leq\varepsilon_0,
\]
where $K_\mu>0$ absorbs the deterministic factors involving $m_2(\mu)$
but does not depend on $s$ or $\varepsilon$. Since
$\mu_0\in\mathcal{P}_2(\mathbb{R}^d)$, dominated convergence over
$\mu_0(\d s)$ gives $\frac{\d}{\d\varepsilon}J(\omega+\varepsilon v)
|_{\varepsilon=0}=v\cdot\int_{\mathbb{R}^d}R(s)\,\mu_0(\d s)$.
Since $v\in\mathbb{R}^p$ is arbitrary, $J$ is differentiable at $\omega$,
we deduce that $\nabla_\omega J(\omega)=\int_{\mathbb{R}^d}R(s)\,\mu_0(\d s)$,
which is \eqref{eq:parametric-pg}.
\end{proof}

\subsection{Proofs for \Cref{sec:LQR}}
\label{APP:PROOFS_SEC_LQR}

\begin{lemma}[Stationary Riccati system: existence, uniqueness, and stabilizing solution]
\label{lem:ARE_full_stationary}
Assume \Cref{ass:LQR}. Then there exists a unique quadruple $(K,\Lambda,Y,R)\in\mathbb S_{++}^d\times\mathbb S_{++}^d\times\mathbb R^d\times\mathbb R$
solving \eqref{eq:ARE_P}-\eqref{eq:ARE_r}. Moreover, $S=N+F^\top K F\in\mathbb S_{++}^m,$ and the matrices
\begin{equation}\label{eq:Ac_definition_new}
A_\beta-BS^{-1}U
\qquad\text{and}\qquad
\tilde A_\beta-BS^{-1}W
\end{equation}
are Hurwitz. Furthermore, among all solutions of
\eqref{eq:ARE_P}-\eqref{eq:ARE_r}, this is the unique one for which the matrices in
\eqref{eq:Ac_definition_new} are Hurwitz.
\end{lemma}

\begin{proof}
We first rewrite \eqref{eq:ARE_P}-\eqref{eq:ARE_Pi} in shifted form:
\begin{align}
0
&=
Q+K A_\beta+A_\beta^\top K+D^\top K D-U^\top S^{-1}U,
\label{eq:ARE_shifted_K_lemma}\\
0
&=
Q+\bar Q+\Lambda \tilde A_\beta+\tilde A_\beta^\top \Lambda
+(D+\bar D)^\top K(D+\bar D)-W^\top S^{-1}W.
\label{eq:ARE_shifted_Lambda_lemma}
\end{align}
Then, a direct application of \cite[Theorem 3.4.1]{sun2020stochastic} yields a unique pair $(K,\Lambda)\in\mathbb S_{++}^d\times\mathbb S_{++}^d$ solving \eqref{eq:ARE_shifted_K_lemma}-\eqref{eq:ARE_shifted_Lambda_lemma}, such that $S=N+F^\top K F\in\mathbb S_{++}^m$ and the matrices $A_\beta-BS^{-1}U,$ and $
\tilde A_\beta-BS^{-1}W$ are Hurwitz. Once $(K,\Lambda)$ is fixed, equation \eqref{eq:ARE_Y} is linear in $Y$. Since $O=H+B^\top Y+F^\top K\gamma,$ we rewrite \eqref{eq:ARE_Y} as
\[
\Bigl((A-BS^{-1}U)^\top-\beta I_d\Bigr)Y
+
M
+
D^\top K\gamma
-
U^\top S^{-1}(H+F^\top K\gamma)
=
0.
\]
Because $A_\beta-BS^{-1}U$ is Hurwitz, the matrix
\[
(A-BS^{-1}U)-\beta I_d
=
\bigl(A_\beta-BS^{-1}U\bigr)-\frac{\beta}{2}I_d
\]
is also Hurwitz, hence invertible. Therefore the above linear equation admits a unique solution
$Y\in\mathbb R^d$. Finally, once $(K,\Lambda,Y)$ is fixed, equation \eqref{eq:ARE_r} is linear in $R$, and since
$\beta>0$ it admits the unique solution.
\end{proof}

\begin{lemma}[Discounted second-moment bound and transversality under $\pi^*$]
\label{lem:transversality_pi_star}
Assume \Cref{ass:LQR}, and let $\pi^*$ be given by \eqref{eq:pi_star}. Let
$(s_t,\tilde s_t,\mu_t)$ be the corresponding solution of
\eqref{eq:rep_SDE}-\eqref{eq:MKV_SDE} under $\pi^*$, with
$s_0=s$, $\tilde s_0\sim\mu$, and $\mu_t=\Law(\tilde s_t)$. Set $\xi_t:=s_t-\bar\mu_t,$ and $\tilde\xi_t:=\tilde s_t-\bar\mu_t.$ Then there exist constants $C\ge1$ and $\kappa>0$, independent of $t$, such that for all $t\ge0$,
\begin{equation}\label{eq:discounted_second_moment_bound}
e^{-\beta t}\Bigl(\mathbb E|\xi_t|^2+\mathbb E|\tilde\xi_t|^2+|\bar\mu_t|^2\Bigr)
\le
C\,e^{-2\kappa t}\Bigl(1+|s|^2+m_2(\mu)+\lambda\Bigr).
\end{equation}
In particular, $\lim_{t\to\infty}e^{-\beta t}\bigl(\mathbb E|s_t|^2+m_2(\mu_t)\bigr) = 0$ and therefore, the function $V$ from \eqref{eq:V_candidate_pf2} satisfies that
\begin{align}\label{eq:transitivity_conditon_V_pi_star}
    \lim_{t\to\infty}e^{-\beta t}\,\mathbb E|V(s_t,\mu_t)| = 0
\end{align}
\end{lemma}

\begin{proof}
Define the feedback coefficients $\Theta^*:=-S^{-1}U,$  $\bar\Theta^*:=-S^{-1}W,$ and $\theta^*:=-S^{-1}O$, so that the mean of the Gaussian kernel $\pi^*$ is $\hat a(s,\mu)
=
\Theta^*(s-\bar\mu)+\bar\Theta^*\,\bar\mu+\theta^*.$ Hence
\[
b^{\pi^*}(s,\mu)
=
(A+B\Theta^*)(s-\bar\mu)
+
(\tilde A+B\bar\Theta^*)\bar\mu
+
B\theta^*,
\]
where $\tilde A:=A+\bar A$. Moreover, if we set $C^*:=D+F\Theta^*,$ $\bar C^*:=D+\bar D+F\bar\Theta^*,$ and $\sigma_0^*:=\gamma+F\theta^*,$ then
\[
\Sigma^{\pi^*}(s,\mu)
=
\bigl(C^*(s-\bar\mu)+\bar C^*\,\bar\mu+\sigma_0^*\bigr)
\bigl(C^*(s-\bar\mu)+\bar C^*\,\bar\mu+\sigma_0^*\bigr)^\top
+
\frac{\lambda}{2}FS^{-1}F^\top.
\]
Taking expectation in the population dynamics yields
\begin{equation*}
\dot{\bar\mu}_t
=
(\tilde A+B\bar\Theta^*)\bar\mu_t+B\theta^*.
\end{equation*}
By Lemma \ref{lem:ARE_full_stationary},
$\tilde A_\beta+B\bar\Theta^*
=
\tilde A_\beta-BS^{-1}W$
is Hurwitz. Therefore there exist constants $C_m\ge1$ and $\kappa_m>0$ such that
\begin{equation}\label{eq:m_discounted_bound}
e^{-\beta t}|\bar\mu_t|^2
\le
C_m e^{-2\kappa_m t}\bigl(1+|m_0|^2\bigr).
\end{equation}
Next, we derive the Lyapunov inequality needed for the fluctuation
$\xi_t=s_t-\bar\mu_t$. Set $M^*:=A+B\Theta^*.$ After completing squares in \eqref{eq:ARE_shifted_K_lemma}, we have
\[
K(A_\beta+B\Theta^*)
+
(A_\beta+B\Theta^*)^\top K
+
(D+F\Theta^*)^\top K(D+F\Theta^*)
=
-
\mathcal Q_{\Theta^*},
\]
where $\mathcal Q_{\Theta^*}:=
Q
+
(\Theta^*)^\top N\Theta^*
+
(\Theta^*)^\top I
+
I^\top\Theta^*.$ Equivalently, since $A_\beta=A-\frac{\beta}{2}I_d$,
\[
(M^*)^\top K
+
K M^*
+
(C^*)^\top K C^*
=
\beta K-\mathcal Q_{\Theta^*}.
\]
Moreover, by \Cref{ass:LQR}, $\mathcal Q_{\Theta^*}
=
Q-I^\top N^{-1}I
+
(\Theta^*+N^{-1}I)^\top N(\Theta^*+N^{-1}I)
\succ 0.$ Since $K\in\mathbb S_{++}^d$, the two positive definite quadratic forms
$\mathcal Q_{\Theta^*}$ and $K$ are comparable. Hence there exists
$\kappa_\xi>0$ such that $\mathcal Q_{\Theta^*}\succeq 2\kappa_\xi K.$ Therefore,
\[
(M^*)^\top K
+
K M^*
+
(C^*)^\top K C^*
=
\beta K-\mathcal Q_{\Theta^*}
\preceq
\beta K-2\kappa_\xi K
=
(\beta-2\kappa_\xi)K.
\]
Thus, with $G:=K$, we obtain
\begin{equation}\label{eq:lyapunov_centered_pi_star}
(A+B\Theta^*)^\top G
+
G(A+B\Theta^*)
+
(D+F\Theta^*)^\top G(D+F\Theta^*)
\preceq
(\beta-2\kappa_\xi)G.
\end{equation}
We now apply It\^o's formula to the quadratic function $x\mapsto x^\top Gx$
along the fluctuation processes $\xi_t$ and $\tilde\xi_t$. Set
$h_t:=\bar C^*\bar\mu_t+\sigma_0^*.$ For $\xi_t$, the drift is $M^*\xi_t$ and
the covariance matrix is $(C^*\xi_t+h_t)(C^*\xi_t+h_t)^\top
+
\frac{\lambda}{2}FS^{-1}F^\top .$ Hence,
\[
\begin{aligned}
\frac{\d}{\d t}\mathbb E[\xi_t^\top G\xi_t]
&=
\mathbb E\!\left[
\xi_t^\top\bigl((M^*)^\top G+GM^*\bigr)\xi_t
\right]
+
\mathbb E\!\left[
(C^*\xi_t+h_t)^\top G(C^*\xi_t+h_t)
\right]
+
\frac{\lambda}{2}\Tr(GFS^{-1}F^\top)
\\
&=
\mathbb E\!\left[
\xi_t^\top\bigl((M^*)^\top G+GM^*+(C^*)^\top G C^*\bigr)\xi_t
\right]
+
2\mathbb E\!\left[\xi_t^\top (C^*)^\top G h_t\right]
\\
&\qquad
+
h_t^\top Gh_t
+
\frac{\lambda}{2}\Tr(GFS^{-1}F^\top).
\end{aligned}
\]
Using \eqref{eq:lyapunov_centered_pi_star}, we get
\[
\begin{aligned}
\frac{\d}{\d t}\mathbb E[\xi_t^\top G\xi_t]
&\le
(\beta-2\kappa_\xi)\mathbb E[\xi_t^\top G\xi_t]
+
2\mathbb E\!\left[\xi_t^\top (C^*)^\top G h_t\right]
+
h_t^\top Gh_t
+
\frac{\lambda}{2}\Tr(GFS^{-1}F^\top).
\end{aligned}
\]
The cross term is controlled by Young's inequality: for every $\varepsilon>0$, $2\xi_t^\top (C^*)^\top G h_t
\le
\varepsilon\, \xi_t^\top G\xi_t
+
C_\varepsilon |h_t|^2.$
Moreover, since $h_t=\bar C^*\bar\mu_t+\sigma_0^*$, we have
$|h_t|^2\le C(1+|\bar\mu_t|^2)$, while
$\frac{\lambda}{2}\Tr(GFS^{-1}F^\top)\le C\lambda.$ Therefore,
possibly reducing $\kappa_\xi>0$ and keeping the same notation, we obtain
\[
\frac{\d}{\d t}\mathbb E[\xi_t^\top G\xi_t]
\le
(\beta-2\kappa_\xi)\mathbb E[\xi_t^\top G\xi_t]
+
C\bigl(1+|\bar\mu_t|^2+\lambda\bigr).
\]
The same computation applied to $\tilde\xi_t$ gives
\[
\frac{\d}{\d t}\mathbb E[\tilde\xi_t^\top G\tilde\xi_t]
\le
(\beta-2\kappa_\xi)\mathbb E[\tilde\xi_t^\top G\tilde\xi_t]
+
C\bigl(1+|\bar\mu_t|^2+\lambda\bigr),
\]
for some constant $C>0$ independent of $t$. Multiplying by $e^{-\beta t}$ gives
\[
\frac{\d}{\d t}\Bigl(e^{-\beta t}\mathbb E[\xi_t^\top G\xi_t]\Bigr)
\le
-2\kappa_\xi\,e^{-\beta t}\mathbb E[\xi_t^\top G\xi_t]
+
C\,e^{-\beta t}\bigl(1+|\bar\mu_t|^2+\lambda\bigr),
\]
and the same inequality with $\tilde\xi_t$ in place of $\xi_t$.

Using \eqref{eq:m_discounted_bound} and Gr\"onwall's inequality, we deduce that there exist
constants $C\ge1$ and $\kappa>0$ such that
\[
e^{-\beta t}\mathbb E|\xi_t|^2
+
e^{-\beta t}\mathbb E|\tilde\xi_t|^2
+
e^{-\beta t}|\bar\mu_t|^2
\le
C\,e^{-2\kappa t}\Bigl(1+|s|^2+m_2(\mu)+\lambda\Bigr),
\qquad t\ge0,
\]
which is exactly \eqref{eq:discounted_second_moment_bound}. Finally, since
\[
\mathbb E|s_t|^2
\le
2\mathbb E|\xi_t|^2+2|\bar\mu_t|^2,
\qquad
m_2(\mu_t)
=
\mathbb E|\tilde s_t|^2
=
\mathbb E|\tilde\xi_t|^2+|\bar\mu_t|^2,
\]
\eqref{eq:discounted_second_moment_bound} implies that $\lim_{t\to\infty}e^{-\beta t}\bigl(\mathbb E|s_t|^2+m_2(\mu_t)\bigr) = 0$. Because the quadratic candidate satisfies
\[
|V(s,\mu)|\le C_V\bigl(1+|s|^2+m_2(\mu)\bigr),
\]
for some constant $C_V>0$, we conclude \eqref{eq:transitivity_conditon_V_pi_star}.
\end{proof}

\begin{proof}[\textbf{Proof of \Cref{thm:LQ_infinite}}]
Fix $(s,\mu)\in \mathbb R^d\times\mathcal P_2(\mathbb R^d)$ and let $\pi\in\Pi_{\mathrm{add}}$.
For $t\ge0$, define
\begin{equation}\label{eq:verification_process_LQR}
    \mathcal S_t^\pi
    :=
    e^{-\beta t}V(s_t,\mu_t)
    +
    \int_0^t e^{-\beta r}
    \Bigl(
        r(s_r,\mu_r,a_r^s)-\lambda\log p_\pi(s_r,\mu_r,a_r^s)
    \Bigr)\,\d r.
\end{equation}
For $n\in\mathbb N$, let $ \tau_n
    :=
    \inf\{
        t\ge0:\ |s_t|^2+m_2(\mu_t)\ge n
    \}.$

\textbf{Step 1: Derivatives of the candidate value function.}
Since $V$ in \eqref{eq:V_candidate_pf2} depends on $\mu$ only through $\bar\mu$, one has
\begin{equation}\label{eq:derivatives_candidate_LQR}
    \nabla_s V(s,\mu)=-2K(s-\bar\mu)-2Y,
    \qquad
    \nabla^2_{ss}V(s,\mu)=-2K,
\end{equation}
and
\begin{equation}\label{eq:measure_derivatives_candidate_LQR}
    \partial_\mu V(s,\mu)(\xi)=2K(s-\bar\mu)-2\Lambda\bar\mu,
    \qquad
    D_\xi\partial_\mu V(s,\mu)(\xi)=0.
\end{equation}

\textbf{Step 2: Localized verification identity.}
Applying It\^o's formula to
$e^{-\beta(t\wedge\tau_n)}V(s_{t\wedge\tau_n},\mu_{t\wedge\tau_n})$,
using \eqref{eq:derivatives_candidate_LQR}-\eqref{eq:measure_derivatives_candidate_LQR},
and taking expectations, we obtain
\begin{equation}\label{eq:localized_verification_identity_LQR}
\begin{aligned}
    \mathbb E\!\left[\mathcal S_{t\wedge\tau_n}^\pi\right]
    &=
    V(s,\mu)
    +
    \mathbb E\!\left[
        \int_0^{t\wedge\tau_n}
        e^{-\beta r}
        \Bigl(
            -\beta V(s_r,\mu_r)
            +
            (\mathcal L_{b,\Sigma}^{\pi}V)(s_r,\mu_r)
            +
            r_\lambda^\pi(s_r,\mu_r)
        \Bigr)\,\d r
    \right].
\end{aligned}
\end{equation}

\textbf{Step 3: Static entropy minimization.}
For $\eta\in\mathbb R^m$ and for every probability density $p$ on $\mathbb R^m$ such that
$\int_{\mathbb R^m}(|a|^2+|\log p(a)|)\,p(a)\,\d a<\infty$, define
\begin{equation}\label{eq:static_entropy_functional_LQR}
    \mathfrak J_\eta(p)
    :=
    \int_{\mathbb R^m}
    \Bigl(
        a^\top S a+2a^\top \eta+\lambda\log p(a)
    \Bigr)p(a)\,\d a.
\end{equation}
Note that the first two terms inside the integral can be rewritten as $a^\top S a+2a^\top \eta
=
\bigl(a+S^{-1}\eta\bigr)^\top S\bigl(a+S^{-1}\eta\bigr)
-
\eta^\top S^{-1}\eta$. Hence, defining 
\begin{equation*}
    g_\eta(a)
    :=
    \frac{\sqrt{\det S}}{(\pi_0\lambda)^{m/2}}
    \exp\!\left(
        -\frac1\lambda
        \bigl(a+S^{-1}\eta\bigr)^\top S\bigl(a+S^{-1}\eta\bigr)
    \right).
\end{equation*}
With this definition, \eqref{eq:static_entropy_functional_LQR} becomes
\begin{equation}\label{eq:KL_identity_LQR}
    \mathfrak J_\eta(p)
    =
    -
    \eta^\top S^{-1}\eta
    -
    \frac{\lambda}{2}\Bigl(m\log(\pi_0\lambda)-\log\det S\Bigr)
    +
    \lambda\,\mathrm{KL}(p\,\|\,g_\eta).
\end{equation}
In particular, $\mathfrak J_\eta(p)$ is minimized uniquely at $p=g_\eta$.

\textbf{Step 4: Expansion of the drift term.}
Set $z:=s-\bar\mu$ and $ \eta(s,\mu):=Us+(W-U)\bar\mu+O.$ Using \eqref{eq:general_coeff_LQR}, \eqref{eq:LQ_running_reward},
\eqref{eq:derivatives_candidate_LQR}-\eqref{eq:measure_derivatives_candidate_LQR},
and the definitions \eqref{eq:tilde_notation}-\eqref{eq:tilde_notation_2},
a direct computation yields
\begin{equation*}
\begin{aligned}
    -\beta V(s,\mu)
    +&
    (\mathcal L_{b,\Sigma}^{\pi}V)(s,\mu)
    +
    r_\lambda^\pi(s,\mu) =
    -z^\top\Bigl(Q-\beta K+KA+A^\top K+D^\top K D\Bigr)z \\
    &\quad
    -\bar\mu^\top\Bigl(Q+\bar Q-\beta\Lambda+\Lambda(A+\bar A)+(A+\bar A)^\top\Lambda
    +(D+\bar D)^\top K(D+\bar D)\Bigr)\bar\mu \\
    &\quad
    -2\Bigl(M+(A^\top-\beta I_d)Y+D^\top K\gamma\Bigr)^\top s
    -\Bigl(-\beta R+\gamma^\top K\gamma\Bigr) \\
    &\quad
    -
    \int_{\mathbb R^m}
    \Bigl(
        a^\top S a
        +
        2a^\top \eta(s,\mu)
        +
        \lambda\log p_\pi(s,\mu,a)
    \Bigr)\pi(\d a\mid s,\mu).
\end{aligned}
\end{equation*}
Applying \eqref{eq:KL_identity_LQR} with $\eta=\eta(s,\mu)$ and $p=p_\pi(s,\mu,\cdot)$, we obtain
\begin{equation*}
\begin{aligned}
    &-\beta V(s,\mu)
    +
    (\mathcal L_{b,\Sigma}^{\pi}V)(s,\mu)
    +
    r_\lambda^\pi(s,\mu)=
    -z^\top\Gamma_K z
    -\bar\mu^\top\Gamma_\Lambda\bar\mu
    -2\Gamma_Y^\top s
    -\Gamma_R
    -\lambda\,\mathrm{KL}\!\left(p_\pi(s,\mu,\cdot)\,\|\,g_{s,\mu}\right),
\end{aligned}
\end{equation*}
where $g_{s,\mu}:=g_{\eta(s,\mu)}$ and
\[
\Gamma_K
=
Q-\beta K+KA+A^\top K+D^\top K D-U^\top S^{-1}U,
\]
\[
\Gamma_\Lambda
=
Q+\bar Q-\beta\Lambda+\Lambda(A+\bar A)+(A+\bar A)^\top\Lambda
+(D+\bar D)^\top K(D+\bar D)-W^\top S^{-1}W,
\]
\[
\Gamma_Y
=
M+(A^\top-\beta I_d)Y+D^\top K\gamma-U^\top S^{-1}O,
\]
and
\[
\Gamma_R
=
-\beta R+\gamma^\top K\gamma-O^\top S^{-1}O
-\frac{\lambda}{2}\Bigl(m\log(\pi_0\lambda)-\log\det S\Bigr).
\]
By \eqref{eq:ARE_P}-\eqref{eq:ARE_r}, all these coefficients vanish. Hence
\begin{equation}\label{eq:drift_KL_only_LQR}
    -\beta V(s,\mu)
    +
    (\mathcal L_{b,\Sigma}^{\pi}V)(s,\mu)
    +
    r_\lambda^\pi(s,\mu)
    =
    -\lambda\,\mathrm{KL}\!\left(p_\pi(s,\mu,\cdot)\,\|\,g_{s,\mu}\right).
\end{equation}

\textbf{Step 5: Comparison for an arbitrary policy.}
Substituting \eqref{eq:drift_KL_only_LQR} into \eqref{eq:localized_verification_identity_LQR} yields
\begin{equation}\label{eq:localized_supermartingale_identity_LQR}
    \mathbb E\!\left[\mathcal S_{t\wedge\tau_n}^\pi\right]
    =
    V(s,\mu)
    -
    \lambda\,
    \mathbb E\!\left[
        \int_0^{t\wedge\tau_n}
        e^{-\beta r}
        \mathrm{KL}\!\left(
            p_\pi(s_r,\mu_r,\cdot)\,\|\,g_{s_r,\mu_r}
        \right)\,\d r
    \right]
    \le
    V(s,\mu).
\end{equation}
Since $V$ has quadratic growth and, by definition of $\Pi_{\mathrm{add}}$, the discounted second moments vanish at infinity, we may let $n\to\infty$ in \eqref{eq:localized_supermartingale_identity_LQR} and obtain, for every $t\ge0$,
\begin{equation}\label{eq:unstopped_supermartingale_identity_LQR}
    \mathbb E\!\left[\mathcal S_t^\pi\right]
    =
    V(s,\mu)
    -
    \lambda\,
    \mathbb E\!\left[
        \int_0^t
        e^{-\beta r}
        \mathrm{KL}\!\left(
            p_\pi(s_r,\mu_r,\cdot)\,\|\,g_{s_r,\mu_r}
        \right)\,\d r
    \right]
    \le
    V(s,\mu).
\end{equation}
Recalling \eqref{eq:verification_process_LQR}, this implies
\begin{equation}\label{eq:prelimit_verification_LQR}
\begin{aligned}
    &\mathbb E\!\left[e^{-\beta t}V(s_t,\mu_t)\right] 
    +
    \mathbb E^{s,\mu,\pi}\!\left[
        \int_0^t e^{-\beta r}
        \Bigl(
            r(s_r,\mu_r,a_r^s)-\lambda\log p_\pi(s_r,\mu_r,a_r^s)
        \Bigr)\,\d r
    \right]
    \le
    V(s,\mu).
\end{aligned}
\end{equation}
Letting $t\to\infty$ in \eqref{eq:prelimit_verification_LQR} and using the defining property of $\Pi_{\mathrm{add}}$, namely
\[
\lim_{t\to\infty}e^{-\beta t}\Bigl(\mathbb E|s_t|^2+m_2(\mu_t)\Bigr)=0,
\]
together with the quadratic growth of $V$, we conclude that
\begin{equation}\label{eq:comparison_result_LQR}
    V^\pi(s,\mu)\le V(s,\mu)
    \qquad\text{for every }\pi\in\Pi_{\mathrm{add}}.
\end{equation}
Hence $V^*(s,\mu)\le V(s,\mu)$.

\textbf{Step 6: Optimality of the candidate policy.}
By construction, the Gaussian kernel \eqref{eq:pi_star} has density $p_{\pi^*}(s,\mu,\cdot)=g_{s,\mu},$ so the Kullback-Leibler term in \eqref{eq:drift_KL_only_LQR} vanishes identically for $\pi=\pi^*$.
Moreover, Lemma \ref{lem:transversality_pi_star} shows that $\pi^*\in\Pi_{\mathrm{add}}$.
Therefore \eqref{eq:unstopped_supermartingale_identity_LQR} becomes
\[
\mathbb E\!\left[\mathcal S_t^{\pi^*}\right]=V(s,\mu),
\qquad t\ge0.
\]
Letting $t\to\infty$ and using Lemma \ref{lem:transversality_pi_star} again, we obtain $V^{\pi^*}(s,\mu)=V(s,\mu).$ Combining this with \eqref{eq:comparison_result_LQR}, we conclude that $V^*(s,\mu)=V^{\pi^*}(s,\mu)=V(s,\mu).$

\textbf{Step 7: Uniqueness of the optimal policy.}
Let $\hat\pi\in\Pi_{\mathrm{add}}$ be another optimal stationary feedback policy. Then
$V^{\hat\pi}(s,\mu)=V(s,\mu)$ for every $(s,\mu)$, and \eqref{eq:unstopped_supermartingale_identity_LQR} implies
\[
\mathbb E\!\left[
    \int_0^\infty
    e^{-\beta r}
    \mathrm{KL}\!\left(
        p_{\hat\pi}(s_r,\mu_r,\cdot)\,\|\,g_{s_r,\mu_r}
    \right)\,\d r
\right]
=0.
\]
Since the integrand is nonnegative, it follows that
\[
\mathrm{KL}\!\left(
    p_{\hat\pi}(s_r,\mu_r,\cdot)\,\|\,g_{s_r,\mu_r}
\right)=0
\qquad
\text{for almost every }r\ge0,\ \text{almost surely.}
\]
Hence $p_{\hat\pi}(s_r,\mu_r,\cdot)=g_{s_r,\mu_r}$ for almost every $r\ge0$, almost surely. Since the initial condition $(s,\mu)$ is arbitrary and the minimizer in \eqref{eq:KL_identity_LQR} is unique, we conclude that $\hat\pi=\pi^*$ as stationary feedback kernels. Therefore $\pi^*$ is the unique optimal stationary feedback policy in $\Pi_{\mathrm{add}}$.
\end{proof}

    \bibliographystyle{plain} 
    \bibliography{biblio.bib} 

\end{document}